\documentclass[12pt]{article}
\usepackage[T1]{fontenc}
\usepackage{tgtermes}
\usepackage{amsmath, amsthm, amssymb}
\usepackage{mathrsfs}
\usepackage{enumitem}
\usepackage{bm}
\usepackage{dsfont}
\usepackage{tikz}
\usepackage{xcolor}
\usepackage{xypic}
\usepackage{verbatim}

\usepackage[
backend=biber,
style=alphabetic,
sorting=ynt
]{biblatex}
\usepackage{hyperref}
\usepackage{nameref}
\usepackage{blindtext}
\usepackage{comment}
\usepackage{titling}
\usepackage{titlesec}

\usepackage{float}
\usepackage{caption}
\usepackage{lipsum}

\usepackage{array}
\usepackage{multirow}

\usepackage{ragged2e}
\usetikzlibrary{quotes,angles}
\usepackage{subcaption}
\usepackage[mathlines]{lineno}

\usepackage{graphicx}
\allowdisplaybreaks

\titleformat{\section}[block]{\color{black}\sc\filcenter}{\thesection}{1em}{}

\titleformat{\subsection}[runin]{\color{black}\normalsize\bfseries}{\sc\thesubsection}{1em}{}

\titleformat{\subsubsection}[runin]{\color{black}\normalsize\bfseries}{\sc\thesubsubsection}{1em}{}

\usepackage[
top    = 1.50cm,
bottom = 1.00cm]{geometry}

\hypersetup{
    colorlinks=true,
    linkcolor=black,
    filecolor=magenta,
    urlcolor=blue,
    citecolor = blue,
}

\newtheoremstyle{mystyle}
  {}
  {}
  {\itshape}
  {}
  {\bfseries}
  {.}
  { }
  {\thmname{#1}\thmnumber{ #2}\thmnote{ (#3)}}

\theoremstyle{mystyle}

\newtheorem{theorem}{Theorem}[section]

\newtheorem{corollary}[theorem]{Corollary}
\newtheorem{definition}[theorem]{Definition}

\newtheorem{lemma}[theorem]{Lemma}
\newtheorem{proposition}[theorem]{Proposition}
\newtheorem{remark}[theorem]{Remark}
\newtheorem*{assumption*}{Assumption}

\newcommand{\bkt}[1]{\biggl( #1 \biggr) }
\newcommand{\sbkt}[1]{\bigl( #1 \bigr) }
\newcommand{\bigbkt}[1]{\biggl[ #1 \biggr] }
\newcommand{\sbigbkt}[1]{\bigl[ #1 \bigr] }

\newcommand{\bignorm}[1]{\biggl|\biggl| #1 \biggr| \biggr|}
\newcommand{\abs}[1]{\biggl| #1 \biggr| }
\newcommand{\sabs}[1]{\bigl| #1 \bigr| }
\newcommand{\snorm}[1]{\bigl| \bigl|#1 \bigr| \bigr|}
\newcommand{\set}[1]{\biggl \{ #1 \biggr\} }

\numberwithin{equation}{section}

\title{\textbf{On the space-time fluctuations of the SHE and the KPZ equation in the entire $L^{2}$-regime for $d \geq 3$}}

\author{
  Te-Chun Wang 
  \thanks{Institute of Mathematics, École polytechnique fédérale de Lausanne, Lausanne, Switzerland.}
  \thanks{Email: \href{mailto:te-chun.wang@epfl.ch}{te-chun.wang@epfl.ch}}
}

\begin{document}
\maketitle
\date{} 

\begin{abstract}
We consider the mollified versions of the Kardar-Parisi-Zhang (KPZ) equation and the stochastic heat equation (SHE) in high dimensions $d\geq 3$ and analyze their probability distributions as the mollification is removed. Up to the $L^2$-criticality, we prove Gaussian limits, possibly with random perturbations, for the space-time fluctuations of the mollified versions around their stationarity. This result establishes a continuous analogue of the discrete case obtained in \cite{GFDDP}, and further extends it to a multi-point framework.


$\newline$
{\textsl{Keywords}: Kardar-Parisi-Zhang equation; stochastic heat equation; continuous directed polymer;}
$\newline$
\end{abstract}

\fontdimen2\font=0.29em

\tableofcontents

\section{Introduction} \label{first_section}
The main interest of this article starts with a stochastic partial differential equation introduced by Kardar, Parisi, and Zhang in \cite{KPZ}, known as the Kardar–Parisi–Zhang (KPZ) equation. At formal level, the KPZ equation is given by
\begin{equation} \label{KPZ}
    \frac{\partial \mathbf{H}^{\beta}}{\partial t} (x,t) = \frac{1}{2} \Delta \mathbf{H}^{\beta}(x,t) + \frac{1}{2}|\nabla \mathbf{H}^{\beta}(x,t)|^{2} - \infty + \beta 
    \xi(x,t), \quad  (x,t) \in \mathbb{R}^{d} \times \mathbb{R}_{+},
\end{equation}
where $\beta > 0$ denotes the strength of a space-time white noise $\xi$, which is a
distribution-valued Gaussian field with the correlation function
\begin{align*} 
    \mathbb{E}[\xi(x,t) \xi(y,s)] = \delta(x-y) \delta(t-s).
\end{align*}
From a mathematical viewpoint, the KPZ equation is ill-defined since the white noise is too irregular for the nonlinear term $|\nabla \mathbf{H}^{\beta}(x,t)|^{2}$ to be well-defined. The subtraction of an infinite constant, indicated by “$-\infty$” in the equation, formally accounts for the required renormalization of the divergence arising from the nonlinear term. The KPZ equation describes the dynamics of non-equilibrium growth processes in $d+1$ dimensions and has been studied extensively in both statistical physics and probability theory. For surveys on this topic in one dimension, we refer to \cite{quastel,ferrari2010random,Corwin1,quastel2015one}. 


To interpret the KPZ equation, the stochastic heat equation (SHE) provides a physically relevant solution via the Cole–Hopf transformation \cite{Burgers}; see also \cite{solving,Para} for alternative rigorous formulations. The SHE is formally given by
\begin{equation} \label{SHE}
    \frac{\partial \mathbf{X}^{\beta}}{\partial t} (x,t) = \frac{1}{2} \Delta \mathbf{X}^{\beta}(x,t) + \beta \mathbf{X}^{\beta}(x,t) 
    \xi(x,t), 
    \quad  (x,t) \in \mathbb{R}^{d} \times \mathbb{R}_{+}.
\end{equation}
In one dimension, the SHE can be rigorously defined by using an extension of Itô's theory \cite{walsh1986introduction}.
Notably, the SHE equation is closely related to a model of directed polymers. In high dimensions ($d \geq 3$), this model, known as the continuous directed polymer, can be viewed as a regularized, high-dimensional analogue of the continuum directed random polymer, which was rigorously constructed in one dimension in \cite{Alberts_2013}. See \cite{weakandstrong} for earlier work of the continuous directed polymer and \cite{continuousdirectedpolymer} for a related work. The corresponding discrete polymer model was originally proposed in the physics literature by Huse and Henley \cite{DDP2} as a simplified model for describing the interfaces of the Ising model with random coupling constants. A similar discrete model was later studied in the mathematical literature by  Imbrie and Spencer  \cite{DDP}, and has subsequently received much attention from the mathematical community; see \cite{comets2004probabilistic,ZYGOURAS2024104431} for a review.


This article investigates the SHE and the KPZ equation in \emph{high dimensions $d\geq 3$}. \emph{The assumption about the spacial dimension will be maintained throughout the remainder of the paper}. In this case, it remains unclear whether the foregoing approaches 
from the one-dimensional setting are valid. The challenge arises from the fact that the SHE becomes ill-defined in the sense of an extension of Itô's theory \cite{walsh1986introduction} since the white noise is too irregular for the product $\mathbf{X}^{\beta}(x,t) \cdot \xi(x,t)$.

\subsection{The mollified versions.} \label{model_section}
In high dimensions, a standard method for studying the SHE and the KPZ equation is to apply regularization and renormalization, which provides a rigorous framework to work with. This method analyzes the mollified versions of the SHE and the KPZ equation and investigates their scaling limits. Specifically, we consider the standard approximation scheme of the SHE, referred to as the \emph{mollified SHE}, which is given by
\begin{equation} \label{Mollified_SHE}
    \frac{\partial \mathbf{X}_{\varepsilon}^{\beta}}{\partial t} (x,t) = \frac{1}{2} \Delta \mathbf{X}^{\beta}_{\varepsilon}(x,t) + \beta_{\varepsilon} \mathbf{X}^{\beta}_{\varepsilon}(x,t) 
    \xi_{\varepsilon}
    (x,t).
\end{equation}
For simplicity, we only adopt the constant initial condition $\mathbf{X}^{\beta}_{\varepsilon}(x,0) \equiv 1$ in this paper. The mollified noise in equation (\ref{Mollified_SHE}) is defined as the spatial convolution of $\xi$ with an approximation of the identity:
\begin{equation} \label{definition:spatially_mollified_white_noise}
    \xi_{\varepsilon} (x,t)
    := \phi_{\varepsilon} * \xi (x,t)
    = \int_{\mathbb{R}^{d}} \phi_{\varepsilon}(x-y) \xi(y,t) dy.
\end{equation}
We work with a more general setup $\sbkt{\xi(x,t)}_{(x,t)\in \mathbb{R}^{d}\times \mathbb{R}}$ in which both the positive and negative time parts of the white noise play a role. This setting is crucial for constructing the stationary solutions of the mollified equation (\ref{Mollified_SHE}), as discussed in Section \ref{section_preliminaries}. The approximation of the identity $\phi_{\varepsilon}$ is defined via the scaling $\phi_{\varepsilon}(x) := \varepsilon^{-d} \phi(x/\varepsilon)$. Here and in the sequel, we work with a fixed probability density $\phi\in C_{c}^{\infty}(\mathbb{R}^{d},\mathbb{R}_{+})$ that is smooth, \emph{symmetric-decreasing}, and compactly supported. Here, we say that $f:\mathbb{R}^{d} \mapsto \mathbb{R}_{+}$ is symmetric-decreasing if $f$ is spherically symmetric and, for any unit vector $e \in \mathbb{R}^d$, the function $r \in \mathbb{R}_{+}\mapsto f(r e)$ is decreasing. Furthermore, it is well known that the correct way to tune down the strength of the noise in high dimensions is given as follows, which ensures the mollified SHE (\ref{Mollified_SHE}) having a nontrivial limit as $\varepsilon\to 0$:
\begin{align*}
    \beta_{\varepsilon} := \beta \cdot\varepsilon^{\frac{d-2}{2}},
\end{align*}
where $\beta>0$ is a fixed constant; see \cite{weakandstrong} for further details. In this framework, the existence of a solution to equation (\ref{Mollified_SHE}) follows from an extension of Itô’s theory \cite{walsh1986introduction}. 

Accordingly, the approximation scheme for the KPZ equation is defined by
\begin{equation}  \label{Cole–Hopf_2}
      \mathbf{H}^{\beta}_{\varepsilon}(x,t) := \log\mathbf{X}^{\beta}_{\varepsilon}(x,t).
\end{equation}
It is straightforward to verify that $\mathbf{H}_{\varepsilon}^{\beta}$ solves the following \emph{mollified KPZ equation} by using Itô's formula:
\begin{equation} \label{Mollified_KPZ}
    \frac{\partial \mathbf{H}_{\varepsilon}^{\beta}}{\partial t} (x,t) = \frac{1}{2} \Delta \mathbf{H}^{\beta}_{\varepsilon}(x,t) + \frac{1}{2}|\nabla \mathbf{H}^{\beta}_{\varepsilon}(x,t)|^{2} - \mathbf{C}_{\varepsilon} + \beta_{\varepsilon} 
    \xi_{\varepsilon}(x,t), \quad\mathbf{C}_{\varepsilon} := \frac{1}{2} \beta_{\varepsilon}^{2} ||\phi_{\varepsilon}||^{2}_{2}.
\end{equation}

\subsection{Overview of the main results.} \label{section:overview_main_result_space_time}

The purpose of this article is to analyze the space-time fluctuations of the mollified SHE (\ref{Mollified_SHE}) and the mollified KPZ equation (\ref{Mollified_KPZ}) around their stationary solutions. This setting was previously considered, for example, by Comets, Cosco, and
Mukherjee in \cite{comets2019space,normal}.
The corresponding fluctuations, denoted by $\mathbf{F}^{\textup{SHE};\beta}_{\varepsilon}(x,t)$ and $\mathbf{F}^{\textup{KPZ};\beta}_{\varepsilon}(x,t)$, are precisely defined in (\ref{fluctuation_SHE_KPZ}).  Among these fluctuations, the limiting behavior of $\mathbf{F}^{\textup{SHE};\beta}_{\varepsilon}(x,t)$ is of particular interest. We establish a randomly perturbed Gaussian limit for $\mathbf{F}^{\textup{SHE};\beta}_{\varepsilon}(x,t)$ and the Gaussianity of 
$\mathbf{F}^{\textup{KPZ};\beta}_{\varepsilon}(x,t)$ at multiple space-time points throughout a critical regime, known as the $L^{2}$-regime; see (\ref{beta_L_2_1}) for the corresponding critical point. Here, the multi-point setup refers to the vectors $\sbkt{\mathbf{F}^{\textup{SHE};\beta}_{\varepsilon}(x_{j},t_{j})}_{1\leq j\leq M}$ and $\sbkt{\mathbf{F}^{\textup{KPZ};\beta}_{\varepsilon}(x_{j},t_{j})}_{1\leq j\leq M}$, where $(x_{j},t_{j})_{1\leq j\leq M}$ is a finite collection of points in $\mathbb{R}^{d}\times \mathbb{R}_{+}$. 

\begin{theorem} [Informal version] \label{formal_main_theorem} 
Assume that $d\geq 3$. The following holds for all $\beta$ in the $L^{2}$-regime as $\varepsilon \to 0$:
\begin{enumerate} [label = (\roman*)]
    \item The rescaled fluctuation ${\mathbf{F}^{\textup{SHE};\beta}_{\varepsilon}(x,t)}$ converges in distribution at multiple space-time points to the independent product of two random fields: a Gaussian random field and a random field whose components form a continuous collection of i.i.d. random variables.
    \item The rescaled fluctuation $\mathbf{F}^{\textup{KPZ};\beta}_{\varepsilon}(x,t)$ converges in distribution toward the above Gaussian random field at multiple space-time points.
\end{enumerate}
\end{theorem}

In a multi-point framework, the limiting distribution of $\mathbf{F}^{\textup{SHE};\beta}_{\varepsilon}(x,t)$ can be viewed as a Gaussian random field perturbed by a continuous collection of i.i.d. random variables. To see this limiting behavior, the key observation relies on an asymptotic decomposition given by
\begin{align} \label{formal_key_decomposition}
    \mathbf{F}^{\textup{SHE};\beta}_{\varepsilon}(x,t) \overset{\varepsilon \to 0}{\approx}   \mathbf{N}^{\beta}_{\varepsilon}(x,t) \times \mathbf{X}^{\beta}_{\varepsilon}(x,t) \quad \forall (x,t) \in \mathbb{R}^{d} \times \mathbb{R}_{+},
\end{align}
where $\mathbf{N}^{\beta}_{\varepsilon}(x,t)$ is a random field depending only on the negative-time part of the white noise $\xi$; see (\ref{KPZ_approxiii}) for the precise formulation. The mollified version $\mathbf{X}^{\beta}_{\varepsilon}(x,t)$, constructed from the mollified equation (\ref{Mollified_SHE}), only depends on the positive-time part of the white noise $\xi$. As a result, the decomposition (\ref{formal_key_decomposition}) gives the independence stated in part (i) of Theorem \ref{formal_main_theorem}. In particular, the limiting distributions of $\mathbf{N}^{\beta}_{\varepsilon}(x,t)$ and $\mathbf{X}^{\beta}_{\varepsilon}(x,t)$ correspond to the two random fields given in part (i) of Theorem \ref{formal_main_theorem}. Nevertheless, the i.i.d. random perturbations produced by $\mathbf{X}^{\beta}_{\varepsilon}(x,t)$ arise only in this \emph{local (multi-point) setting}.

More precisely, in the spatial averaging framework, the integration smooths out the contribution from $\mathbf{X}^{\beta}_{\varepsilon}(x,t)$ as $\varepsilon \to 0$:
\begin{align} \label{formal_sp_convergence}
    \int_{\mathbb{R}^{d}} dx f(x) \mathbf{F}^{\textup{SHE};\beta}_{\varepsilon}(x,t) \overset{\varepsilon \to 0}{\approx} \int_{\mathbb{R}^{d}} dx f(x)  \mathbf{N}^{\beta}_{\varepsilon}(x,t) \quad \forall f \in C_{c}^{\infty}(\mathbb{R}^{d}).
\end{align}
Intuitively, within a small region of $\mathbb{R}^{d}$ where the function $f$ can be regarded as a constant, the spikes of $\mathbf{X}^{\beta}_{\varepsilon}(x,t)$ above the expected value $\mathbb{E}[\mathbf{X}^{\beta}_{\varepsilon}(x,t)] = 1$ cancel out those below the expected value as $\varepsilon \to 0$. Hence, in (\ref{formal_sp_convergence}), the random field $\mathbf{X}^{\beta}_{\varepsilon}(x,t)$ behaves like a constant $1$. See Figure \ref{figure:observation_2} for an overview of the behavior of $\mathbf{F}^{\textup{SHE};\beta}_{\varepsilon}(x,t)$ across different setups.


By contrast, regardless of the above two frameworks, the limiting behavior of $\mathbf{F}^{\textup{KPZ};\beta}_{\varepsilon}(x,t)$ is simply that of a Gaussian random field. The difference between the SHE case and the KPZ case suggests a special \emph{cancellation}, which eliminates the random perturbations from $\mathbf{F}^{\textup{SHE};\beta}_{\varepsilon}(x,t)$. This cancellation is detailed in Section \ref{section_heuristics}.
More precisely, $\mathbf{F}^{\textup{KPZ};\beta}_{\varepsilon}(x,t)$ admits the following linear approximation induced by the logarithmic transformation (\ref{Cole–Hopf_2}) and the asymptotic decomposition (\ref{formal_key_decomposition}):
\begin{align} \label{KPZ_approximation}
    \mathbf{F}^{\textup{KPZ};\beta}_{\varepsilon}(x,t) \overset{\varepsilon \to 0}{\approx} \frac{1}{\mathbf{X}^{\beta}_{\varepsilon}(x,t)} \mathbf{F}^{\textup{SHE};\beta}_{\varepsilon}(x,t) \overset{\varepsilon \to 0}{\approx}  \mathbf{N}^{\beta}_{\varepsilon}(x,t) \quad \forall (x,t) \in \mathbb{R}^{d} \times \mathbb{R}_{+}.
\end{align}

Notably, the linear approximation (\ref{KPZ_approximation}) plays a crucial role in circumventing the main challenge encountered when proving the Gaussianity of $\mathbf{F}^{\textup{KPZ};\beta}_{\varepsilon}(x,t)$. More specifically, applying the decomposition (\ref{formal_key_decomposition}) eliminates the problematic denominator in the middle of (\ref{KPZ_approximation}), which might lead to the use of higher moment bounds for the limiting partition function when analyzing the limit of $\mathbf{F}^{\textup{KPZ};\beta}_{\varepsilon}(x,t)$. Such moment bounds may not be valid throughout the $L^{2}$-regime and thereby restrict the derivation of the Gaussianity of $\mathbf{F}^{\textup{KPZ};\beta}_{\varepsilon}(x,t)$. A similar challenge also arose in the work of Comets, Cosco, and Mukherjee \cite{comets2019space}, who were the first to establish part (ii) of Theorem~\ref{formal_main_theorem} under the assumption that the coupling constant $\beta$ lies within a proper subset of the $L^{2}$-regime; see \cite[Section 4.2]{comets2019space} for details. We also note that the above difference between the SHE case and the KPZ case also contrasts the known fact that the spatially averaged fluctuations of the mollified versions (\ref{Mollified_SHE}) and (\ref{Mollified_KPZ}) around their expected values exhibit similar limiting behavior. Specifically, they both converge to the Edwards–Wilkinson equations; see \cite{COSCO2022127,gu2018edwards,guKPZ,scaling} for details.

\begin{figure} [!htb]
\centering
\begin{tikzpicture} [scale=0.7]
\node (a1) at (1,4) {$\mathbf{F}^{\textup{SHE};\beta}_{\varepsilon}(0,t)$};
\node (b1) at (1,1) {$\mathbf{U}^{\beta}(0,t) \cdot \mathcal{Z}_{\infty}^{\beta}(0;\xi)$};
\node (c1) at (1,2.5) [anchor=east]{\small $d$};

\node (d1) at (4,4.5) [anchor=south]{\small multi-point extension};

\node (a2) at (9,4) {$\sbkt{\mathbf{F}^{\textup{SHE};\beta}_{\varepsilon}(x,t)}_{(x,t)\in \mathbb{R}^{d}\times \mathbb{R}_{+}}$};
\node (b2) at (9,1) {$\sbkt{\mathbf{U}^{\beta}(x,t) \cdot \mathbf{Z}^{\beta}(x,t)}_{(x,t) \in \mathbb{R}^{d}\times \mathbb{R}_{+}}$};
\node (c2) at (9,2.5) [anchor=east]{\small $\text{f.d.m.}$};

\node (d1) at (13.5,4.5) [anchor=south]{\small spatial averaging};

\node (a3) at (18,4) {$\displaystyle\int_{\mathbb{R}^{d}} dx f(x) \mathbf{F}^{\textup{SHE};\beta}_{\varepsilon}(x,t)$};
\node (b3) at (18,1) {$\displaystyle \int_{\mathbb{R}^{d}} dx f(x) \mathbf{U}^{\beta}(x,t)$};
\node (c3) at (18,2.5) [anchor=east]{\small $d$};

\draw[->] [thick] (a1) to (b1);
\draw[->] [thick] (a2) to (b2);
\draw[->] [thick] (a3) to (b3);

\draw[->] [red,very thick] (a1) to (a2);
\draw[->] [red,very thick] (b1) to (b2);

\draw[->] [blue,very thick] (a2) to (a3);
\draw[->] [blue,very thick] (b2) to (b3);

\end{tikzpicture}
\caption{ \justifying 
This figure provides a framework for how the limiting behavior of the random field $\sbkt{\mathbf{F}^{\textup{SHE};\beta}_{\varepsilon}(x,t)}_{(x,t)\in \mathbb{R}^{d}\times \mathbb{R}_{+}}$ changes across different settings, where $\mathbf{U}^{\beta}(x,t)$ and $\mathbf{Z}^{\beta}(x,t)$ are the two random fields mentioned in part (i) of Theorem \ref{formal_main_theorem}. Here, $\mathcal{Z}_{\infty}^{\beta}(0;\xi)$ is the limiting partition function, which will be defined in (\ref{limit_CDP}).}
\label{figure:observation_2}
\end{figure}


From our perspective, the convergence results in Theorem~\ref{formal_main_theorem} are especially meaningful in the study of the continuous directed polymer \cite{weakandstrong}, whose partition function $\mathcal{Z}_{T}^{\beta}(x;\xi)$ is defined in (\ref{partition_function_1}). As an application, Theorem~\ref{formal_main_theorem} establishes the \emph{rates of convergence} of the partition function and its associated free energy toward their limits as $T\to\infty$. In high dimensions, a \emph{key feature} of the continuous directed polymer is the non-triviality of the limiting partition function as $T \to \infty$, as well as that of the corresponding limiting free energy; see (\ref{limit_CDP}) for details. This key feature naturally motivates the study of large-scale fluctuations, given in (\ref{definition:PF_fluctuation}) and (\ref{definition:FE_fluctuation}), respectively. Within this framework, we establish the limiting distributions of these large-scale fluctuations in the multi-point framework, given by the vectors
\begin{align} \label{fluctuation:multi-point framework}
\sbkt{\mathbf{F}^{\textup{PF};\beta}_{T}(x_{j})}_{1\leq j\leq M} \quad \text{and} \quad \sbkt{\mathbf{F}^{\textup{FE};\beta}_{T}(x_{j})}_{1\leq j\leq M}.
\end{align}
See Corollary~\ref{corollary_CDP} for details.


The randomly perturbed Gaussian limit for a discrete analogue of $\mathbf{F}^{\textup{PF};\beta}_{T}(0)$ was first proved by Comets and Liu  \cite{rate_Comets}, and later extended to the full discrete $L^2$-regime by Cosco and Nakajima \cite{GFDDP}. In particular, Cosco and Nakajima noted in \cite[p.6]{GFDDP} that their approach, based on a specific version of the martingale central limit theorem \cite[Lemma 3.1]{rate_Comets}, might also be extended to the continuous setting $\mathbf{F}^{\textup{PF};\beta}_{T}(0)$ and $\mathbf{F}^{\textup{FE};\beta}_{T}(0)$. We find this extension plausible in the single-point case. However, in the multi-point framework (\ref{fluctuation:multi-point framework}), it is unclear to us how to apply \cite[Lemma 3.1]{rate_Comets}. More specifically, applying \cite[Lemma 3.1]{rate_Comets} to handle the fluctuations in (\ref{fluctuation:multi-point framework}) would naturally lead one to consider a particular sequence of martingales. \emph{For this sequence of
martingales}, we show that a key condition required by \cite[Lemma 3.1]{rate_Comets} fails to hold; see Section \ref{section:Fluctuations_for_the_continuum_directed_random_polymer} for details. Our approach takes a different path to study the fluctuations in \eqref{fluctuation:multi-point framework}; see Section \ref{section_heuristics} for details.

Building on the convergence of \eqref{fluctuation:multi-point framework} as $T\to\infty$, we further investigate how disorder affects this convergence. Roughly speaking, one of our main results shows that the limiting distribution of $\sbkt{\mathbf{F}^{\textup{FE};\beta}_{T}(x_{j})}_{1\leq j\leq M}$ is independent of the underlying disorder, whereas the limiting law of $\sbkt{\mathbf{F}^{\textup{PF};\beta}_{T}(x_{j})}_{1\leq j\leq M}$ may be affected by it, depending on the choice of starting positions. See Corollary \ref{remark_stability} for the precise formulation. Our main result can be viewed as a continuous and multi-point analogue of the previous works \cite{GFDDP,rate_Comets} for the discrete random polymer. In particular, we go further by explicitly characterizing the limiting distribution affected by disorder; see Remark~\ref{remark:stable_coupling} for further discussions.


\paragraph{Organization.} The remainder of this paper is organized as follows. In Section~\ref{section:fluctuation_main}, we present the main results. The construction of stationary solutions is discussed in Section~\ref{section_preliminaries}. Section~\ref{related} reviews related works, and Section~\ref{section_heuristics} outlines the key ideas and heuristics behind the proofs. The core arguments begin in Section~\ref{section_proofoutline}, where we provide a proof outline for our main result, Theorem~\ref{ST_SHE_and_KPZ}. The proofs of the main techniques are presented in Sections~\ref{SHE_key_proposition_section}, \ref{ratio_bound_secion}, \ref{section:proof_of_trivial_lemma}, and \ref{section_exponential_functional}. Finally, in Section~\ref{section_proof_CDP}, we apply our results to the continuous directed polymer and prove Corollary \ref{corollary_CDP} and Corollary \ref{remark_stability}.

\section{Main results} \label{main_section}
In the sequel, we denote the transition density of $d$-dimensional standard Brownian motion and the positive constant $\gamma(\beta)$ as follows, respectively:
\begin{align} 
    G_{t}(x) := \frac{1}{(2\pi t)^{d/2}} \exp\bkt{-\frac{|x|^{2}}{2t}}, \quad
    \gamma(\beta)^{2} :=
    \beta^{2} \int_{\mathbb{R}^{d}} dx R(x) \mathbf{E}_{\frac{x}{\sqrt{2}}}\bigbkt{\exp\bkt{\beta^{2}\int_{0}^{\infty} ds R(\sqrt{2}B(s))}}. \label{gamma_and_R}
\end{align}
It can be shown that $\gamma(\beta) < \infty$ whenever $\beta \in (0,\beta_{L^{2}})$, where the critical point is defined by
\begin{equation}\label{beta_L_2_1}
    \beta_{L^{2}} := \sup\set{\beta > 0: \sup_{\varepsilon > 0,\; (x,t) \in \mathbb{R}^{d} \times \mathbb{R}_{+}}\mathbb{E}[\mathbf{X}^{\beta}_{\varepsilon}(x,t)^{2}] < \infty}.
\end{equation}
Furthermore, we use $\overset{\textup{f.d.m.}}{\longrightarrow}$ to denote the finite-dimensional convergence. 

\subsection{Multi-point fluctuations.} \label{section:fluctuation_main}
Our main interest lies in the study of the rescaled fluctuations of the mollified SHE (\ref{Mollified_SHE}) and the mollified KPZ equation (\ref{Mollified_KPZ}).  These fluctuations are defined respectively by
\begin{equation} \label{fluctuation_SHE_KPZ}
    \mathbf{F}^{\textup{SHE};\beta}_{\varepsilon}(x,t) := \varepsilon^{-\frac{d-2}{2}} \sbkt{\mathbf{X}_{\varepsilon}^{\textup{st};\beta}(x,t)-\mathbf{X}^{\beta}_{\varepsilon}(x,t)}, \quad 
    \mathbf{F}^{\textup{KPZ};\beta}_{\varepsilon}(x,t) := \varepsilon^{-\frac{d-2}{2}} \sbkt{\mathbf{H}_{\varepsilon}^{\textup{st};\beta}(x,t)-\mathbf{H}^{\beta}_{\varepsilon}(x,t)}.
\end{equation}
Here, $\mathbf{X}_{\varepsilon}^{\textup{st};\beta}(x,t)$ and $\mathbf{H}_{\varepsilon}^{\textup{st};\beta}(x,t)$ are stationary solutions of the mollified equations (\ref{Mollified_SHE}) and (\ref{Mollified_KPZ}), respectively; see Section \ref{section_preliminaries} for details. 

The main result of this article is stated as follows. The underlying ideas of the proof are outlined in Section~\ref{section_heuristics}, and the detailed proof begins in Section~\ref{section_proofoutline}.
\begin{theorem} \label{ST_SHE_and_KPZ}
Assume that $d\geq 3$ and $\beta \in (0,\beta_{L^{2}})$. Then, as $\varepsilon \to 0$, the following holds:
\begin{alignat} {2}
    &\sbkt{\mathbf{F}^{\textup{SHE};\beta}_{\varepsilon}(x,t)}_{(x,t) \in \mathbb{R}^{d} \times \mathbb{R}_{+}} &&\overset{\textup{f.d.m.}}{\longrightarrow} \sbkt{\mathbf{Z}^{\beta}(x,t)\cdot\mathbf{U}^{\beta}(x,t)}_{(x,t) \in \mathbb{R}^{d} \times \mathbb{R}_{+}}, \label{SHE_pointwise_fluctuation} \\
    &\sbkt{\mathbf{F}^{\textup{KPZ};\beta}_{\varepsilon}(x,t)}_{(x,t) \in \mathbb{R}^{d} \times \mathbb{R}_{+}}
    &&\overset{\textup{f.d.m.}}{\longrightarrow}
    \sbkt{\mathbf{U}^{\beta}(x,t)}_{(x,t) \in \mathbb{R}^{d} \times \mathbb{R}_{+}}, \label{KPZ_pointwise_fluctuation} 
\end{alignat}
where $\sbkt{\mathbf{Z}^{\beta}(x,t)}_{(x,t)\in \mathbb{R}^{d} \times \mathbb{R}_{+}}$ and $\sbkt{\mathbf{U}^{\beta}(x,t)}_{(x,t)\in \mathbb{R}^{d} \times \mathbb{R}_{+}}$ are independent random fields. Their probability distributions are described as follows.
\begin{itemize} 
    \item The random field $\sbkt{\mathbf{Z}^{\beta}(x,t)}_{(x,t)\in \mathbb{R}^{d} \times \mathbb{R}_{+}}$ is a continuous collection of i.i.d. random variables such that the probability distribution of a single point is the same as the law of the limiting partition function:  
    \begin{align*}
        \mathbf{Z}^{\beta}(x,t) \overset{d}{=}\mathcal{Z}^{\beta}_{\infty}(0;\xi) \quad \forall (x,t) \in \mathbb{R}^{d} \times \mathbb{R}_{+},
    \end{align*}
    where the random variable $\mathcal{Z}^{\beta}_{\infty}(0;\xi)$ will be explained in (\ref{limit_CDP}).
    \item The random field $\sbkt{\mathbf{U}^{\beta}(x,t)}_{(x,t)\in \mathbb{R}^{d} \times \mathbb{R}_{+}}$ solves the heat equation with a random initial data:
    \begin{equation} \label{GFF}
        \frac{\partial \mathbf{U}^{\beta}}{\partial t} (x,t) = \frac{1}{2}\Delta \mathbf{U}^{\beta}(x,t), \quad \mathbf{U}^{\beta}(x,0) = \mathbf{U}^{\beta}_{0}(x),
    \end{equation}
    where $\sbkt{\mathbf{U}^{\beta}_{0}(x)}_{x\in \mathbb{R}^{d}}$ is a centered Gaussian field with  the singular correlation function
    \begin{equation*}
        \mathbb{E}\sbigbkt{\mathbf{U}^{\beta}_{0}(x_{1})\mathbf{U}^{\beta}_{0}(x_{2})}
        = \gamma(\beta)^{2} \int_{0}^{\infty} G_{2s}(x_{2}-x_{1}) ds. 
    \end{equation*}
\end{itemize}
\end{theorem}

Let us now turn to the continuous directed polymer. We recall that the
partition function at inverse temperature $\beta$ is given by
\begin{equation} \label{partition_function_1}
    \mathcal{Z}^{\beta}_{T}(x;\xi)
    := \mathbf{E}_{x}\bigbkt{\exp\bkt{\beta \int_{0}^{T} ds \xi_{1}(B(s),s) - \frac{\beta^{2} R(0) T }{2}}}, 
\end{equation}
where the mollified white noise $\xi_{1}(\cdot,\cdot)$ is defined in (\ref{definition:spatially_mollified_white_noise}). Here, we set $R(x) := R_{1}(x)$, where $R_{1}$ is given by 
\begin{equation} \label{definition_R_epsilon}
    R_{\varepsilon}(x) := \phi_{\varepsilon} * \phi_{\varepsilon}(x),
\end{equation}
and $\phi_{\varepsilon}$ is the approximation of the identity defined in Section \ref{model_section}. In addition, we denote by $\mathbf{P}_{x}$ the law of $d$-dimensional Brownian motion $B$ starting at $x \in \mathbb{R}^{d}$. The almost sure limit of the partition function $\mathcal{Z}^{\beta}_{T}(x;\xi)$, denoted by $\mathcal{Z}^{\beta}_{\infty}(x;\xi)$ and given in (\ref{definition_Z_x}), naturally gives rise to the study of the large-scale fluctuations
\begin{alignat}{2}
    &\mathbf{F}^{\textup{PF};\beta}_{T}(x) &&:= T^{\frac{d-2}{4}} \sbkt{\mathcal{Z}_{\infty}^{\beta}(\sqrt{T}x;\xi)-\mathcal{Z}^{\beta}_{T}(\sqrt{T}x;\xi)},\label{definition:PF_fluctuation}\\
    &\mathbf{F}^{\textup{FE};\beta}_{T}(x) &&:= T^{\frac{d-2}{4}} \sbkt{\log\mathcal{Z}_{\infty}^{\beta}(\sqrt{T}x;\xi)-\log\mathcal{Z}^{\beta}_{T}(\sqrt{T}x;\xi)}. \label{definition:FE_fluctuation}
\end{alignat}
See \eqref{SHE_D1} and \eqref{KPZ_D0} for their connections to the rescaled fluctuations $\mathbf{F}^{\textup{SHE};\beta}_{\varepsilon}(x,t)$ and $\mathbf{F}^{\textup{KPZ};\beta}_{\varepsilon}(x,t)$. As a byproduct, Theorem \ref{ST_SHE_and_KPZ} establishes the following asymptotics, which will be proved in Section \ref{section_proof_CDP}.
\begin{corollary} \label{corollary_CDP}
Assume that $d\geq 3$ and $\beta \in (0,\beta_{L^{2}})$. Then, as $T \to \infty$, the following holds:
\begin{align} \label{pointwise_fluctuation} 
   \sbkt{\mathbf{F}^{\textup{PF};\beta}_{T}(x)}_{x \in \mathbb{R}^{d}} \overset{\textup{f.d.m.}}{\longrightarrow} \sbkt{\mathbf{Z}^{\beta}(x,1)\cdot\mathbf{U}^{\beta}(x,1)}_{x \in \mathbb{R}^{d}}, \quad\sbkt{\mathbf{F}^{\textup{FE};\beta}_{T}(x)}_{x \in \mathbb{R}^{d}}
    \overset{\textup{f.d.m.}}{\longrightarrow}
    \sbkt{\mathbf{U}(x,1)}_{x \in \mathbb{R}^{d}},
\end{align}
where $\mathbf{Z}^{\beta}(x,1)$ and $\mathbf{U}^{\beta}(x,1)$ are defined in Theorem \ref{ST_SHE_and_KPZ}.
\end{corollary}

Let us now present how the convergences in \eqref{pointwise_fluctuation} are influenced by disorder. To set the stage, we recall the relevant definitions from \cite[Section 1.3]{rate_Comets}, as outlined below. 


\begin{definition} \label{definition:stable}
Let $\Vec{F}_{n}:=(\Vec{F}_{n}(1),\ldots,\Vec{F}_{n}(M))_{n\geq 1}$ be a sequence of random vectors defined on a common probability space $(\widetilde{\Omega},\widetilde{\mathcal{F}},\widetilde{\mathbb{P}})$ such that that $\Vec{F}_{n} \overset{d}{\longrightarrow} \Vec{F}_{\infty}$.
\begin{enumerate} [label=(\roman*)]
    \item The convergence $\Vec{F}_{n} \overset{d}{\longrightarrow} \Vec{F}_{\infty}$ is said to be \textbf{stable} if, in addition to the assumed convergence, for each $A \in \widetilde{\mathcal{F}}$ with positive measure, the conditional law of $\Vec{F}_{n}$ given $A$ is convergent.
    \item  The convergence $\Vec{F}_{n} \overset{d}{\longrightarrow} \Vec{F}_{\infty}$ is said to be \textbf{mixing} if it is stable and 
    the limit is independent of $A$. 
\end{enumerate}
\end{definition}

The influence of disorder can be described as follows, which will be shown in Section \ref{section_proof_CDP}.
\begin{corollary} \label{remark_stability}
Assume that $d\geq 3$ and $\beta < \beta_{L^{2}}$. Let $\{x_{j}:1\leq j\leq M\} \subseteq \mathbb{R}^{d}$ and let $W \in \sigma(\xi)$. We recall that the random fields $\mathbf{Z}^{\beta}$ and $\mathbf{U}^{\beta}$ are defined in Theorem \ref{ST_SHE_and_KPZ}, and we further suppose that both of them are independent of $\sigma(\xi)$. Then the following holds.
\begin{enumerate} [label = (\roman*)]
    \item The convergences of $(\mathbf{F}^{\textup{PF};\beta}_{T}(x_{j}))_{1\leq j\leq M, \; x_{j} \neq 0}$ and $(\mathbf{F}^{\textup{FE};\beta}_{T}(x_{j}))_{1\leq j\leq M}$ are mixing. Equivalently, as $T\to\infty$, we have
    \begin{alignat} {2}
        &\bkt{(\mathbf{F}^{\textup{PF};\beta}_{T}(x_{j}))_{1\leq j\leq M, \; x_{j} \neq 0}, W} &&\overset{d}{\longrightarrow} \bkt{(\mathbf{Z}^{\beta}(x_{j},1)\cdot\mathbf{U}^{\beta}(x_{j},1))_{1\leq j\leq M, \; x_{j} \neq 0}, W},\label{proposition:pf1}\\
        &\bkt{(\mathbf{F}^{\textup{FE};\beta}_{T}(x_{j}))_{1\leq j\leq M}, W} &&\overset{d}{\longrightarrow} \bkt{(\mathbf{U}^{\beta}(x_{j},1))_{1\leq j\leq M}, W}. \label{proposition:fe1}
    \end{alignat}
    \item If we further assume that $x_{1} = 0$ and $x_{j} \neq 0$ for all $2\leq j\leq M$, then the convergence of $(\mathbf{F}^{\textup{PF};\beta}_{T}(x_{j}))_{1\leq j\leq M}$ is stable. In particular, as $T\to\infty$, it holds that
    \begin{align} \label{proposition:pf2}
        \bkt{\mathbf{F}^{\textup{PF};\beta}_{T}(0),(\mathbf{F}^{\textup{PF};\beta}_{T}(x_{j}))_{2\leq j\leq M}, W} \overset{d}{\longrightarrow} \bkt{\mathcal{Z}_{\infty}^{\beta}&(0;\xi)\cdot \mathbf{U}^{\beta}(0,1),\sbkt{\mathbf{Z}^{\beta}(x_{j},1)\cdot\mathbf{U}^{\beta}(x_{j},1)}_{2\leq j\leq M},W}.
    \end{align}
\end{enumerate}
\end{corollary} 

\begin{remark} \label{remark:explanation_pf_fe}
Part (i) of Corollary \ref{remark_stability} may seem intuitive, as the starting point of the polymer moves off to infinity at a rate proportional to $\sqrt{T}$. However, it is important to note that in part (ii) of Corollary \ref{remark_stability}, as $T\to\infty$, $\mathbf{F}^{\textup{PF};\beta}_{T}(0)$ decomposes into two distinct components: one that is influenced by disorder, and another that is independent of disorder. The decomposition of $\mathbf{F}^{\textup{PF};\beta}_{T}(0)$ will be explicit in the heuristic discussions provided in Section \ref{section_heuristics}.
\end{remark}

\begin{remark} \label{remark:stable_coupling}
It is worth noting that the stability of $\Vec{F}_{n} \overset{d}{\to} \Vec{F}_{\infty}$ only guarantees that the pair $(\Vec{F}_{n},A)$ converges in distribution to some coupling of $\Vec{F}_{\infty}$ and $A$, where $A \in \widetilde{\mathcal{F}}$. However, part (ii) of the corollary explicitly identifies the form of the limiting distribution.
\end{remark}


\subsection{Stationary solutions.} \label{section_preliminaries}
In this subsection, we review the approach developed in \cite{normal} for constructing stationary solutions to equations (\ref{Mollified_SHE}) and (\ref{Mollified_KPZ}). To this end, we begin by recalling key properties of the partition function. It was shown in \cite{Feynman-Kac,weakandstrong} that the solution $\mathbf{X}^{\beta}_{\varepsilon}(x,t)$ can be expressed in terms of the partition function of the continuous directed polymer. More precisely, we have
\begin{equation} \label{partition_function_0}
    \mathbf{X}^{\beta}_{\varepsilon}(x,t) 
    = \mathcal{Z}^{\beta}_{\frac{t}{\varepsilon^{2}}}(0;\xi^{(\varepsilon,x,t)}) \quad \text{and} \quad \mathbf{H}^{\beta}_{\varepsilon}(x,t) 
    = \log\mathcal{Z}^{\beta}_{\frac{t}{\varepsilon^{2}}}(0;\xi^{(\varepsilon,x,t)}),
\end{equation}
where $\mathcal{Z}^{\beta}_{T}(x;\xi)$ is defined in (\ref{partition_function_1}), and $\xi^{(\varepsilon,x,t)}$ denotes the diffusively rescaled, time-reversed space-time white noise given by
\begin{equation} \label{rescaled_space_time_noise}
    \xi^{(\varepsilon,x,t)}(y,s) := \varepsilon^{\frac{d+2}{2}}\xi(\varepsilon y + x,t - \varepsilon^{2}s).
\end{equation}
A key result from \cite{weakandstrong} establishes a phase transition for the partition function: there exists a critical inverse temperature $\beta_{c}>0$ such that the family $(\mathcal{Z}^{\beta}_{T}(0;\xi))_{T\geq 0}$ is uniformly integrable if and only if $\beta < \beta_{c}$. In particular, it holds that
\begin{equation} \label{limit_CDP}
    \lim_{T\to\infty}  \mathcal{Z}^{\beta}_{T}(0;\xi) = \mathcal{Z}^{\beta}_{\infty}(0;\xi) \quad\text{if } \beta < \beta_{c}; \quad \lim_{T\to\infty}  \mathcal{Z}^{\beta}_{T}(0;\xi) = 0 \quad\text{if }\beta > \beta_{c},
\end{equation}
where $\mathcal{Z}^{\beta}_{\infty}(0;\xi)$ is a non-degenerate random variable. Notably, $\beta_{L^{2}} < \beta_{c}$; see \cite{moment,junk_new,betaclessthenbetaL2_3} for details.
In particular, the almost sure convergence (\ref{limit_CDP}) implies the existence of a measurable function $\mathfrak{f}_{\beta}$,
defined on the space $\mathscr{S}^{'}(\mathbb{R}^{d} \times \mathbb{R})$ of Schwartz distributions, such that 
\begin{equation} \label{definition_f_beta}
    \mathfrak{f}_{\beta}(\xi) = \mathcal{Z}^{\beta}_{\infty}(0;\xi) \quad \text{(a.s.)}.
\end{equation}
The measurable function $\mathfrak{f}_{\beta}$ allows us to express the limiting partition function of the continuous directed polymer started at $x$ as
\begin{align} \label{definition_Z_x}
    \mathcal{Z}^{\beta}_{\infty}(x;\xi) = \mathfrak{f}_{\beta}(\xi(\cdot+x,\cdot)) \quad \forall x \in \mathbb{R}^{d}.
\end{align}

By using \eqref{definition_f_beta}, the authors of \cite{normal} proved that
$\mathbf{X}_{\varepsilon}^{\textup{st};\beta}(x,t)$ and 
$\mathbf{H}_{\varepsilon}^{\textup{st};\beta}(x,t)$, defined below, serve as $\varepsilon\to 0$ approximations to the mollified versions $\mathbf{X}_{\varepsilon}^{\beta}(x,t)$ and $\mathbf{H}_{\varepsilon}^{\beta}(x,t)$, respectively:
\begin{equation} \label{stationary}
    \mathbf{X}_{\varepsilon}^{\textup{st};\beta}(x,t) := \mathfrak{f}_{\beta}(\xi^{(\varepsilon,x,t)})
    \quad
    \text{and}
    \quad
    \mathbf{H}_{\varepsilon}^{\textup{st};\beta}(x,t) := \log \mathbf{X}_{\varepsilon}^{\textup{st};\beta}(x,t).
\end{equation}
See \cite[Section 2]{normal} for further discussions. Roughly speaking, for a fixed $\varepsilon>0$, (\ref{stationary}) can be viewed as the $t \to \infty$ limit of (\ref{partition_function_0}) while keeping the driving noise $\xi^{(\varepsilon,x,t)}$ fixed. In particular, (\ref{stationary}) are the stationary solutions of equations (\ref{Mollified_SHE}) and (\ref{Mollified_KPZ}) in the following sense:
\begin{lemma} \label{lemma:explanation_stationary_solution}
Assume that $d\geq 3$ and $\beta<\beta_{c}$. Then the following are ture.
\begin{enumerate} [label=(\roman*)]
    \item Both of the laws of $\mathbf{X}_{\varepsilon}^{\textup{st};\beta}(x,t)$ and $\mathbf{H}_{\varepsilon}^{\textup{st};\beta}(x,t)$ are stationary in space-time variables.
    \item For each $\varepsilon>0$, $\mathbf{X}_{\varepsilon}^{\textup{st};\beta}(x,t)$ and $\mathbf{H}_{\varepsilon}^{\textup{st};\beta}(x,t)$ are the solutions of the equations (\ref{Mollified_SHE}) and (\ref{Mollified_KPZ}) with random initial datum $\mathfrak{f}_{\beta}(\xi^{(\varepsilon,x,0)})$ and $\log\mathfrak{f}_{\beta}(\xi^{(\varepsilon,x,0)})$, respectively.
\end{enumerate}
\end{lemma}
\begin{proof}
The first property follows from the fact that $\xi^{(\varepsilon,x,t)}$ is a white noise. The second property can be proved by using the Markov property of Brownian
motion and (\ref{limit_CDP}).
\end{proof}

\subsection{Related literature and additional comments.} \label{related}
In this subsection, we explore the connection between our results and existing works in the literature. For simplicity, we only focus on the high-dimensional case.

\subsubsection{Fluctuations for the continuous directed polymer.} \label{section:Fluctuations_for_the_continuum_directed_random_polymer}
In this part, we concentrate on the large-scale fluctuations $\mathbf{F}^{\textup{PF};\beta}_{T}(x)$ and $\mathbf{F}^{\textup{FE};\beta}_{T}(x)$ \emph{in multi-point framework}. In particular, we continue the discussions from the second to last paragraph of Section \ref{section:overview_main_result_space_time}, where we attempt to apply \cite[Lemma 3.1]{rate_Comets} to a specific sequence of martingales. In the following, we demonstrate that one of the key conditions required by \cite[Lemma 3.1]{rate_Comets} fails to hold.

Suppose that $\beta < \beta_{L^{2}}$ and $M\geq 2$. To apply the martingale central limit theorem \cite[Lemma 3.1]{rate_Comets} to the vector $\sbkt{\mathbf{F}^{\textup{PF};\beta}_{T}(x_{j})}_{1\leq j\leq M}$, one may choose the following sequence of martingales indexed by $T \in \mathbb{N}$:
\begin{alignat}{2} 
    &S^{\Vec{\lambda}}_{T,k} &&:= \sum_{j=1}^{M} \Vec{\lambda}(j) \cdot T^{\frac{d-2}{4}} \sbkt{\mathcal{Z}_{T+k}^{\beta}(\sqrt{T}x_{j};\xi) - \mathcal{Z}_{T}^{\beta}(\sqrt{T}x_{j};\xi)},\quad k\in \mathbb{N}\cup \{0\},\quad \Vec{\lambda} \in \mathbb{R}^{M},\nonumber\\
    &\mathcal{F}_{T,k} &&:= \sigma(\xi(x,t):x\in \mathbb{R}^{d},\; t \leq T+k).\label{definition:sequence_maringales}
\end{alignat}
Then the main challenge in applying \cite[Lemma 3.1]{rate_Comets} lies in verifying the \emph{convergence in probability} of the conditional variance defined by  
\begin{align} \label{definition:conditional_variance}
    (V^{\Vec{\lambda}}_{T,\infty})^{2} := \sum_{k=1}^{\infty} \mathbb{E}[\sbkt{S^{\Vec{\lambda}}_{T,k} - S^{\Vec{\lambda}}_{T,k-1}}^{2}| \mathcal{F}_{T,k}].
\end{align}
According to \cite[Lemma 3.1]{rate_Comets}, it is necessary to show that there exists a non-negative random variable $(V^{\Vec{\lambda}})^{2}$ such that
\begin{align} \label{VT_to_V_in_probability}
    (V^{\Vec{\lambda}}_{T,\infty})^{2} \overset{\mathbb{P}}{\longrightarrow} (V^{\Vec{\lambda}})^{2} \quad \text{as} \quad T \to \infty.  
\end{align}
However, \eqref{VT_to_V_in_probability} fails to hold, which will be proved in Section \ref{section:proof_of_lemma:not_converge_in_probability}:
\begin{lemma} \label{lemma:not_converge_in_probability}
Assume that $d \geq 3$, $\beta < \beta_{L^{2}}$, and $M\geq 2$. Then, for any non-zero $\Vec{\lambda} \in \mathbb{R}^{M}$, $(V^{\Vec{\lambda}}_{T,\infty})^{2}$ does not converge in probability as $T \to \infty$. 
\end{lemma}
Let us now offer a heuristic verification for Lemma \ref{lemma:not_converge_in_probability}. Suppose that all requirements of \cite[Lemma 3.1]{rate_Comets}, especially including (\ref{VT_to_V_in_probability}), could be verified. From (\ref{limit_CDP}), we have
\begin{align*}
    S^{\Vec{\lambda}}_{T,\infty} = \sum_{j=1}^{M} \Vec{\lambda}(j) \cdot \mathbf{F}^{\text{PF};\beta}_{T}(x_{j}) \quad \forall T\in \mathbb{N},
\end{align*}
where $S^{\Vec{\lambda}}_{T,\infty} := \lim_{k\to\infty} S^{\Vec{\lambda}}_{T,k}$. Then, by applying \cite[Lemma 3.1]{rate_Comets}, we have
\begin{align} \label{consequence_of_MCLT}
    S^{\Vec{\lambda}}_{T,\infty} \overset{d}{\longrightarrow} V^{\Vec{\lambda}} \cdot \mathcal{N}(0,1), \quad\text{as} \quad T\to\infty,
\end{align}
where $\mathcal{N}(0,1)$ is the standard Gaussian random variable that is independent of the non-negative random variable $V^{\Vec{\lambda}}$ given in (\ref{VT_to_V_in_probability}). 

To see why \eqref{VT_to_V_in_probability} fails, we begin by examining the properties of $V^{\Vec{\lambda}}$. Notice that the probability distribution of $V^{\Vec{\lambda}}$ follows by using Corollary \ref{corollary_CDP}:
\begin{align} \label{limting_martinagle_ST}
    S^{\Vec{\lambda}}_{T,\infty} &\overset{d}{\longrightarrow} \sum_{j=1}^{M} \Vec{\lambda}(j) \cdot \mathbf{Z}^{\beta}(x_{j},1) \mathbf{U}^{\beta}(x_{j},1) \nonumber\\
    &\overset{d}{=} \bkt{\sum_{1\leq j,j' \leq M} \Vec{\lambda}(j)\Vec{\lambda}(j')\mathbf{Z}^{\beta}(x_{j},1) \mathbf{Z}^{\beta}(x_{j'},1) \gamma^{2}(\beta) \int_{0}^{\infty} G_{2s+2}(x_{j}-x_{j'}) ds} \cdot \mathcal{N}(0,1),
\end{align}
where the first random variable on the right-hand side of \eqref{limting_martinagle_ST}, denoted by $(U^{\Vec{\lambda}})^{2}$, is independent of
$\mathcal{N}(0,1)$. Comparing \eqref{limting_martinagle_ST} with (\ref{consequence_of_MCLT}), we see that $V^{\Vec{\lambda}}$, given in (\ref{VT_to_V_in_probability}), has the same probability distribution as $U^{\Vec{\lambda}}$, which may further suggest that $V^{\Vec{\lambda}}$ should have a similar expression as $U^{\Vec{\lambda}}$. In this context, $V^{\Vec{\lambda}}$ should consist of a term that corresponds to $\mathbf{Z}^{\beta}(x_{j},1) \mathbf{Z}^{\beta}(x_{j'},1)$, which may originate from  $\mathcal{Z}_{\tau_{T}}^{\beta}(\sqrt{T}x_{j};\xi)\mathcal{Z}_{\tau_{T}}^{\beta}(\sqrt{T}x_{j'};\xi)$, due to \eqref{pointwise_convergence_formula_proof_section_goal_1}. Here, $\tau_{T} \to \infty$ as $T \to \infty$. By combining the above observations with (\ref{VT_to_V_in_probability}), we expect 
\begin{align} \label{sum_lambda}
    \sum_{1\leq j,j' \leq M} \Vec{\lambda}(j)\Vec{\lambda}(j') C_{\infty}(j,j') \mathcal{Z}_{\tau_{T}}^{\beta}(\sqrt{T}x_{j};\xi)\mathcal{Z}_{\tau_{T}}^{\beta}(\sqrt{T}x_{j'};\xi) \overset{\mathbb{P}}{\longrightarrow} (V^{\Vec{\lambda}})^{2}
\end{align}
for some positive deterministic coefficients $C_{\infty}(j,j')$.


However, $\mathcal{Z}_{\tau_{T}}^{\beta}(\sqrt{T}x_{j};\xi)\mathcal{Z}_{\tau_{T}}^{\beta}(\sqrt{T}x_{j'};\xi)$ \emph{does not converge in probability whenever either $x_{j}$ or $x_{j'}$ is nonzero}, a situation that necessarily arises in the multi-point framework. This is because the convergence of $\mathcal{Z}_{\tau_{T}}^{\beta}(\sqrt{T}x_{j};\xi)\mathcal{Z}_{\tau_{T}}^{\beta}(\sqrt{T}x_{j'};\xi)$ contradicts to the fact that the limiting partition function $\mathcal{Z}_{\infty}^{\beta}(0;\xi)$ is a non-degenerate random variable, which was proved in \cite[Corollary 2.5]{weakandstrong}. Therefore, it is unlikely that condition (\ref{VT_to_V_in_probability}) can hold.

\subsubsection{Spatially averaged fluctuations.} \label{section:Spatially_averaged_fluctuations}
We now turn to spatial averaging frameworks. The Gaussianity of $\overline{\mathbf{F}}^{\textup{KPZ};\beta}_{\varepsilon}(f,t)$ was first proved by \cite[Theorem 2.3]{comets2019space} in a restricted part of the $L^{2}$-regime. As an application of Theorem~\ref{ST_SHE_and_KPZ}, we extend \cite[Theorem 2.3]{comets2019space} up to the $L^2$-threshold and include the SHE case as well. Specifically, we define
\begin{equation} \label{fluctuations_mean_2}
    \overline{\mathbf{F}}^{\bullet;\beta}_{\varepsilon}(f,t) := \int_{\mathbb{R}^{d}} dx f(x) \mathbf{F}^{\bullet;\beta}_{\varepsilon}(x,t),\quad \bullet = \textup{SHE} \text{ or } \textup{KPZ},
\end{equation}
and show that 
\begin{align} \label{fluctuation:overline}
    \bkt{\overline{\mathbf{F}}^{\bullet;\beta}_{\varepsilon}(f_{j},t_{j})}_{1\leq j\leq M} \overset{d}{\longrightarrow} \bkt{\int_{\mathbb{R}^{d}} dx f_{j}(x) \mathbf{U}^{\beta}(x,t_{j})}_{1\leq j\leq M} \quad \forall\beta<\beta_{L^{2}}.
\end{align}
Here, the test functions $f_{1},\ldots,f_{M} \in C_{c}^{\infty}(\mathbb{R}^{d})$ and $t_{1},\ldots,t_{M} > 0$. We will briefly outline the proof ideas of (\ref{fluctuation:overline}) in Section \ref{section:fluctuation_overline}. 

In addition to the fluctuations in (\ref{fluctuations_mean_2}), one can also consider the following spatially averaged fluctuations:
\begin{align}
    \label{fluctuations_mean}
    &\bkt{\int_{\mathbb{R}^{d}} dx f_{j}(x)\varepsilon^{-\frac{d-2}{2}} \sbkt{\mathbf{J}^{\beta}_{\varepsilon}(x,t_{j})-\mathbb{E}[\mathbf{J}^{\beta}_{\varepsilon}(x,t_{j})]}}_{1\leq j\leq M}, \quad \mathbf{J}^{\beta}_{\varepsilon}(x,t) = \mathbf{X}^{\beta}_{\varepsilon}(x,t) \text{ or } \mathbf{H}^{\beta}_{\varepsilon}(x,t).
\end{align}
These types of fluctuations have been studied in recent works, including \cite{COSCO2022127,gu2018edwards,guKPZ,scaling}. From our perspective, it is important to investigate the limits of \eqref{fluctuations_mean} for all values of $\beta$, as the limiting distributions may vary with increasing noise strength. In particular, the $L^{2}$-regime may represent a critical regime, where the fluctuations in \eqref{fluctuations_mean} undergo a phase transition. Notably, Cosco, Nakajima, and Nakashima \cite{COSCO2022127} proved Gaussianity of the fluctuations in \eqref{fluctuations_mean} throughout the full $L^2$-regime. For related developments, see \cite{nonlinearSHE,nonlinearSHE2,polymer_3d}. In particular, \cite{junk_new_new} analyzes fluctuations of the discrete directed polymer in the weak disorder regime beyond the $L^2$-regime.



\section{Heuristic arguments} \label{section_heuristics}
In this section, we briefly explain the gist and the key ideas to conclude Theorem \ref{ST_SHE_and_KPZ}, with particular emphasis on the how we bypass the main obstacle, namely \emph{the use of any higher moment bounds for the limiting partition function}. Moreover, it will become clear that the rescaled fluctuation $\mathbf{F}^{\textup{\text{SHE}};\beta}_{\varepsilon}(x,t)$ naturally arises in the proof of (\ref{KPZ_pointwise_fluctuation}). 


Let us now concentrate on the proof of (\ref{KPZ_pointwise_fluctuation}).  We begin by using the idea of a linearization applied in \cite[Section 3.6]{GFDDP} to simplify the logarithmic
transformations in $\mathbf{F}^{\text{KPZ};\beta}_{\varepsilon}(x,t)$ as follows:
\begin{align} \label{KPZ_approxi} 
    \mathbf{F}^{\text{KPZ};\beta}_{\varepsilon}(x,t) 
    \overset{\varepsilon \to 0}{\approx} \varepsilon^{-\frac{d-2}{2}} \frac{\mathbf{X}_{\varepsilon}^{\textup{st};\beta}(x,t)-\mathbf{X}^{\beta}_{\varepsilon}(x,t)}{\mathbf{X}^{\beta}_{\varepsilon}(x,t)}.
\end{align}
More specifically, we will prove that, with high probability,
\begin{equation*}
    \frac{\mathbf{X}_{\varepsilon}^{\textup{st};\beta}(x,t)-\mathbf{X}^{\beta}_{\varepsilon}(x,t)}{\mathbf{X}^{\beta}_{\varepsilon}(x,t)}
    \quad \text{is of order $\varepsilon^{\frac{d-2}{2}}$ as } \varepsilon \to 0.
\end{equation*}
Consequently, applying the Taylor expansion of $\log(1+x)$ at $x = 0$ shows that the rescaled fluctuation $\mathbf{F}^{\textup{\text{KPZ}};\beta}_{\varepsilon}(x,t)$ is dominated by the linear term of its Taylor expansion: 
\begin{align*} 
    \mathbf{F}^{\text{KPZ};\beta}_{\varepsilon}(x,t) 
    = \varepsilon^{-\frac{d-2}{2}}\log\bkt{1+\frac{\mathbf{X}_{\varepsilon}^{\textup{st};\beta}(x,t)-\mathbf{X}^{\beta}_{\varepsilon}(x,t)}{\mathbf{X}^{\beta}_{\varepsilon}(x,t)}}
    \overset{\varepsilon \to 0}{\approx} \text{(R.H.S) of (\ref{KPZ_approxi})}.
\end{align*}

Now we further develop a linear approximation to the right-hand side of (\ref{KPZ_approxi}). The idea starts with studying the rescaled fluctuation $\mathbf{F}^{\textup{SHE};\beta}_{\varepsilon}(x,t)$, which appears in the numerator on the right-hand side (\ref{KPZ_approxi}). The \emph{main technique} is based on the following \emph{asymptotic decomposition}:
\begin{align}\label{KPZ_approxii}
    T^{\frac{d-2}{4}}\sbkt{\mathcal{Z}^{\beta}_{\infty}(x\sqrt{T};\xi) - \mathcal{Z}^{\beta}_{Tt}(x\sqrt{T};\xi)} &\overset{T\to\infty}{\approx}
    \mathcal{Z}^{\beta}_{Tt}(x\sqrt{T};\xi) \nonumber\\
    &\times\int_{\mathbb{R}^{d}} dy G_{t}(x - y)
    T^{\frac{d-2}{4}}
    (\mathcal{Z}^{\beta}_{\infty}(y\sqrt{T};\xi(\cdot,\cdot+tT))-1).
\end{align}
To prove (\ref{KPZ_approxii}), it will become clear that that handling the limiting error introduced by the approximation (\ref{KPZ_approxii}) requires establishing a uniform bound for the following expected Brownian-bridge  exponential functional, as well as proving its asymptotic behavior:
\begin{equation} \label{definition:brownian_functional} 
    \mathbf{E}_{0,0}^{T,\sqrt{T}z}\bigbkt{\exp\bkt{\beta^{2}\int_{0}^{T} R(\sqrt{2}B(s))}},
\end{equation}
where the function $R$ is defined in (\ref{definition_R_epsilon}) and $\mathbf{P}_{0,a}^{T,b}$ denotes the law of the $d$-dimensional Brownian bridge starting from $(0,a)$ to $(T,b)$.

Let us now introduce the \emph{key linear approximation} used in our argument:
\begin{equation} \label{KPZ_approxiiii} 
    \varepsilon^{-\frac{d-2}{2}} \frac{\mathbf{X}_{\varepsilon}^{\textup{st};\beta}(x,t)-\mathbf{X}^{\beta}_{\varepsilon}(x,t)}{\mathbf{X}^{\beta}_{\varepsilon}(x,t)}
    \overset{\varepsilon\to 0}{\approx} \int_{\mathbb{R}^{d}} dy G_{t}(x - y)
    \varepsilon^{-\frac{d-2}{2}}
    (\mathcal{Z}^{\beta}_{\infty}(y/\varepsilon;\xi^{(\varepsilon,0,0)})-1).
\end{equation}
The approximation \eqref{KPZ_approxiiii} is crucial, as it removes the most problematic term, the denominator on the left-hand side of (\ref{KPZ_approxiiii}), and thereby allows us to avoid using higher moment bounds for the partition function. The above linear approximation is achieved by the following \emph{cancellation}. Combining the approximation (\ref{KPZ_approxii}) with the Feynman–Kac representations (\ref{partition_function_0}) and (\ref{stationary}) yields 
\begin{align} \label{KPZ_approxiii}
    \varepsilon^{-\frac{d-2}{2}}\sbkt{\mathbf{X}_{\varepsilon}^{\textup{st};\beta}(x,t) - \mathbf{X}^{\beta}_{\varepsilon}(x,t)} 
    \overset{\varepsilon\to 0}{\approx}
    \mathbf{X}^{\beta}_{\varepsilon}(x,t)\times \int_{\mathbb{R}^{d}} dy G_{t}(x - y)
    \varepsilon^{-\frac{d-2}{2}}
    (\mathcal{Z}^{\beta}_{\infty}(y/\varepsilon;\xi^{(\varepsilon,0,0)})-1),
\end{align}
where the noise $\xi^{(\varepsilon,0,0)}$ is defined in (\ref{rescaled_space_time_noise}). Consequently, by combining (\ref{KPZ_approxi}) and (\ref{KPZ_approxiii}), we observe that the factor $\mathbf{X}^{\beta}_{\varepsilon}(x,t)$ on the right-hand side of (\ref{KPZ_approxiii}) cancels with the corresponding term in the denominator of the right-hand side of (\ref{KPZ_approxi}). Therefore, we conclude (\ref{KPZ_approxiiii}).


The cancellation in foregoing step is the main reason why the rescaled fluctuations $\mathbf{F}^{\text{SHE};\beta}_{\varepsilon}(x,t)$ and $\mathbf{F}^{\text{KPZ};\beta}_{\varepsilon}(x,t)$ have different limiting distributions, since the term $\mathbf{X}^{\beta}_{\varepsilon}(x,t)$, which causes the random perturbations, is removed in the KPZ case. In particular, the remaining term given below is the source that produces the Gaussian limit in \eqref{KPZ_pointwise_fluctuation}:
\begin{align} \label{KPZ_approxiiiii}
    \mathbf{F}^{\text{KPZ};\beta}_{\varepsilon}(x,t) \overset{\varepsilon\to 0}{\approx}
    \int_{\mathbb{R}^{d}} dy G_{t}(x - y)
    \varepsilon^{-\frac{d-2}{2}}
    (\mathcal{Z}^{\beta}_{\infty}(y/\varepsilon;\xi^{(\varepsilon,0,0)})-1).
\end{align}
Here, the above approximation follows by combining the approximations (\ref{KPZ_approxi}) and (\ref{KPZ_approxiiii}). Moreover, the Gaussianity of the right-hand side of (\ref{KPZ_approxiiiii}) had been obtained by Cosco, Nakajima, and Nakashima
in \cite{COSCO2022127}. Specifically, by applying the martingale central limit theorem \cite[p. 473]{LT}, they showed that 
\begin{align} \label{first_gaussian}
    \int_{\mathbb{R}^{d}} dy f(y)
    \varepsilon^{-\frac{d-2}{2}}
    (\mathcal{Z}^{\beta}_{\infty}(y/\varepsilon;\xi)-1) \overset{\textup{d}}{\longrightarrow} \int_{\mathbb{R}^{d}} dy f(y)\mathbf{U}_{0}^{\beta}(y) \quad f \in C_{c}^{\infty}(\mathbb{R}^{d}).
\end{align}
Furthermore, the above test functions can be further extended to Schwarz functions. Hence, by taking $f(y) := G_{t}(x-y)$, the right-hand side of (\ref{first_gaussian}) is equal to $\mathbf{U}^{\beta}(x,t)$, which is defined in Theorem \ref{ST_SHE_and_KPZ}. Therefore, we deduce that
$\mathbf{F}^{\textup{\text{KPZ}};\beta}_{\varepsilon}(x,t)$ converges in law to $\mathbf{U}^{\beta}(x,t)$, which implies \eqref{KPZ_pointwise_fluctuation}. 

Finally, let us turn to the rescaled fluctuation $\mathbf{F}^{\textup{\text{SHE}};\beta}_{\varepsilon}(x,t)$. The \emph{key idea} is that the first and second terms on the right-hand side of \eqref{KPZ_approxiii} are independent, since they are driven by disjoint parts of the white noise, namely, $(\xi(x,t))_{(x,t)\in \mathbb{R}^{d} \times \mathbb{R}_{+}}$ and $(\xi(x,t))_{(x,t)\in \mathbb{R}^{d} \times \mathbb{R}_{-}}$, respectively. 
Consequently, we can determine the limiting distribution of the product on the right-hand side of (\ref{KPZ_approxiii}) by analyzing each term separately and then combining the results via independence. Moreover, the approximation scheme $\sbkt{\mathbf{X}^{\beta}_{\varepsilon}(x,t)}_{(x,t) \in \mathbb{R}^{d} \times \mathbb{R}_{+}}$ converges in the sense of finite dimensional convergence towards $\sbkt{\mathbf{Z}^{\beta}(x,t)}_{(x,t)\in \mathbb{R}^{d} \times \mathbb{R}_{+}}$. This follows  (\ref{limit_CDP}) together with the asymptotic independence of  $\mathbf{X}^{\beta}_{\varepsilon}(x_{1},t_{1})$ and $\mathbf{X}^{\beta}_{\varepsilon}(x_{2},t_{2})$ whenever $(x_{1},t_{1}) \neq (x_{2},t_{2})$. This phenomenon anticipated in \cite[Remark 2.3]{weakandstrong} and proved in detail in Section \ref{section_asymptotics of the renormalizations}. Therefore, by combining the convergence of the approximation scheme $\sbkt{\mathbf{X}^{\beta}_{\varepsilon}(x,t)}_{(x,t) \in \mathbb{R}^{d} \times \mathbb{R}_{+}}$ and the Gaussian limit (\ref{first_gaussian}) with test function $f(y) := G_{t}(x-y)$, we obtain the desired result (\ref{SHE_pointwise_fluctuation}).



\section{Outline of the proof}\label{section_proofoutline}
In this section, we outline major steps to conclude Theorem \ref{ST_SHE_and_KPZ}. For clarity, we first establish (\ref{SHE_pointwise_fluctuation}) in Section \ref{section_proof_SHE_pointwise} and then verify (\ref{KPZ_pointwise_fluctuation}) in Section \ref{section_proof_KPZ_pointwise}. We begin by recalling key properties of the partition function, which were proved in \cite[Corollary 2.7]{COSCO2022127} and \cite[Lemma 5.8]{COSCO2022127}, respectively. 



\begin{theorem} [C. Cosco, S. Nakajima, and M. Nakashima \cite{COSCO2022127}]\label{properties_Z}
Assume that $\beta < \beta_{L^{2}}$ and $d \geq 3$. Then the following holds.
\begin{enumerate} [label=(\roman*)]
    \item For any collection of test functions $(f_{j})_{1\leq j\leq M} \subseteq \mathscr{S}(\mathbb{R}^{d})$, as $T\to \infty$,
    \begin{align}\label{corollary_2.7}
    \bkt{\int_{\mathbb{R}^{d}} dx f_{j}(x) T^{\frac{d-2}{4}}(\mathcal{Z}^{\beta}_{\infty}(x\sqrt{T};\xi)-1)}_{1\leq j\leq M} \overset{d}{\longrightarrow}
    \bkt{\int_{\mathbb{R}^{d}} dx f_{j}(x) \mathbf{U}^{\beta}(x,0)}_{1\leq j\leq M},
    \end{align}
    where $\mathcal{Z}^{\beta}_{\infty}(x\sqrt{T};\xi)$ is defined in (\ref{definition_Z_x}) and $\mathbf{U}^{\beta}(x,0)$ is defined in Theorem \ref{ST_SHE_and_KPZ}.
    \item For each positive integer $m$, it holds that
    \begin{equation} \label{negative_moment_bound}
    \mathbb{E}[\sup_{T\geq 0} \mathcal{Z}_{T}^{\beta}(0;\xi)^{-m}] < \infty.
    \end{equation}
\end{enumerate}
\end{theorem}
\begin{remark}
It is straightforward to extend the test functions in \cite[Corollary 2.7]{COSCO2022127} from compactly supported smooth functions to Schwarz functions
by using smooth cutoff functions that approximate to $1$ on $\mathbb{R}^{d}$.  
\end{remark}

\subsection{Fluctuation for the mollified SHE.} \label{section_proof_SHE_pointwise}
In this subsection, we present the key steps to verify \eqref{SHE_pointwise_fluctuation}. 

\paragraph{Step 1: Large-scale setup.} We begin by deriving a representation of the rescaled fluctuation $\mathbf{F}_{\varepsilon}^{\textup{SHE};\beta}(x,t)$, which provides a sufficient condition for \eqref{SHE_pointwise_fluctuation}. Fix a collection of space-time points 
\begin{align} \label{definition:space_time_points}
    (x_{j},t_{j})_{1\leq j\leq M}\subseteq \mathbb{R}^{d} \times \mathbb{R}_{+}, \quad 0 < t_{1} \leq t_{2} \leq \ldots \leq t_{M},
\end{align}
and set
\begin{align} \label{notation:change_of_variable}
    \delta_{j} := t_{M} - t_{j} \quad \forall 1\leq j\leq M, \quad T := \varepsilon^{-2}.
\end{align}
With these notations, $\mathbf{F}_{\varepsilon}^{\textup{SHE};\beta}(x_{j},t_{j})$ can then be expressed in terms of the partition functions, as shown in (\ref{SHE_D1}). In particular, to prove (\ref{SHE_pointwise_fluctuation}), it suffices to show that
\begin{equation} \label{SHE_D2}
    \text{(R.H.S) of (\ref{SHE_D1})} \overset{d}{\longrightarrow} \sbkt{\mathbf{Z}^{\beta}(x_{j},t_{j})\cdot \mathbf{U}^{\beta}(x_{j},t_{j})}_{1\leq j\leq M}
    \quad\text{as} \quad T\to\infty.
\end{equation}
\begin{lemma}
Suppose that $d\geq 3$, $\beta <\beta_{L^{2}}$, and $\varepsilon>0$. Then, it holds that
\begin{equation} \label{SHE_D0}
    \bkt{\mathbf{X}_{\varepsilon}^{\textup{st};\beta}(x_{j},t_{j}),\mathbf{X}^{\beta}_{\varepsilon}(x_{j},t_{j})}_{1\leq j\leq M}
    \overset{d}{=} \bkt{\mathcal{Z}^{\beta}_{\infty}(x_{j}\sqrt{T};\xi(\cdot,\cdot+\delta_{j} T)),\mathcal{Z}^{\beta}_{t_{j} T}(x_{j}\sqrt{T};\xi(\cdot,\cdot+\delta_{j} T))}_{1\leq j\leq M}.
\end{equation}
In particular, we have
\begin{align} \label{SHE_D1}
    \sbkt{\mathbf{F}_{\varepsilon}^{\textup{SHE};\beta}(x_{j},t_{j})}_{1\leq j\leq M}\overset{d}{=} 
    \bkt{T^{\frac{d-2}{4}} 
    \bkt{
    \mathcal{Z}_{\infty}^{\beta}(x_{j}\sqrt{T};\xi(\cdot,\cdot+\delta_{j} T))
    -
    \mathcal{Z}_{t_{j} T}^{\beta}(x_{j}\sqrt{T};\xi(\cdot,\cdot+\delta_{j} T))
    }}_{1\leq j\leq M}.
\end{align}
\end{lemma}
\begin{proof}
    
We recall the Feynman-Kac representations (\ref{partition_function_0}) and (\ref{stationary}) and note that 
\begin{equation} \label{FK}
    \mathbf{X}^{\beta}_{\varepsilon}(x,t)  
    = \mathcal{Z}^{\beta}_{tT}(\sqrt{T}x;\xi^{(T^{-1/2},0,t)}), \quad
    \mathbf{X}_{\varepsilon}^{\textup{st};\beta}(x,t)  = \mathcal{Z}^{\beta}_{\infty}(\sqrt{T}x;\xi^{(T^{-1/2},0,t)}),
\end{equation}
where $\xi^{(\varepsilon,x,t)}(\cdot,\cdot)$ is defined in (\ref{rescaled_space_time_noise}).
As a result, by combining (\ref{FK}) with the identity 
\begin{equation} \label{change_noise}
    \xi^{(T^{-1/2},0,t_{j})}(y,s) = \xi^{(T^{-1/2},0,t_{M})}(y,s+\delta_{j}T) \quad \forall 1\leq j\leq M,
\end{equation}
we conclude
\begin{align} \label{SHE_D0_1}
    &\bkt{\mathbf{X}_{\varepsilon}^{\textup{st};\beta}(x_{j},t_{j}),\mathbf{X}^{\beta}_{\varepsilon}(x_{j},t_{j})}_{1\leq j\leq M}\nonumber\\
    &= \bkt{\mathcal{Z}^{\beta}_{\infty}(x_{j}\sqrt{T};\xi^{(T^{-1/2},0,t_{M})}(\cdot,\cdot+\delta_{j} T)),\mathcal{Z}^{\beta}_{t_{j} T}(x_{j}\sqrt{T};\xi^{(T^{-1/2},0,t_{M})}(\cdot,\cdot+\delta_{j} T))}_{1\leq j\leq M}.
\end{align}
Therefore, (\ref{SHE_D0}) follows from the fact that $\xi^{(T^{-1/2},0,t_{M})}(\cdot,\cdot)$ is a white noise.
\end{proof}

\paragraph{Step 2: Asymptotic decomposition.}
To prove (\ref{SHE_D2}), we further simplify the right-hand side of (\ref{SHE_D1}). The \emph{key observation} is the $L^{2}$-approximation \eqref{SHE_key_lemma_formula}. For simplicity, we postpone its proof until Section \ref{SHE_key_proposition_section}.

\begin{proposition}\label{SHE_key_proposition}
Assume that $d \geq 3$ and $\beta < \beta_{L^{2}}$. Then, for each $(x,t) \in \mathbb{R}^{d} \times \mathbb{R}_{+}$, as $T \to \infty$, 
\begin{align} \label{SHE_key_lemma_formula}
    T^{\frac{d-2}{4}}\sbkt{&\mathcal{Z}^{\beta}_{\infty}(x\sqrt{T};\xi)
    -
    \mathcal{Z}^{\beta}_{t T}(x\sqrt{T};\xi)}\nonumber\\
    &= \mathcal{Z}^{\beta}_{Tt}(x\sqrt{T};\xi) \times T^{\frac{d-2}{4}}\int_{\mathbb{R}^{d}} dy G_{t}(x - y)
    \sbkt{\mathcal{Z}^{\beta}_{\infty}(y\sqrt{T};\xi(\cdot,\cdot+tT))-1} + o(1) \quad\text{in } L^{2}.
\end{align}
\end{proposition}

\paragraph{Step 3.} 
Now we are ready to prove (\ref{SHE_D2}). Applying Proposition \ref{SHE_key_proposition} implies the $L^{2}-$approximation
\begin{align} \label{SHE_L2}
    &\text{(R.H.S) of (\ref{SHE_D1})} = \nonumber\\
    &
    \bkt{\mathcal{Z}^{\beta}_{t_{j} T}(x_{j}\sqrt{T};\xi(\cdot,\cdot+\delta_{j} T))\times T^{\frac{d-2}{4}}\int_{\mathbb{R}^{d}} dy G_{t_{j}}(x_{j} - y)
    (\mathcal{Z}^{\beta}_{\infty}(y\sqrt{T};\xi(\cdot,\cdot+t_{M}T))-1)
    }_{1\leq j\leq M}
    + o(1).    
\end{align}
The key point for proving (\ref{SHE_D2}) is that the two components on the right-hand side of (\ref{SHE_L2}) are independent. Indeed, by (\ref{partition_function_1}) and (\ref{definition_Z_x}), the first term depends on the noise $\sbkt{\xi(x,t)}_{(x,t)\in\mathbb{R}^{d}\times(0,t_{M}T)}$, while the second term depends on $\sbkt{\xi(x,t)}_{(x,t)\in\mathbb{R}^{d}\times(t_{M}T,\infty)}$. Hence their limits can be analyzed separately. Notice that the convergence of $\sbkt{\mathcal{Z}^{\beta}_{t_{j} T}(x_{j}\sqrt{T};\xi(\cdot,\cdot+\delta_{j} T))}_{1\leq j\leq M}$ follows from (\ref{SHE_D0}) and Lemma \ref{pointwise_convergence}. In addition, by Theorem \ref{properties_Z}, it holds that
\begin{align} \label{SHE_gaussian_limit_0}
    \bkt{T^{\frac{d-2}{4}}\int_{\mathbb{R}^{d}} dy G_{t_{j}}(x_{j} - y)(\mathcal{Z}^{\beta}_{\infty}(y\sqrt{T};\xi(\cdot,\cdot+t_{M}T))-1)}_{1\leq j\leq M}  \overset{d}{\longrightarrow} \sbkt{ G_{t_{j}} * \mathbf{U}_{0}^{\beta}(x_{j})}_{1\leq j\leq M},
\end{align}
where we have used the translation-invariance of the white noise. Therefore, (\ref{SHE_D2}) follows by combining these convergences. 

\subsection{Fluctuation for the mollified KPZ equation.} \label{section_proof_KPZ_pointwise}
We now establish (\ref{KPZ_pointwise_fluctuation}). Before proceeding, we first derive a sufficient condition for \eqref{KPZ_pointwise_fluctuation}. Observe the following large-scale form of the rescaled fluctuation $\mathbf{F}_{\varepsilon}^{\textup{KPZ};\beta}(x,t)$: 
\begin{align} \label{KPZ_D0}
    \sbkt{\mathbf{F}_{\varepsilon}^{\textup{KPZ};\beta}&(x_{j},t_{j})}_{1\leq j\leq M}\overset{d}{=} \nonumber\\
    &\bkt{T^{\frac{d-2}{4}} 
    \bkt{
    \log\mathcal{Z}^{\beta}_{\infty}(x_{j}\sqrt{T};\xi(\cdot,\cdot+\delta_{j} T))
    -
    \log\mathcal{Z}^{\beta}_{t_{j} T}(x_{j}\sqrt{T};\xi(\cdot,\cdot+\delta_{j} T))
    }}_{1\leq j\leq M},
\end{align}
where we have used (\ref{SHE_D0}). The collection of space-time points $(x_{j},t_{j})_{1\leq j\leq M}$ is defined in (\ref{definition:space_time_points}), and 
the parameter 
$\delta_{j}$ is defined in (\ref{notation:change_of_variable}). Hence, to prove (\ref{KPZ_pointwise_fluctuation}), it is enough to show that
\begin{equation} \label{KPZ_D00}
    \text{(R.H.S) of (\ref{KPZ_D0})} \overset{d}{\longrightarrow} \sbkt{\mathbf{U}^{\beta}(x_{j},t_{j})}_{1\leq j\leq M}
    \quad\text{as } T\to\infty.
\end{equation}

\paragraph{Step 1: Linearizing the logarithm.} We begin by simplifying the right-hand side of (\ref{KPZ_D0}).
\begin{proposition} \label{proposition_KPZ_D1}
Assume that $d\geq 3$ and $\beta<\beta_{L^{2}}$. Then, we have the following $L^{1}$-approximation:
\begin{align} \label{KPZ_D1}
    &T^{\frac{d-2}{4}} 
    \bkt{
    \log\mathcal{Z}^{\beta}_{\infty}(x_{j}\sqrt{T};\xi(\cdot,\cdot+\delta_{j} T))
    -
    \log\mathcal{Z}^{\beta}_{t_{j} T}(x_{j}\sqrt{T};\xi(\cdot,\cdot+\delta_{j} T))
    }\nonumber\\
    &\quad\quad\quad\quad= T^{\frac{d-2}{4}} \times
    \frac{\mathcal{Z}^{\beta}_{\infty}(x_{j}\sqrt{T};\xi(\cdot,\cdot+\delta_{j} T))
    -
    \mathcal{Z}^{\beta}_{t_{j} T}(x_{j}\sqrt{T};\xi(\cdot,\cdot+\delta_{j} T))}{\mathcal{Z}^{\beta}_{t_{j} T}(x_{j}\sqrt{T};\xi(\cdot,\cdot+\delta_{j} T))}+o(1).
\end{align}
\end{proposition}
Proposition \ref{proposition_KPZ_D1} follows from the moment bounds stated below, which will be verified in Section \ref{ratio_bound_secion}.

\begin{proposition}\label{ratio_bound}
For each $d \geq 3$ and every $\beta < \beta_{L^{2}}$, there exists $\delta > 0$ such that 
\begin{equation} \label{I_0}
    \sup_{T \geq 1}\mathbf{M}^{(\nu)}_{T;\beta,2+\delta} \leq C \quad \forall \nu = 1,2,
\end{equation}
where $C = C(\phi,\beta,d,\delta)$ is a positive constant and $\mathbf{M}^{(\nu)}_{T;\beta,p}$ is given by
\begin{equation} \label{I}
    \mathbf{M}^{(1)}_{T;\beta,p} := \bigl | \bigl |
    T^{\frac{d-2}{4}}\sbkt{\log \mathcal{Z}^{\beta}_{\infty}(0;\xi) - \log \mathcal{Z}^{\beta}_{T}(0;\xi)}
    \bigr | \bigr |_{L^{p}},
    \quad
    \mathbf{M}^{(2)}_{T;\beta,p} := \biggl | \biggl |
    T^{\frac{d-2}{4}} 
    \frac{
    \mathcal{Z}^{\beta}_{\infty}
    (0;\xi)-
    \mathcal{Z}^{\beta}_{T}(0;\xi)
    }{\mathcal{Z}^{\beta}_{T}(0;\xi)} 
    \biggr | \biggr |_{L^{p}}.
\end{equation}
\end{proposition}

\begin{remark}
The proof of Proposition~\ref{proposition_KPZ_D1} only requires the estimate \eqref{I_0} with $\delta = 0$. However, to establish \eqref{I_0} with $\nu = 1$ (even in the case $\delta = 0$), one needs the estimate \eqref{I_0} with $\nu = 2$ for some $\delta > 0$. This is why we prove \eqref{I_0} for some  $\delta > 0$.
\end{remark}


\begin{proof} [Proof of Proposition \ref{proposition_KPZ_D1}.]
To control the $L^{1}$-error induced by (\ref{KPZ_D1}), we observe that $|\log(1+x) - x| \leq \overline{C}x^{2}$ for every $|x| \leq \frac{1}{2}$, which follows from Taylor's theorem. As a result, the error can be bounded in the following way:
\begin{align} \label{key_lemma_error}
    &\text{The $L^{1}$-error induced by (\ref{KPZ_D1})}\leq \frac{\overline{C}}{
    T^{\frac{d-2}{4}}}
    \biggl | \biggl |
    T^{\frac{d-2}{4}} 
    \frac{
    \mathcal{Z}_{\infty}^{\beta}(0;\xi)
    -
    \mathcal{Z}_{tT}^{\beta}(0;\xi)
    }{\mathcal{Z}_{tT}^{\beta}(0;\xi)}
    \biggr | \biggr |_{L^{2}}^{2}+\nonumber\\
    &\mathbb{E}\bigbkt{\abs{T^{\frac{d-2}{4}}\sbkt{
    \log\mathcal{Z}_{\infty}^{\beta}(0;\xi)
    - 
    \log\mathcal{Z}_{tT}^{\beta}(0;\xi)
    }-T^{\frac{d-2}{4}} 
    \frac{
    \mathcal{Z}_{\infty}^{\beta}(0;\xi)
    -
    \mathcal{Z}_{tT}^{\beta}(0;\xi)
    }{\mathcal{Z}_{tT}^{\beta}(0;\xi)}};\abs{ 
    \frac{
    \mathcal{Z}_{\infty}^{\beta}(0;\xi)
    -
    \mathcal{Z}_{tT}^{\beta}(0;\xi)
    }{\mathcal{Z}_{tT}^{\beta}(0;\xi)}}\geq \frac{1}{2}}.
\end{align}
In addition, by Hölder's inequality and Markov's inequality, the second term on the right-hand side of (\ref{key_lemma_error}) can be further bounded by
\begin{align} \label{def_A_2_T}
    \biggl | \biggl |
    T^{\frac{d-2}{2}}
    \sbkt{
    \log\mathcal{Z}^{\beta}_{\infty}(0;\xi)- 
    \log\mathcal{Z}^{\beta}_{tT}(0;\xi)
    }- 
    T^{\frac{d-2}{2}}
    &\frac{
    \mathcal{Z}^{\beta}_{\infty}(0;\xi)
    -
    \mathcal{Z}^{\beta}_{tT}(0;\xi)
    }{\mathcal{Z}^{\beta}_{tT}(0;\xi)}
    \biggr | \biggr |_{L^{2}}\nonumber\\
    &\times \frac{2}{
    T^{\frac{d-2}{4}}}
    \biggl | \biggl |
    T^{\frac{d-2}{4}} 
    \frac{
    \mathcal{Z}^{\beta}_{\infty}(0;\xi)
    -
    \mathcal{Z}_{tT}(0;\xi)
    }{\mathcal{Z}_{tT}(0;\xi)}
    \biggr | \biggr |_{L^{2}}.
\end{align}
THerefore, by Proposition \ref{ratio_bound} and \eqref{def_A_2_T}, both terms on the right-hand side of (\ref{key_lemma_error}) vanish as $T\to\infty$. 
\end{proof}

\paragraph{Step 2: Linearizing the ratio.} In this step, we further simplify the right-hand side of (\ref{KPZ_D1}).
\begin{proposition} \label{proposition_KPZ_D2}
Assume that $d\geq 3$ and $\beta<\beta_{L^{2}}$. Then, we have the following $L^{1}$-approximation:
\begin{align} \label{KPZ_D2}
    &T^{\frac{d-2}{4}} 
    \frac{\mathcal{Z}^{\beta}_{\infty}(x_{j}\sqrt{T};\xi(\cdot,\cdot+\delta_{j} T))
    -
    \mathcal{Z}^{\beta}_{t_{j} T}(x_{j}\sqrt{T};\xi(\cdot,\cdot+\delta_{j} T))}{\mathcal{Z}^{\beta}_{t_{j} T}(x_{j}\sqrt{T};\xi(\cdot,\cdot+\delta_{j} T))}\nonumber\\
    &\quad\quad\quad\quad\quad\quad\quad\quad\quad= T^{\frac{d-2}{4}}\int_{\mathbb{R}^{d}} dy G_{t_{j}}(x_{j} - y)
    (\mathcal{Z}^{\beta}_{\infty}(y\sqrt{T};\xi(\cdot,\cdot+t_{M}T))-1)+o(1).   
\end{align}
\end{proposition}
\begin{proof}
The $L^{1}$-error induced by (\ref{KPZ_D2}) vanishes by combining Proposition \ref{SHE_key_proposition}, Cauchy-Schwarz inequality, and the negative moment bound in Theorem \ref{properties_Z}.
\end{proof}
\paragraph{Step 3.} Combining Proposition \ref{proposition_KPZ_D1} and Proposition \ref{proposition_KPZ_D2}, the desired convergence \eqref{KPZ_D00} follows from \eqref{SHE_gaussian_limit_0}.

\section{Asymptotic decomposition of $\mathbf{F}^{\textup{PF};\beta}_{T}(x)$} \label{SHE_key_proposition_section}
The purpose of this section is to establish Proposition \ref{SHE_key_proposition}. By the translation-invariance of the white noise, we can assume that $x = 0$. In addition, by a change of variable, it is enough to consider the case of $t = 1$. As a result, to prove Proposition \ref{SHE_key_proposition}, it suffices to show that
\begin{equation} \label{goal_key_lemma_proof}
    \mathbb{E}\bigbkt{\bkt{\mathbf{F}^{\textup{PF};\beta}_{T}(0) - 
    \mathcal{Z}^{\beta}_{T}(0;\xi) \times T^{\frac{d-2}{4}} 
    \int_{\mathbb{R}^{d}} dx G_{1}(x)
    (\mathcal{Z}^{\beta}_{\infty}(x\sqrt{T};\xi(\cdot,\cdot+T))-1)
    }^{2}}
    = o(1),
\end{equation}
where $\mathbf{F}^{\textup{PF};\beta}_{T}(0)$ is given in (\ref{definition:PF_fluctuation}). We now prove \eqref{goal_key_lemma_proof} in the following steps.

\paragraph{Step 1.} Our first step toward \eqref{goal_key_lemma_proof} is to obtain a representation of $\mathbf{F}^{\textup{PF};\beta}_{T}(0)$.
\begin{lemma} \label{lemma:representation_different}
Assume that $d\geq 3$ and $\beta < \beta_{c}$. Recall that $\mathbf{F}^{\textup{PF};\beta}_{T}(0)$ is given in (\ref{definition:PF_fluctuation}).
Then, for every $T>0$, the following holds almost surely:
\begin{align} \label{representation_different}
    &\mathbf{F}^{\textup{PF};\beta}_{T}(0) = \int_{\mathbb{R}^{d}} dx G_{1}(x)\nonumber\\
    &\times \mathbf{E}_{0,0}^{T,\sqrt{T}x}\bigbkt{\exp\bkt{\beta \int_{0}^{T} ds \xi_{1}(B(s),s) - \frac{\beta^{2} R(0) T }{2}}}\times T^{\frac{d-2}{4}} (\mathcal{Z}^{\beta}_{\infty}(x\sqrt{T};\xi(\cdot,\cdot+T))-1),
\end{align}
where the noise $\xi_{1}(\cdot,\cdot)$ is defined in (\ref{definition:spatially_mollified_white_noise}), and $\mathbf{P}_{0,a}^{T,b}$ is defined by the law of the $d$-dimensional Brownian bridge starting from $(0,a)$ to $(T,b)$. 
\end{lemma}
\begin{proof}
The identity (\ref{representation_different}) immediately follows from the Markov property of Brownian motion and (\ref{limit_CDP}).
\end{proof}

\paragraph{Step 2.} To establish \eqref{goal_key_lemma_proof}, we now derive a representation of the left-hand side of \eqref{goal_key_lemma_proof}.

\begin{lemma}
Assume that $d\geq 3$ and $\beta < \beta_{c}$. Then, it holds that
\begin{align} \label{error_representation}
    \text{(L.H.S) of (\ref{goal_key_lemma_proof})} = \int_{\mathbb{R}^{2\times d}} dx_{1} dx_{2} G_{1}(x_{1}) G_{1}(x_{2}) H_{\beta;(0,T)}(x_{1},x_{2}) \times H_{\beta;(T,\infty)}(x_{1},x_{2}).
\end{align}
Here, the kernels $H_{\beta;(0,T)}(x_{1},x_{2})$ and $H_{\beta;(T,\infty)}(x_{1},x_{2})$ are given by
\begin{align}
    H_{\beta;(0,T)}(x_{1},x_{2}) &:= \mathbf{E}_{0,0,0}^{T,*,*}\bigbkt{\exp\bkt{\beta^{2}\int_{0}^{T}R(B_{2}(s)-B_{1}(s)) ds}}\nonumber\\
    &-\mathbf{E}_{0,0,0}^{T,\sqrt{T}x_{2},*}\bigbkt{\exp\bkt{\beta^{2}\int_{0}^{T}R(B_{2}(s)-B_{1}(s)) ds}}\nonumber\\
    &-\mathbf{E}_{0,0,0}^{T,*,\sqrt{T}x_{1}}\bigbkt{\exp\bkt{\beta^{2}\int_{0}^{T}R(B_{2}(s)-B_{1}(s)) ds}}\nonumber\\
    &+\mathbf{E}_{0,0,0}^{T,\sqrt{T}x_{2},\sqrt{T}x_{1}}\bigbkt{\exp\bkt{\beta^{2}\int_{0}^{T}R(B_{2}(s)-B_{1}(s)) ds}} \label{definition:kernel_H_0_T},\\
    H_{\beta;(T,\infty)}(x_{1},x_{2})&:= T^{\frac{d-2}{2}} \bkt{\mathbf{E}_{\sqrt{T}x_{2},\sqrt{T}x_{1}}\bigbkt{\exp\bkt{\beta^{2}\int_{0}^{\infty}R(B_{2}(s)-B_{1}(s)) ds}} -1}, \label{definition:kernel_H_T_infty}
\end{align}
where $\mathbf{P}_{0,0}^{T,*}$ denotes the law of $d$-dimensional Brownian motion $B$ starting at $x \in \mathbb{R}^{d}$ and ending at time $T$ with free endpoint, and $\mathbf{P}_{0,0,0}^{T,a,b}$ is the product measure of $\mathbf{P}_{0,0}^{T,a}$ and $\mathbf{P}_{0,0}^{T,b}$.
\end{lemma}
\begin{proof}
To prove (\ref{error_representation}), we apply Lemma \ref{lemma:representation_different} to obtain
\begin{align} \label{representation_different_1}
    &\mathbf{F}^{\textup{PF};\beta}_{T}(0) - 
    \mathcal{Z}^{\beta}_{T}(0;\xi) \times T^{\frac{d-2}{4}} 
    \int_{\mathbb{R}^{d}} dx G_{1}(x)
    (\mathcal{Z}^{\beta}_{\infty}(x\sqrt{T};\xi(\cdot,\cdot+T))-1)\nonumber\\
    &= \int_{\mathbb{R}^{d}} dx G_{1}(x)\bkt{\mathbf{E}_{0,0}^{T,\sqrt{T}x}\bigbkt{\exp\bkt{\beta \int_{0}^{T} ds \xi_{1}(B(s),s) - \frac{\beta^{2} R(0) T }{2}}} \nonumber\\
    &\quad\quad\quad-\mathbf{E}_{0}\bigbkt{\exp\bkt{\beta \int_{0}^{T} ds \xi_{1}(B(s),s) - \frac{\beta^{2} R(0) T }{2}}}  }\times \bkt{ T^{\frac{d-2}{4}} (\mathcal{Z}^{\beta}_{\infty}(x\sqrt{T};\xi(\cdot,\cdot+T))-1)}.
\end{align}
To evaluate the expectation on the left-hand side of \eqref{goal_key_lemma_proof}, we square the integral with respect to $dx$ on the right-hand side of \eqref{representation_different_1} by re-writing it as a double integral over $x_{1},x_{2}$, and then we take expectation with respect to $\mathbb{P}$. Note that the two bracketed factors on the right-hand side of \eqref{representation_different_1} are independent, since the first depends only on the noise $\sbkt{\xi(x,t)}_{(x,t) \in \mathbb{R}^{d} \times (0,T)}$ while the second depends only on the noise $\sbkt{\xi(x,t)}_{(x,t) \in \mathbb{R}^{d} \times (T,\infty)}$. Thus, the expectation with respect to $\mathbb{P}$ can be factorized, which yields \eqref{error_representation}. 
\end{proof}

\paragraph{Step 3.} To show that the right-hand side of (\ref{error_representation}) vanishes as $T\to\infty$, we investigate the asymptotics of the kernels $H_{\beta;(0,T)}$ and $H_{\beta;(T,\infty)}$, and establish uniform estimates for them. For the sake of simplicity, all the proofs of the following properties will be presented in Section \ref{section:proof_of_trivial_lemma} and Section \ref{section_exponential_functional}, respectively. 

To analyze the kernel function $H_{\beta;(T,\infty)}$, we re-express it as
\begin{align} \label{definition:kernel_H_T_infty_1}
    H_{\beta;(T,\infty)}(x_{1},x_{2}) = T^{\frac{d-2}{2}} \bkt{\mathbf{E}_{\sqrt{T} \frac{x_{2}-x_{1}}{\sqrt{2}}}\bigbkt{\exp\bkt{\beta^{2}\int_{0}^{\infty}R(\sqrt{2}B(s)) ds}} -1}.
\end{align}
\begin{lemma} \label{trivial_lemma}
Recall that $H_{\beta;(T,\infty)}$ is given in (\ref{definition:kernel_H_T_infty_1}).
If $d\geq 3$ and $\beta < \beta_{L^{2}}$, then the following holds:
\begin{enumerate} [label = (\roman*).]
    \item It holds that
    \begin{align} \label{trivial_uniform_bound}
    \sup_{T\geq 1,\; x_{1}\neq x_{2}}|x_{1}-x_{2}|^{d-2} H_{\beta;(T,\infty)}(x_{1},x_{2}) < \infty 
    \end{align}
    \item For each $x_{1} \neq x_{2}$, as $T\to\infty$, we have
    \begin{align}\label{trivial_limit}
        &H_{\beta;(T,\infty)}(x_{1},x_{2}) = \mathcal{G}^{0}(x_{2}-x_{1})  \int_{\mathbb{R}^{d}} dx \beta^{2} R(\sqrt{2}x) \mathbf{E}_{x}\bigbkt{\exp\bkt{\beta^{2}\int_{0}^{\infty}R(\sqrt{2}B(s)) ds}}  +o(1),
    \end{align}
    where $\mathcal{G}^{0}(z)$ is the $d$-dimensional Yukawa potential given by 
    \begin{align} \label{Yukawa}
        \mathcal{G}^{0}(z) = \int_{0}^{\infty} dt G_{t}(z) = \chi_{d}/|z|^{d-2}, \quad \chi_{d} > 0.
    \end{align}
\end{enumerate}
\end{lemma}

Evaluation of the limit of $H_{\beta;(0,T)}$ reduces to analyzing the expected Brownian-bridge exponential functional in \eqref{bridge_bound_formula}.
\begin{proposition} \label{rescaled_additive_functional}
Assume that $d\geq 3$ and $\beta < \beta_{L^{2}}$. Then the following holds:
\begin{enumerate} [label = (\roman*).]
    \item It holds that
    \begin{equation}\label{bridge_bound_formula}
        \sup_{T\geq 1,\; a,b\in \mathbb{R}^{d}} \mathbf{E}_{0,a}^{T,b}\bigbkt{\exp\bkt{\beta^{2}\int_{0}^{T} R(\sqrt{2}B(s)) ds}} < \infty
    \end{equation} 
    \item For each $z \neq 0$, as $T\to\infty$, we have
    \begin{align}\label{bridge_limit_formula}
        \mathbf{E}_{0,0}^{T,\sqrt{T}z}\bigbkt{\exp\bkt{\beta^{2}\int_{0}^{T} R(\sqrt{2}B(s))}}
        = \mathbf{E}_{0}\bigbkt{\exp\bkt{\beta^{2}\int_{0}^{\infty} R(\sqrt{2}B(s))}} + o(1).
    \end{align}
\end{enumerate}
\end{proposition}
\begin{remark}
We thank Professor Clément Cosco for pointing out that the uniform bound \eqref{bridge_bound_formula} was proved in \cite[(4.16)]{COSCO2022127} using the shear invariance of white noise.
But, both Lemma \ref{trivial_lemma} and Proposition \ref{rescaled_additive_functional} extend to a more general setting that does not assume the convolution condition $R = \phi*\phi$; see \eqref{definition_R_epsilon} for details. Instead, it suffices to impose the following decay and regularity on $R$: $R\in L^{1}(\mathbb{R}^{d};\mathbb{R}_{+})$ is bounded and symmetric-decreasing (see Section \ref{model_section} for the definition). Under this assumption and the condition
\begin{align*}
    \beta < \sup\set{\beta>0: \mathbf{E}_{0}\bigbkt{\exp\bkt{\beta^{2} \int_{0}^{\infty} R(\sqrt{2}B(u)) du }  } },
\end{align*}
which is equivalent to the condition $\beta < \beta_{L^{2}}$,
both Lemma \ref{trivial_lemma} and Proposition \ref{rescaled_additive_functional} hold. For example, $R(x) = \mathbf{1}_{\{|x|\leq 1\}}$ satisfies our assumptions but cannot be written as 
$\phi *\phi$ with any $\phi \in L^{1}(\mathbb{R}^{d})$, since $\phi*\phi$ is continuous. From the perspective of studying expected exponential functionals for Brownian bridges and Brownian motions, our proof may be of independent interest.
\end{remark}

\paragraph{Step 4.} We can now conclude Proposition \ref{SHE_key_proposition}.
\begin{proof} [Proof of Proposition \ref{SHE_key_proposition}.]
Observe that the first three terms on the right-hand side of (\ref{definition:kernel_H_0_T}) can be further decomposed using the following idea. For instance, the second term satisfies
\begin{align} \label{trick:two_brownianmotion_to_one_brownianmotion}
    &\mathbf{E}_{0,0,0}^{T,\sqrt{T}x_{2},*}\bigbkt{\exp\bkt{\beta^{2}\int_{0}^{T}R(B_{2}(s)-B_{1}(s)) ds}} = \int_{\mathbb{R}^{d}} dy_{1} G_{1}(y_{1}) \nonumber\\
    &\quad\quad\quad\quad\quad\quad\quad\quad\quad\quad\quad\quad\quad\quad\quad\times\mathbf{E}_{0,0}^{T,\sqrt{T}(x_{2}-y_{1})/\sqrt{2}}\bigbkt{\exp\bkt{\beta^{2}\int_{0}^{T}R(\sqrt{2}B(s)) ds}}.
\end{align}
Here, we have used that $\sbkt{(B_{2}(t)-B_{1}(t))/\sqrt{2}}_{t\geq 0}$ is a Brownian motion. Then, by (\ref{trivial_uniform_bound}) and (\ref{bridge_bound_formula}), the integrand of the integral on the right-hand side of (\ref{error_representation}) can be bounded by
\begin{align*}
    C G_{1}(x_{1})G_{1}(x_{2})\frac{1}{|x_{1}-x_{2}|^{d-2}} \in L^{1}(dx_{1} \otimes dx_{2}).
\end{align*}
Therefore, applying dominated convergence theorem, together with (\ref{trivial_limit}) and (\ref{bridge_limit_formula}), shows that 
the integral on the right-hand side of (\ref{error_representation}) vanishes as $T\to\infty$, which implies
(\ref{goal_key_lemma_proof}). Indeed, since the first three terms on the right-hand side of (\ref{definition:kernel_H_0_T}) can all be rewritten via the same trick as in \eqref{trick:two_brownianmotion_to_one_brownianmotion}, each term on the right-hand side of (\ref{definition:kernel_H_0_T}) has the same limit, namely, the right-hand side of \eqref{bridge_limit_formula}. It follows that
\begin{align*}
    H_{\beta;(0,T)}(x_{1},x_{2}) = o(1) \quad \text{as $T\to\infty$} \quad \forall x_{1}\neq x_{2}.
\end{align*}
\end{proof}

\section{Moment bounds for linearizations} \label{ratio_bound_secion}

The goal of this section is to prove Proposition~\ref{ratio_bound}, which serves as a key ingredient in the proofs of Proposition~\ref{proposition_KPZ_D1} and Proposition~\ref{proposition_KPZ_D2}.

\subsection{Moment bounds for the normalized fluctuation.}
Now we prove (\ref{I_0}) for $\nu = 2$. Let $\delta > 0$ be a parameter to be chosen later. 


\paragraph{Step 1.} We begin by bounding the numerator in $\mathbf{M}_{T;\beta,2^{+}}^{(2)}$, namely $\mathbf{F}^{\textup{PF};\beta}_{T}(0)$. Applying Itô’s formula yields a stochastic–integral representation of
$\mathbf{F}^{\textup{PF};\beta}_{T}(0)$. The estimate \eqref{bound_A_2_step1} then follows applying the Burkholder–Davis–Gundy inequality to this representation.
\begin{lemma} \label{lemma_p_delta}
Assume that $d\geq 3$ and $\beta < \beta_{L^{2}}$, and set $2p := (2+\delta)(1+\delta)$. Then it holds that
\begin{align} 
    &\mathbf{M}_{T;\beta,2+\delta}^{(2)} \leq C(\beta,\delta,d)  \bkt{\int_{1}^{\infty} \frac{d\tau}{\tau^{d/2}} \bignorm{(T\tau)^{d/2} \cdot\mathbf{E}_{0,0}\bigbkt{\prod_{j=1}^{2} \Phi_{T\tau}^{\beta}(B_{j};\xi)\cdot R(B_{2}(T\tau)-B_{1}(T\tau))}}_{L^{p}}}^{\frac{1}{2}},\label{bound_A_2_step1}\\
    &\Phi_{\tau}^{\beta}(B;\xi):=\exp\bkt{\beta \int_{0}^{\tau} ds \xi_{1}(B(s),s) - \frac{\beta^{2} R(0) \tau }{2}}. \nonumber
\end{align}
\end{lemma}
\begin{proof}
We start by removing the denominator in $\mathbf{M}_{T;\beta,2^{+}}^{(2)}$. Applying Hölder's inequality and the negative moment bound (\ref{negative_moment_bound}) gives
\begin{align} \label{bound_A_holder}
    \mathbf{M}_{T;\beta,2+\delta}^{(2)} \leq ||\mathcal{Z}_{T}^{\beta}(0;\xi)^{-1}||_{L^{2p/\delta}} \cdot ||\mathbf{F}^{\textup{PF};\beta}_{T}(0)||_{L^{2p}} \lesssim ||\mathbf{F}^{\textup{PF};\beta}_{T}(0)||_{L^{2p}}.
\end{align}
To control the right-hand side of (\ref{bound_A_holder}), we observe that for every $T' \geq T$, 
\begin{align} \label{martingale}
    &T^{\frac{d-2}{4}} (\mathcal{Z}^{\beta}_{T'}(0;\xi
    ) - \mathcal{Z}^{\beta}_{T}(0;\xi
    )) = T^{\frac{d-2}{4}}\beta \int_{T}^{T'} d\tau \int_{\mathbb{R}^{d}} dx \mathbf{E}_{0}\bigbkt{ \Phi_{\tau}^{\beta}(B;\xi)\phi(B(\tau)-x)} \xi(x,\tau),
\end{align}
where we have applied Itô's formula to the right-hand side of (\ref{partition_function_1}). In addition, we notice that the right-hand side of (\ref{martingale}) is a martingale indexed by $T' \geq T$ with quadratic variation 
\begin{align*}
    &\langle T^{\frac{d-2}{4}} (\mathcal{Z}^{\beta}_{\cdot}(0;\xi
    ) - \mathcal{Z}^{\beta}_{T}(0;\xi
    )) \rangle (T') = T^{\frac{d}{2}}\beta^{2} \int_{1}^{T'/T} d\tau \mathbf{E}_{0,0}\bigbkt{\prod_{j=1}^{2} \Phi_{T\tau}^{\beta}(B_{j};\xi)\cdot R(B_{2}(T\tau)-B_{1}(T\tau))},
\end{align*}
where we have used the change of variable $\tau \mapsto \tau/T$. Therefore, applying Burkholder-Davis-Gundy inequality \cite[Theorem 3.28]{MR917065} yields that
\begin{align*}
    &\text{(R.H.S) of (\ref{bound_A_holder})} \lesssim\mathbb{E}\bigbkt{\abs{\beta^{2} \int_{1}^{\infty} \frac{d\tau}{\tau^{d/2}} (T\tau)^{\frac{d}{2}} \cdot \mathbf{E}_{0,0}\bigbkt{\prod_{j=1}^{2} \Phi_{T\tau}^{\beta}(B_{j};\xi)\cdot R(B_{2}(T\tau)-B_{1}(T\tau))}}^{p}}^{\frac{1}{2p}}.
\end{align*}
In particular, by using Minkowski inequality, we conclude the estimate (\ref{bound_A_2_step1}).
\end{proof}

\paragraph{Step 2.} To bound $\mathbf{M}_{T;\beta,2+\delta}^{(2)}$, it is enough to verify the uniform estimate \eqref{bound_A_2_step2}, thanks to (\ref{bound_A_2_step1}). The idea is to control the right-hand side of \eqref{bound_A_2_step2} by using the expected Brownian-bridge exponential functional in \eqref{bridge_bound_formula}.
\begin{lemma} \label{lemma_brdige_bound}
Assume that $d\geq 3$ and $\beta < \beta_{L^{2}}$. Then, when $\delta$ is small enough, it holds that
\begin{align} \label{bound_A_2_step2}
    \sup_{T \geq 1} \bignorm{T^{d/2} \cdot\mathbf{E}_{0,0}\bigbkt{\prod_{j=1}^{2} \Phi_{T}^{\beta}(B_{j};\xi)\cdot R(B_{2}(T)-B_{1}(T))}}_{L^{p}} < \infty.
\end{align} 
\end{lemma}
\begin{proof}
To show (\ref{bound_A_2_step2}), we first re-express the random variable in (\ref{bound_A_2_step2}) as follows:
\begin{align*}
    &T^{d/2} \mathbf{E}_{0,0}\bigbkt{\prod_{j=1}^{2} \Phi_{T}^{\beta}(B_{j};\xi)\cdot R(B_{2}(T)-B_{1}(T))} \\
    &= T^{d/2} \int_{\mathbb{R}^{2\times d}} dx_{1} dx_{2} G_{T}(x_{1}) G_{T}(x_{2}) R(x_{2}-x_{1}) \mathbf{E}_{0,0}^{T,x_{1}}[\Phi_{T}^{\beta}(B_{1};\xi)] \mathbf{E}_{0,0}^{T,x_{2}}[\Phi_{T}^{\beta}(B_{2};\xi)]\\
    &= \int_{\mathbb{R}^{2\times d}} dz dz' \sbkt{T^{\frac{d}{2}}\cdot G_{T}(z)} G_{T}(z') R(\sqrt{2}z) \mathbf{E}_{0,0}^{T,(z'-z)/\sqrt{2}}[\Phi_{T}^{\beta}(B_{1};\xi)] \mathbf{E}_{0,0}^{T,(z'+z)/\sqrt{2}}[\Phi_{T}^{\beta}(B_{2};\xi)],
\end{align*}
where we have used the following change of variable and the identity in the last equality:
\begin{align} \label{change_of_variable:two_body}
    (x_{1},x_{2}) \mapsto (z,z') := \bkt{\frac{x_{2}-x_{1}}{\sqrt{2}},\frac{x_{2}+x_{1}}{\sqrt{2}}}, \quad G_{T}(x_{1})G_{T}(x_{2}) = G_{T}(z) G_{T}(z').
\end{align}
As a result, one has
\begin{align} \label{bound_A_2_step2_1}
    &\bignorm{T^{d/2} \cdot\mathbf{E}_{0,0}\bigbkt{\prod_{j=1}^{2} \Phi_{T}^{\beta}(B_{j};\xi)\cdot R(B_{2}(T)-B_{1}(T))}}_{L^{p}}\nonumber\\
    &\lesssim \int_{\mathbb{R}^{2 \times d}} dzdz' R(\sqrt{2}z) G_{T}(z') \bignorm{ \mathbf{E}_{0,0}^{T,(z'-z)/\sqrt{2}}\bigbkt{\Phi_{T}^{\beta}(B;\xi)}}_{L^{2p}} \bignorm{ \mathbf{E}_{0,0}^{T,(z'+z)/\sqrt{2}}\bigbkt{\Phi_{T}^{\beta}(B;\xi)}}_{L^{2p}},
\end{align}
where we have applied Minkowski’s inequality, Hölder's inequality, and the estimate 
\begin{align*}
    \sup_{z \in \mathbb{R}^{d}} T^{d/2} G_{T}(z) \lesssim 1.
\end{align*}

We now derive a spatially uniform estimate for the two $L^{2p}$-norms on the right-hand side of (\ref{bound_A_2_step2_1}). The key point is to apply the hypercontractivity \cite[Theorem 5.1]{GHS}. To apply this, we observe that  
\begin{align} \label{chaos_expansion}
    &\mathbf{E}_{0,0}^{T,z}\bigbkt{\Phi_{T}^{\beta}(B;\xi)} = 1+\sum_{k=1}^{\infty} \beta^{k} \int_{\mathbb{R}^{d\times k}}
    \int_{0 < t_{1}<\ldots<t_{k} < T} 
    \theta_{T;k}(t_{1},q_{1},\ldots,t_{k},q_{k})
    \prod_{j=1}^{k} dt_{j} dq_{j} 
    \xi(q_{j},t_{j}),\\
    &\theta_{T;k}(t_{1},q_{1},\ldots,t_{k},q_{k}) := G_{T}(z)^{-1}\int_{\mathbb{R}^{d \times k}, \; z_{0} = 0}
    \bkt{
    \prod_{j=1}^{k} dz_{j} G_{t_{j} - t_{j-1}}(z_{j} - z_{j-1}) \phi(q_{j} - z_{j}) }
    \cdot G_{T-t_{k}}(z-z_{k}). \nonumber
\end{align}
Indeed, \eqref{chaos_expansion} can be obtained by applying the following consequence of Itô's formula iteratively:
\begin{align*}
    \Phi_{T}^{\beta}(B;\xi) = 1 + \beta\int_{0}^{T} dt \int_{\mathbb{R}^{d}} dx \Phi_{t}^{\beta}(B;\xi) \phi(B(t)-x) \xi(x,t).
\end{align*}
Therefore, applying \cite[Theorem 5.1]{GHS} with $A = (2p-1)^{-1/2} I$, we conclude 
\begin{align} \label{bound_A_2_step2_2}
    \bignorm{ \mathbf{E}_{0,0}^{T,z}\bigbkt{\Phi_{T}^{\beta}(B;\xi)}}_{L^{2p}} \leq \bignorm{ \mathbf{E}_{0,0}^{T,z}\bigbkt{\Phi_{T}^{\beta_{p}}(B;\xi)}}_{L^{2}}, \quad \text{where } \beta_{p} := \beta \cdot (2p-1)^{1/2}.
\end{align}
Here $I$ is the identity map on the Gaussian Hilbert space induced by the white noise $\xi$. Hence, we have
\begin{align} \label{bound_A_2_step2_3}
    \bignorm{ \mathbf{E}_{0,0}^{T,z}\bigbkt{\Phi_{T}^{\beta_{p}}(B;\xi)}}_{L^{2}}
    = \mathbf{E}_{0,0}^{T,0}\bigbkt{\exp\bkt{\beta_{p}^{2}\int_{0}^{T} R(\sqrt{2}B(s)) ds}}^{\frac{1}{2}} < \infty.
\end{align}
Indeed, since $\beta_{p} < \beta_{L^{2}}$ (achieved by choosing $\delta$ sufficiently small), the inequality in \eqref{bound_A_2_step2_3} follows from (\ref{bridge_bound_formula}). Therefore, combining (\ref{bound_A_2_step2_1}), (\ref{bound_A_2_step2_2}), and (\ref{bound_A_2_step2_3}) yields the required bound 
(\ref{bound_A_2_step2}).
    
\end{proof}

\paragraph{Step 3.} Combining Lemma \ref{lemma_p_delta} and Lemma \ref{lemma_brdige_bound} concludes (\ref{I_0}) for $\nu = 2$. The proof is complete.

\subsection{Moment bounds for the fluctuation of the free energy.}
Let us now establish (\ref{I_0}) for $\nu = 1$. Fix a parameter $\delta > 0$ to be chosen later. Our approach is to control $\mathbf{M}_{T;\beta,2+\delta}^{(1)}$ by using $\mathbf{M}_{T;\beta,2^{+}}^{(2)}$. We begin by noting that the inequality
\begin{align*}
    |\log(1+x)| \leq c|x| \quad \forall |x| \leq \frac{1}{2}, \quad c>0
\end{align*}
implies the following decomposition of $\mathbf{M}_{T;\beta,2+\delta}^{(1)}$:
\begin{align} \label{bound_A_1_1}
    \mathbf{M}_{T;\beta,2+\delta}^{(1)}
    \lesssim \mathbb{E}\bigbkt{\abs{
    T^{\frac{d-2}{4}} \sbkt{\log \mathcal{Z}^{\beta}_{\infty}(0;\xi) - \log \mathcal{Z}^{\beta}_{T}(0;\xi)}}^{2+\delta}; 
    \abs{
    \frac{\mathcal{Z}^{\beta}_{\infty}(0;\xi) - \mathcal{Z}^{\beta}_{T}(0;\xi)}{\mathcal{Z}^{\beta}_{T}(0;\xi)}
    }\geq \frac{1}{2}
    }^{\frac{1}{2+\delta}} + \mathbf{M}_{T;\beta,2+\delta}^{(2)}.
\end{align}
Next, by Hölder's inequality and Markov’s inequality, the first term on the right-hand side of (\ref{bound_A_1_1}) can be bounded as
\begin{align*} 
    &\text{(L.H.S) of (\ref{bound_A_1_1})} \\
    &\leq C(\beta)
    T^{\frac{d-2}{4}} \mathbb{P}\bkt{
    \abs{
    \frac{\mathcal{Z}_{\infty}^{\beta}(0;\xi) - \mathcal{Z}^{\beta}_{T}(0;\xi)}{\mathcal{Z}^{\beta}_{T}(0;\xi)}
    }\geq \frac{1}{2}
    }^{\frac{1}{(1+\delta)(2+\delta)}}\lesssim C(\beta) 
    \bignorm{T^{\frac{d-2}{4}}
    \frac{\mathcal{Z}_{\infty}^{\beta}(0;\xi) -\mathcal{Z}^{\beta}_{T}(0;\xi)}{\mathcal{Z}^{\beta}_{T}(0;\xi)}}_{L^{(2+\delta)(1+\delta)}}\\
    &= C(\beta) \mathbf{M}_{T;\beta,(2+\delta)(1+\delta)}^{(2)}, \quad C(\beta):=\sup_{1\leq T\leq \infty}\snorm{\log \mathcal{Z}_{T}^{\beta}(0;\xi)}_{L^{\frac{(2+\delta)(1+\delta)}{\delta}}}.
\end{align*}
Therefore, by choosing $\delta>0$ sufficiently small so that $\mathbf{M}_{T;\beta,(2+\delta)(1+\delta)}^{(2)} < \infty$, we obtain an uniform bound for $\mathbf{M}_{T;\beta,2+\delta}^{(1)}$. It remains to check that $C(\beta) < \infty$ for every $\beta < \beta_{L^{2}}$. Indeed, if $m \in \mathbb{N}$, then 
\begin{align}
    \sup_{T\geq 1}\mathbb{E}
    \bigbkt{
    \frac{\sabs{2\log \mathcal{Z}_{T}^{\beta}(0;\xi)}^{m}}{m!};\mathcal{Z}_{T}^{\beta}(0;\xi) > 1}
    &\leq \sup_{T\geq 1}\mathbb{E}[\exp(2\log \mathcal{Z}_{T}^{\beta}(0;\xi))]
    \leq \mathbb{E}[\mathcal{Z}_{\infty}^{\beta}(0;\xi)^{2}] < \infty, \nonumber\\
    \sup_{T\geq 1}\mathbb{E}[\sabs{\log \mathcal{Z}_{T}^{\beta}(0;\xi)}^{m}; 0 < \mathcal{Z}_{T}^{\beta}(0;\xi) < 1] &\leq \sup_{T\geq 1}\mathbb{E}[\mathcal{Z}_{T}^{\beta}(0;\xi)^{-m}] <\infty, \label{estimate_0_z_1}
\end{align}
where the second estimate follows from $-\log(x) \leq 1/x$ for all $0 < x < 1$ and from (\ref{negative_moment_bound}).

\section{Large-scale limits for expected Brownian functionals} \label{section:proof_of_trivial_lemma}
Let us now verify Lemma \ref{trivial_lemma}.
We begin by re-expressing $H_{\beta;(T,\infty)}(x_{1},x_{2})$ from \eqref{definition:kernel_H_T_infty} as 
\begin{align}
    H_{\beta;(T,\infty)}(x_{1},x_{2})
    = T^{\frac{d-2}{2}}\bkt{\mathfrak{h}_{\beta}\bkt{\sqrt{T}\cdot \frac{x_{2}-x_{1}}{\sqrt{2}}}-1}, \quad \mathfrak{h}_{\beta}(z) := \mathbf{E}_{z}\bigbkt{\exp\bkt{\beta^{2}\int_{0}^{\infty}R(\sqrt{2}B(s)) ds}}. \label{identity_H_and_small_h}
\end{align}
Before proving Lemma \ref{trivial_lemma}, we first analyze the properties of the function $\mathfrak{h}_{\beta}(z)$.
\begin{lemma} \label{lemma_h_beta_property_1}
Assume that $d\geq 3$ and $\beta<\beta_{L^{2}}$. Then the following holds:
\begin{enumerate} [label = (\roman*).]
    \item It holds that $\sup_{z\in \mathbb{R}^{d}} \mathfrak{h}_{\beta}(z) \leq \mathfrak{h}_{\beta}(0) < \infty$.
    \item $\mathfrak{h}_{\beta}(\cdot)$ satisfies
    \begin{align}
        \mathfrak{h}_{\beta}(\sqrt{T}z) = 1+\int_{\mathbb{R}^{d}} dx\mathcal{G}^{0}(\sqrt{T}z - x) \beta^{2} R(\sqrt{2}x) \mathfrak{h}_{\beta}(x), \quad \mathcal{G}^{0}(z) \text{is defined in \eqref{Yukawa}.}\label{integral_equation_h_beta}
    \end{align}
\end{enumerate}
\end{lemma}
\begin{proof}
Fix $z\neq 0$. Recalling from Section \ref{model_section} that $\phi$ is spherically symmetric, we know that
$R(x) = \phi*\phi(x)$ is spherically symmetric, which implies $\mathfrak{h}_{\beta}(z)$ is spherically symmetric. In particular, by considering the first hitting time $\aleph_{|z|}$ of $\partial B(0,|z|)$ by Brownian motion, one has
\begin{align*}
    \mathbf{E}_{0}\bigbkt{\exp\bkt{\beta^{2}\int_{0}^{\infty}R(\sqrt{2}B(s)) ds}} 
    &=\mathbf{E}_{0}\bigbkt{\exp\bkt{\beta^{2}\int_{0}^{\aleph_{|z|}}R(\sqrt{2}B(s)) ds} \mathfrak{h}_{\beta}(B(\tau))}\\
    &=\mathbf{E}_{0}\bigbkt{\exp\bkt{\beta^{2}\int_{0}^{\aleph_{|z|}}R(\sqrt{2}B(s)) ds} } \mathfrak{h}_{\beta}(z) \geq \mathfrak{h}_{\beta}(z),
\end{align*}
which gives the first inequality in part (i) of Lemma \ref{lemma_h_beta_property_1}. The second inequality in part (i) follows from $\beta < \beta_{L^{2}}$. 

Let us now verify the integral equation (\ref{integral_equation_h_beta}). By using fundamental theorem of calculus, one has
\begin{align*} 
    \exp\bkt{\beta^{2}\int_{0}^{\infty} R(\sqrt{2}B(s)) dt} = 1+ \int_{0}^{\infty} dt \beta ^{2}R(\sqrt{2}B(t)) \exp\bkt{\beta^{2}\int_{t}^{\infty} R(\sqrt{2}B(u)) du}.
\end{align*}
As a result, applying the Markov property of Brownian
motion shows that
\begin{align*}
    \mathfrak{h}_{\beta}(\sqrt{T} z) &= 1 + \int_{\mathbb{R}^{d}} dx \mathcal{G}^{0}(\sqrt{T}z-x) \beta^{2} R(\sqrt{2}x) \mathfrak{h}_{\beta}(x).
\end{align*}

\end{proof}

Now, we are ready to conclude Lemma \ref{trivial_lemma}.  The idea is to analyze the integral equation \eqref{integral_equation_h_beta}.
\begin{proof} [Proof of Lemma \ref{trivial_lemma}]
We start with showing (\ref{trivial_limit}). Fix $z\neq 0$, and split the integral in \eqref{integral_equation_h_beta} as
\begin{align*}
    &T^{\frac{d-2}{2}} (\mathfrak{h}_{\beta}(\sqrt{T}z) - 1)
    = \int_{\mathbb{R}^{d}} dx T^{\frac{d-2}{2}}\mathcal{G}^{0}(\sqrt{T}z-x) \beta^{2} R(\sqrt{2}x) \mathfrak{h}_{\beta}(x) \mathbf{1}_{\{|x-\sqrt{T}z|\geq \frac{1}{2}|z| \sqrt{T}\}}\\
    &+ \int_{\mathbb{R}^{d}} dx T^{\frac{d-2}{2}}\mathcal{G}^{0}(\sqrt{T}z-x) \beta^{2} R(\sqrt{2}x) \mathfrak{h}_{\beta}(x) \mathbf{1}_{\{|x-\sqrt{T}z|\leq \frac{1}{2}|z| \sqrt{T}\}}
    =: \mathbf{J}_{T}^{\geq} + \mathbf{J}_{T}^{\leq}.
\end{align*}
We first observe that applying part (i) of Lemma \ref{lemma_h_beta_property_1} and a change of variable yields
\begin{align*}
    \mathbf{J}_{T}^{\leq} \lesssim \int_{\mathbb{R}^{d}} dx \frac{1}{|z-x|^{d-2}} R_{\varepsilon}(\sqrt{2}x) \mathbf{1}_{\{|x-z|\leq \frac{1}{2}|z|\}} = o(1) \quad \text{as} \quad T \to \infty, \quad T = \varepsilon^{-2}.
\end{align*}
Here we have used that $R \in L^{1}(\mathbb{R}^{d};\mathbb{R}_{+})$. Now, applying dominated convergence theorem to $\mathbf{J}_{T}^{\geq}$ yields
\begin{align} \label{convergent_rate_small_h}
    T^{\frac{d-2}{2}} (\mathfrak{h}_{\beta}(\sqrt{T}z) - 1) 
    &= \mathcal{G}^{0}(z) \cdot \int_{\mathbb{R}^{d}} dx \beta^{2} R(\sqrt{2}x) \mathfrak{h}_{\beta}(x)+o(1) \quad \text{as } T\to\infty,
\end{align}
where we have used that for each $x \in \mathbb{R}^{d}$, $\mathbf{1}_{\{|x-\sqrt{T}z|\geq \frac{1}{2}|z| \sqrt{T}\}}$ tends to $1$ as $T\to \infty$. Here, to bound the integrand of the integral on the right-hand side of (\ref{integral_equation_h_beta}), we use part (i) of Lemma \ref{lemma_h_beta_property_1} to bound $\mathfrak{h}_{\beta}(x)$, and the estimate 
\begin{align*}
    \sup_{T\geq 1, \; |x-\sqrt{T}z| \geq \frac{1}{2}|z|\sqrt{T}} \frac{T^{(d-2)/2}}{|\sqrt{T}z-x|^{d-2}} R(\sqrt{2}x) \lesssim \frac{1}{|z|^{d-2}} R(\sqrt{2}x) \in L^{1}(\mathbb{R}^{d})
\end{align*}
to bound $T^{(d-2)/2}\mathcal{G}^{0}(\sqrt{T}z-x) R(\sqrt{2}x)$. Therefore, combining (\ref{convergent_rate_small_h}) with (\ref{identity_H_and_small_h}) implies (\ref{trivial_limit}).

Now we turn to (\ref{trivial_uniform_bound}). Observe that
\begin{align}\label{estimate_small_h}
    &\sup_{T\geq 1, \; z\neq 0}|\sqrt{T}z|^{d-2} (\mathfrak{h}_{\beta}(\sqrt{T}z) - 1) 
    \lesssim \sup_{T\geq 1, \; z\neq 0} \int_{\mathbb{R}^{d}} dx\frac{(|\sqrt{T}z|-|x|+|x|)^{d-2}}{|\sqrt{T} z-x|^{d-2}} \beta^{2} R(\sqrt{2}x) \mathfrak{h}_{\beta}(x)\nonumber\\
    &\lesssim \int_{\mathbb{R}^{d}} dx\beta^{2} R(\sqrt{2}x) \mathfrak{h}_{\beta}(x) + \sup_{T\geq 1, \; z\neq 0}\int_{\mathbb{R}^{d}} dx\frac{|x|^{d-2}}{|\sqrt{T} z-x|^{d-2}} \beta^{2} R(\sqrt{2}x) \mathfrak{h}_{\beta}(x) < \infty,
\end{align}
where we have used $(|\sqrt{T}z|-|x|+|x|)^{d-2} \lesssim |\sqrt{T}z-x|^{d-2} + |x|^{d-2}$ in the second inequality in (\ref{estimate_small_h}), and we have applied part (i) of Lemma \ref{lemma_h_beta_property_1} and the following estimate to conclude the last inequality in (\ref{estimate_small_h}):
\begin{align*}
    \sup_{z\in \mathbb{R}^{d}} \int_{\mathbb{R}^{d}} dx \frac{|x|^{d-2}}{|z-x|^{d-2}} R(\sqrt{2}x) < \infty.
\end{align*}
Here we have used that $R \in L^{1}(\mathbb{R}^{d})$. Consequently, combining \eqref{estimate_small_h} with (\ref{identity_H_and_small_h}) yields (\ref{trivial_uniform_bound}). 
\end{proof}

\section{Expected Brownian-bridge exponential functionals} \label{section_exponential_functional}
The goal of this section is to prove Proposition~\ref{rescaled_additive_functional}. For each $a,b \in \mathbb{R}^{d}$ and any $T\geq 1$, we set
\begin{align} \label{CabT}
    \mathbf{A}_{\beta}(a,b,T)
    :=\mathbf{E}_{0,a}^{T,b}\bigbkt{\exp\bkt{\beta^{2}\int_{0}^{T} R(\sqrt{2}B(s)) ds}}.
\end{align}

\subsection{Uniform estimates.} \label{section_uniform_bound_rescaled}
This subsection is dedicated to the verification of (\ref{bridge_bound_formula}). 

\paragraph{Step 1.} We start with 
a simplification indicating that it is enough to consider the case that $a = b = 0$.
\begin{lemma} \label{lemma_a_b_r}
Let $\mathbf{A}_{\beta}(a,b,T)$ be defined in (\ref{CabT}). Then 
\begin{align}
    \sup_{a,b \in \mathbb{R}^{d}, \; T\geq 1}\mathbf{A}_{\beta}(a,b,T) &\leq 
    \sup_{T\geq 1}\mathbf{A}_{\beta}(0,0,T).\label{a_b_r}
\end{align}
\end{lemma}
\begin{proof}
Fix $a,b \in \mathbb{R}^{d}$ and $T\geq 1$. Applying the representation of the Brownian bridge \cite[the solution of (6.23)]{MR917065} yields
\begin{align} \label{functional_bridge_form}
    \mathbf{A}_{\beta}(a,b,T) = \mathbf{E}_{0}\bigbkt{\exp\bkt{\beta^{2} \int_{0}^{T} R\bkt{\sqrt{2}\bkt{a+B(u)-\frac{u}{T}(B(T)-(b-a)) } } du }}.
\end{align}
In addition, we decompose the right-hand side of \eqref{functional_bridge_form} as 
\begin{align}\label{functional_bridge_form_series_representation}
    &\mathbf{A}_{\beta}(a,b,T) = 1+\sum_{k=1}^{\infty} \beta^{2k} \int_{0<t_{1}<\ldots<t_{k}<T} \prod_{j=1}^{k} dt_{j} \int_{\mathbb{R}^{(k+1) \times d}, \; z_{0} = 0}  \prod_{j=1}^{k+1} dz_{j} \nonumber\\
    &\times \bkt{\prod_{j=1}^{k} G_{t_{j}-t_{j-1}}(z_{j}-z_{j-1}) R\bkt{\sqrt{2} \bkt{z_{j} - \frac{t_{j}}{T}z_{k+1}  + \bkt{a+\frac{t_{j}}{T}(b-a)}  } } } G_{T-t_{k}}(z_{k+1}-z_{k}).
\end{align}
Here we have used the following consequence of fundamental theorem of calculus:
\begin{align*}
    &\exp\bkt{\beta^{2} \int_{0}^{\ell} R\bkt{\sqrt{2}\bkt{a+B(u)-\frac{u}{T}(B(T)-(b-a)) } } du }\\
    &= 1+\int_{0}^{\ell} dt \beta^{2}  R\bkt{\sqrt{2}\bkt{a+B(t)-\frac{t}{T}(B(T)-(b-a)) } } \\
    &\quad\quad\quad\quad\quad\quad\times \exp\bkt{\beta^{2} \int_{0}^{t} R\bkt{\sqrt{2}\bkt{a+B(u)-\frac{u}{T}(B(T)-(b-a)) } } du }, \quad \ell \geq 0.
\end{align*}
Therefore, by applying general rearrangement inequality \cite[Theorem 3.8]{lieb2001analysis} to the spatial integral on the right-hand side of \eqref{functional_bridge_form_series_representation}, we obtain
\begin{align*}
    &\mathbf{A}_{\beta}(a,b,T) \leq 1+\sum_{k=1}^{\infty} \beta^{2k} \int_{0<t_{1}<\ldots<t_{k}<T} \prod_{j=1}^{k} dt_{j} \int_{\mathbb{R}^{(k+1) \times d}, \; z_{0} = 0}  \prod_{j=1}^{k+1} dz_{j} \\
    &\times \bkt{\prod_{j=1}^{k} G_{t_{j}-t_{j-1}}(z_{j}-z_{j-1}) R\bkt{\sqrt{2} \bkt{z_{j} -\frac{t_{j}}{T}z_{k+1}   } }} G_{T-t_{k}}(z_{k+1}-z_{k}) = \mathbf{A}_{\beta}(0,0,T),
\end{align*}
where we have used that $R$ is symmetric-decreasing. The proof is complete.
\end{proof}

\paragraph{Step 2.} Due to Lemma \ref{lemma_a_b_r}, we now control $\mathbf{A}_{\beta}(0,0,T)$, where $\mathbf{A}_{\beta}(a,b,T)$ is given in (\ref{CabT}). 
\begin{lemma} \label{lemma:a_b_less_r}
Suppose that $d\geq 3$ and $\beta < \beta_{L^{2}}$. Then it holds that
\begin{align} \label{a_b_less_r}
    \sup_{T\geq 1}\mathbf{A}_{\beta}(0,0,T) < \infty.
\end{align}    
\end{lemma}
Before proceeding with the proof, let us first explain the main idea behind (\ref{a_b_less_r}), which relies on the decomposition
\begin{alignat}{2}
    \mathbf{A}_{\beta}(0,0,T) &= 1 + \sum_{k=1}^{\infty} \mathbf{A}_{\beta;k}(0,0,T), \label{decomposition_path_integral}\\
    \mathbf{A}_{\beta;k}(0,0,T) &:= \beta^{2k} G_{T}(0)^{-1} \int_{\mathbb{R}^{k \times d},\; z_{0} = z_{k+1} = 0} \prod_{j=1}^{k} dz_{j} \int_{0<\sum_{j=1}^{k} u_{j}<T, \; u_{k+1} = T-\sum_{j=1}^{k} u_{j}} \prod_{j=1}^{k} du_{j}\nonumber\\
    &\times \prod_{j=1}^{k} G_{u_{j}}(z_{j}-z_{j-1}) R(\sqrt{2}z_{j}) \times G_{u_{k+1}}(z_{k+1}-z_{k}). \label{definition_A_k}
\end{alignat}
The decomposition (\ref{decomposition_path_integral}) is obtained by iteratively applying the identity
\begin{align}  \label{FTC}
    \exp\bkt{\beta^{2}\int_{0}^{T} R(\sqrt{2}B(t)) dt} = 1+ \int_{0}^{T} dt \beta ^{2}R(\sqrt{2}B(t)) \exp\bkt{\beta^{2}\int_{t}^{T} R(\sqrt{2}B(u)) du}.
\end{align}
together with a change of variables in the time variables. The identity \eqref{FTC} follows directly from the fundamental theorem of calculus. Let us now explain our approach to estimate $\mathbf{A}_{\beta;k}(0,0,T)$. Notice that 
\begin{align*}
    G_{T}(0)^{-1} \lesssim T^{\frac{d}{2}}.
\end{align*}
Hence, how to cancel this $T^{\frac{d}{2}}$ is the key point to prove the estimate (\ref{a_b_less_r}). The key observation is that $\sum_{j=1}^{k} u_{j} = T$ on the right-hand side of (\ref{definition_A_k}), and hence, there exists a period $u_{j} > T/(k+1)$. To this end, one can cancel $T^{\frac{d}{2}}$ by using the estimate
\begin{align} \label{estimate:key_point}
    G_{u_{j}}(z_{j}-z_{j-1}) \lesssim (k+1)^{\frac{d}{2}} T^{-\frac{d}{2}}.
\end{align}

Now we are ready to conclude (\ref{a_b_less_r}). 

\begin{proof} [Proof of Lemma \ref{lemma:a_b_less_r}]
Fix $k\geq 1$. According to the aforementioned idea, we have
\begin{align} \label{first_decomposition}
    \mathbf{A}_{\beta;k}(0,0,T) \leq \sum_{i=1}^{k+1} \mathbf{A}_{\beta;k,i}(0,0,T),
\end{align}
where $\mathbf{A}_{\beta;k,i}(0,0,T)$ is defined as the right-hand side of (\ref{definition_A_k}) with an additional integrand $\mathbf{1}_{\{u_{i} > T/(k+1)\}}$. Fix $1\leq i\leq k+1$. We now bound $\mathbf{A}_{\beta;k,i}(0,0,T)$ as follows.  By (\ref{estimate:key_point}), the term $T^{\frac{d}{2}}G_{u_{i}}(z_{i}-z_{i-1})$ in $\mathbf{A}_{\beta;k,i}(0,0,T)$ can be bounded by $C_{1} \cdot (k+1)^{\frac{d}{2}}$. Moreover, the factor $(k+1)^{\frac{d}{2}}$ can be absorbed by slightly increasing $\beta$, namely,
\begin{align*}
    (1+m)^{\frac{d}{2}}\cdot \beta^{2m} \leq C_{2} \beta_{*}^{2m} \quad \forall m\geq 1,
\end{align*}
where $\beta_{*} \in (\beta, \beta_{L^{2}})$. Next, we perform a change of variables in  $\mathbf{A}_{\beta;k,i}(0,0,T)$ for $i\neq k+1$. Originally, $u_{i}$ is a variable, while $u_{k+1}$ is determined by $u_{k+1} := T- \sum_{j=1}^{k} u_{j}$. In our change of variable, we instead set $u_{k+1}$ as a variable and define $u_{i}$ by the function $u_{i} := T-\sum_{1\leq j\leq k+1, \; j\neq i} u_{j}$. As a result, we conclude
\begin{alignat}{2}
    &\mathbf{A}_{\beta;k,i}(0,0,T) \leq (C_{1}C_{2}) \mathfrak{h}_{\beta_{*},i-1}(0) \cdot \mathfrak{h}_{\beta_{*};k-i+1}&&(0) \quad \forall 1\leq i\leq k+1, \label{second_decomposition}\\
    &\mathfrak{h}_{\beta_{*},0}(z) := 1, \quad \mathfrak{h}_{\beta_{*},m}(z) := \beta_{*}^{2m}\int_{\mathbb{R}^{m \times d},\; z_{0} = z} &&\prod_{j=1}^{m} dz_{j} \int_{\mathbb{R}_{+}^{m}} \prod_{j=1}^{m} du_{j} \nonumber\\
    & &&\times \prod_{j=1}^{m} G_{u_{j}}(z_{j}-z_{j-1}) R(\sqrt{2}z_{j}) \quad \forall m\geq 1. \label{definition_h_k}
\end{alignat}
Therefore, by combining (\ref{decomposition_path_integral}), (\ref{first_decomposition}), and (\ref{second_decomposition}), we deduce
\begin{align} \label{inequality_A_hh}
    \mathbf{A}_{\beta}(0,0,T) \leq 1+ (C_{1}C_{2})\mathfrak{h}_{\beta_{*}}(0) \cdot \mathfrak{h}_{\beta_{*}}(0) < \infty,
\end{align}
where $\mathfrak{h}_{\beta_{*}}(z)$ is defined in \eqref{identity_H_and_small_h}. Here we used the fact that
\begin{align} \label{sum_h}
    \mathfrak{h}_{\beta_{*}}(z) = \sum_{k=0}^{\infty} \mathfrak{h}_{\beta_{*};k}(z),
\end{align}
which can be proved analogously to \eqref{decomposition_path_integral} using the fundamental theorem of calculus. The final inequality in \eqref{inequality_A_hh} follows from part (i) of Lemma \ref{lemma_h_beta_property_1}.
\end{proof}

\subsection{Large-scale asymptotics.} \label{section_asymptotic_rescaled}

Let us now verify (\ref{bridge_limit_formula}). Recall that $\mathbf{A}_{\beta}(a,b,T)$ is defined in (\ref{CabT}). Before proceeding with the proof, we begin by explaining the proof idea. Fix $z\neq 0$, and set $\Lambda_{T} = o(T)$ such that $\Lambda_{T} \to \infty$ as $T\to\infty$. Now we decompose the right-hand side of (\ref{decomposition_path_integral}) with $a = 0$ and $b = \sqrt{T}z$ as follows:
\begin{alignat}{2}
    & &&\mathbf{A}_{\beta}(0,\sqrt{T}z,T) =  \bkt{1+\sum_{k=1}^{\infty} \mathbf{A}^{(1)}_{\beta;k}(0,\sqrt{T}z,T)} + \bkt{\sum_{k=1}^{\infty} \mathbf{A}^{(2)}_{\beta;k}(0,\sqrt{T}z,T)}, \label{decomposition_Lambda}\\
    & &&\quad\mathbf{A}^{(1)}_{\beta;k}(a,b,T) := \beta^{2k} G_{T}(a-b)^{-1} \int_{\mathbb{R}^{k \times d},\; z_{0} = a, \; z_{k+1} = b} \prod_{j=1}^{k} dz_{j} \int_{0 = t_{0} < t_{1} < \ldots < t_{k} < \Lambda_{T}} \prod_{j=1}^{k} dt_{j}\nonumber\\
    & &&\times \prod_{j=1}^{k} G_{t_{j}-t_{j-1}}(z_{j}-z_{j-1}) R(\sqrt{2}z_{j}) \times G_{T-t_{k}}(z_{k+1}-z_{k}), \label{definition_A_k_(1)}\\
    & &&\quad\mathbf{A}^{(2)}_{\beta;k}(a,b,T) := \beta^{2k} G_{T}(a-b)^{-1} \int_{\mathbb{R}^{k \times d},\; z_{0} = a, \; z_{k+1} = b} \prod_{j=1}^{k} dz_{j} \int_{0 = t_{0} < t_{1} < \ldots < t_{k} < T, \; \exists t_{i} > \Lambda_{T}} \prod_{j=1}^{k} dt_{j}\nonumber\\
    & &&\times \prod_{j=1}^{k} G_{t_{j}-t_{j-1}}(z_{j}-z_{j-1}) R(\sqrt{2}z_{j}) \times G_{T-t_{k}}(z_{k+1}-z_{k}). \label{definition_A_k_(2)}
\end{alignat}
To prove (\ref{bridge_limit_formula}), we rely on applying a sequential version of dominated
convergence theorem \cite[Exercises 5.23]{ZZZ} to the summations on the right-hand side of (\ref{decomposition_Lambda}). As a result, it is enough to show that for each $k\geq 1$, as $T\to\infty$,
\begin{align} \label{limit_goal_first} 
    \mathbf{A}^{(1)}_{\beta;k}(0,\sqrt{T}z,T) = \mathfrak{h}_{\beta;k}(0)+o(1) 
\end{align}
and
\begin{align} \label{limit_goal}
    \mathbf{A}^{(2)}_{\beta;k}(0,\sqrt{T}z,T) = o(1),
\end{align}
where $\mathfrak{h}_{\beta;k}(0)$ is defined in (\ref{definition_h_k}). Let us now explain the reason. By combining (\ref{first_decomposition}) and (\ref{second_decomposition}), one has
\begin{align*}
    \mathbf{A}^{(\nu)}_{\beta;k}(0,\sqrt{T}z,T) \leq (C_{1}C_{2}) \sum_{i=1}^{k+1} \mathfrak{h}_{\beta_{*},i-1}(0) \cdot \mathfrak{h}_{\beta_{*};k-i+1}(\sqrt{T}z), \quad \nu = 1,2.
\end{align*}
In addition, thanks to (\ref{sum_h}) and (\ref{convergent_rate_small_h}), it holds that
\begin{align}
    \sum_{k=1}^{\infty} \sum_{i=1}^{k+1} \mathfrak{h}_{\beta,i-1}(0) \cdot \mathfrak{h}_{\beta;k-i+1}(\sqrt{T}z) = \mathfrak{h}_{\beta}(0) \mathfrak{h}_{\beta}(\sqrt{T}z) = \mathfrak{h}_{\beta}(0) + o(1) \quad \text{as } T \to\infty.
\end{align}
Therefore, (\ref{bridge_limit_formula}) follows by applying \cite[Exercises 5.23]{ZZZ} to \eqref{decomposition_Lambda} together with \eqref{limit_goal_first} and (\ref{limit_goal}).

We are now ready to conclude (\ref{bridge_limit_formula}).
\begin{proof} [Proof of (\ref{bridge_limit_formula})]
Based on the foregoing idea, it is enough to show (\ref{limit_goal_first}) and (\ref{limit_goal}). To begin with, we observe that (\ref{limit_goal_first}) follows from dominated convergence theorem. Indeed, since $z \neq 0$ and $\Lambda_{T} = o(T)$ so that $\Lambda_{T} \to \infty$, it follows that
\begin{align}
    &\frac{G_{T - t_{k}}(\sqrt{T}z - z_{k})}{G_{T}(\sqrt{T}z)} = 1+o(1) \quad \text{as $T\to\infty$} \quad \forall t_{k}>0, \; z_{k} \in \mathbb{R}^{d}, \label{tail_gaussian_limit}\\
    \sup_{t_{k} < \Lambda_{T}, \; z_{k} \in \mathbb{R}^{d}} &\frac{G_{T - t_{k}}(\sqrt{T}z - z_{k})}{G_{T}(\sqrt{T}z)} 
    \leq \exp(|z|^{2}/2) \quad \forall T\geq 1. \label{tail_gaussian_estimate}
\end{align}
In particular, applying \eqref{tail_gaussian_estimate} shows that the integrand of $\mathbf{A}{\beta;k}^{(1)}(0,\sqrt{T}z,T)$ is bounded above by the integrand of $\mathfrak{h}{\beta;k}(0)$. The latter, defined in (\ref{definition_h_k}), is finite by part (i) of Lemma \ref{lemma_h_beta_property_1}. Therefore, since $\Lambda_{T} \to\infty$, applying dominated convergence theorem together with \eqref{tail_gaussian_limit} implies \eqref{limit_goal_first}.

Let us now turn to the proof of (\ref{limit_goal}). For this purpose, observe that
\begin{align} \label{decomposition_A_2}
    &\mathbf{A}_{\beta;k}^{(2)}(0,\sqrt{T}z,T) \leq (C_{1}C_{2}) \bkt{\sum_{i=1}^{k} \mathfrak{h}_{\beta_{*};i-1}(0) \cdot \mathfrak{h}_{\beta_{*};k-i+1}(\sqrt{T}z)\nonumber\\
    &+\beta_{*}^{2k} \int_{\mathbb{R}^{k \times d},\; z_{0} = 0} \prod_{j=1}^{k} dz_{j} \int_{0 = t_{0} < t_{1} < \ldots < t_{k} < T, \; \exists t_{i} > \Lambda_{T}} \prod_{j=1}^{k} dt_{j}\times \prod_{j=1}^{k} G_{t_{j}-t_{j-1}}(z_{j}-z_{j-1}) R(\sqrt{2}z_{j})}. 
\end{align}
Indeed, we bound $\mathbf{A}_{\beta;k}^{(2)}(0,\sqrt{T}z,T)$ as (\ref{first_decomposition}) and (\ref{second_decomposition}). But, in the final term (i.e., when $i = k+1$), we keep the conditions $\exists t_{i} > \Lambda_{T}$ and $0<t_{1}<\ldots<t_{k}<T$. Therefore, $\mathbf{A}_{\beta;k}^{(2)}(0,\sqrt{T}z,T) = o(1)$ follows from (\ref{decomposition_A_2}). Indeed, due to \eqref{convergent_rate_small_h} and \eqref{sum_h},
the first term on the right-hand side of (\ref{decomposition_A_2}) tends to zero since $\mathfrak{h}_{\beta;m}(\sqrt{T}z) \to 0$ for every $m\geq 1$.  The second term on the right-hand side of (\ref{decomposition_A_2}) converges to zero thanks to dominated convergence theorem and $\mathbf{1}_{\{\exists t_{i} > \Lambda_{T}\}} \to 0$. Here we have used that $\Lambda_{T} \to \infty$. As a result, we complete the proof of both of (\ref{limit_goal_first}) and (\ref{limit_goal}), which implies (\ref{bridge_limit_formula}).

\end{proof}

\section{Fluctuations for the continuous directed polymer.} \label{section_proof_CDP}
This subsection is devoted to the proof of Corollary \ref{corollary_CDP} and Corollary \ref{remark_stability}. Corollary \ref{corollary_CDP} follows directly by combining Theorem \ref{ST_SHE_and_KPZ} and the identities (\ref{SHE_D1}) and (\ref{KPZ_D0}) with $t_{j} = 1$. Hence, we only focus on the proof of Corollary \ref{remark_stability}. Before proceeding, we begin by decomposing the random fields $\mathbf{F}^{\textup{PF};\beta}_{T}(x_{j})$ and $\mathbf{F}^{\textup{PF};\beta}_{T}(x_{j})$. Fix a collection $(x_{j})_{1\leq j\leq M}$. We apply Proposition \ref{SHE_key_proposition} with $t = 1$, and apply Proposition \ref{proposition_KPZ_D1} and Proposition \ref{proposition_KPZ_D2}, also evaluated at $t_{j} = 1$, to obtain the following approximations:
\begin{align} 
    &\mathbf{F}^{\textup{PF};\beta}_{T}(x_{j})
    = \mathcal{Z}^{\beta}_{T}(x_{j}\sqrt{T};\xi) \times T^{\frac{d-2}{4}}\int_{\mathbb{R}^{d}} dy G_{1}(x_{j} - y)
    (\mathcal{Z}^{\beta}_{\infty}(y\sqrt{T};\xi(\cdot,\cdot+T))-1) + o(1) \quad\text{in } L^{2}, \label{approximation_partition_function_1}\\
    &\mathbf{F}^{\textup{FE};\beta}_{T}(x_{j})
    = T^{\frac{d-2}{4}}\int_{\mathbb{R}^{d}} dy G_{1}(x_{j} - y)
    (\mathcal{Z}^{\beta}_{\infty}(y\sqrt{T};\xi(\cdot,\cdot+T))-1)+o(1) \quad \text{in } L^{1}. \label{approximation_free_energy_1}
\end{align}
For convenience, let $\mathcal{X} \in \sigma(\xi)$ and let $\Gamma \subseteq \mathbb{R}^{d}\times \mathbb{R}_{+}$ be a measurable set, then, for simplicity, we say
\begin{align} \label{notation:sigma}
    \mathcal{X} \in \sigma(\xi(x,t): (x,t) \in \Gamma) 
    \quad\text{whenever}\quad
    \mathcal{X}\in \sigma\sbkt{ \xi(\mathbf{1}_{A}): A\subseteq \Gamma \text{ is a measurable subset}}.
\end{align}

\subsection{Proof of Corollary \ref{remark_stability}.}


We begin with a few observations to simplify the proof. Let $A$ be a measurable set in $\sigma(\xi)$, with $\mathbb{P}(A) > 0$. To prove Corollary \ref{remark_stability}, it is enough to compute $\mathbb{E}[\exp(i\Vec{\lambda}\cdot \Vec{X}_{T})|A]$ for every $\Vec{\lambda} \in \mathbb{R}^{M}$, where $\Vec{X}_{T}$ stands for $\sbkt{\mathbf{F}^{\textup{PF};\beta}_{T}(x_{j})}_{1\leq j\leq M}$ or $\sbkt{\mathbf{F}^{\textup{FE};\beta}_{T}(x_{j})}_{1\leq j\leq M}$. This is because the limiting distribution of $\mathbb{P}(\Vec{X}_{T}\in \cdot |A)$ yields the limiting distribution of the pair $(\Vec{X}_{T},W)$. Moreover, by a standard density argument, it is enough to consider the case where $A \in \sigma(\xi(x,t):(x,t) \in K)$, for some compact subset $K \subseteq \mathbb{R}^{d} \times \mathbb{R}_{+}$. 

\paragraph{Step 1.} Let us first work on the case of $\sbkt{\mathbf{F}^{\textup{FE};\beta}_{T}(x_{j})}_{1\leq j\leq M}$. Observe that the right-hand side of (\ref{approximation_free_energy_1}) depends on the noise $(\xi(x,t))_{(x,t)\in \mathbb{R}^{d} \times (T,\infty)}$, and is therefore independent of $A$ when $T$ is large enough, since $A \in \sigma(\xi(x,t):(x,t) \in K)$. Consequently, by using (\ref{SHE_gaussian_limit_0}) with $t_{j} = 1$, we know that 
\begin{align*}
    \mathbb{E}\bigbkt{\exp\bkt{i\Vec{\lambda}\cdot \sbkt{\mathbf{F}^{\textup{FE};\beta}_{T}(x_{j})}_{1\leq j\leq M}}; A}
    \overset{T\to\infty}{\longrightarrow} 
    \mathbb{E}\bigbkt{\exp\bkt{i\Vec{\lambda} \cdot \sbkt{\mathbf{U}^{\beta}(x_{j},1)}_{1\leq j\leq M}}}.
\end{align*}
The proof of (\ref{proposition:fe1}) is complete.

\paragraph{Step 2.} We now turn to the case of $\sbkt{\mathbf{F}^{\textup{PE};\beta}_{T}(x_{j})}_{1\leq j\leq M}$, with $x_{1} = 0$. That is, we aim to establish 
\begin{align} \label{proof:CCP_goal}
    \mathbb{E}\bigbkt{\exp\bkt{i\Vec{\lambda}\cdot \sbkt{\mathbf{F}^{\textup{PF};\beta}_{T}(x_{j})}_{1\leq j\leq M}}; A}
    &\overset{T\to\infty}{\longrightarrow} 
    \mathbb{E}\bigbkt{\exp\bkt{i\Vec{\lambda}(1)\cdot \sbkt{\mathcal{Z}_{T}^{\beta}(0;\xi) \cdot \mathbf{U}^{\beta}(0,1)}\nonumber\\
    &+ i(\Vec{\lambda}(j))_{2\leq j\leq M} \cdot \sbkt{\mathbf{Z}^{\beta}(x_{j},1)\cdot \mathbf{U}^{\beta}(x_{j},1)}_{2\leq j\leq M}}}.
\end{align}
We begin by deriving a sufficient condition for \eqref{proof:CCP_goal} to simplify the proof. Similarly, we notice that the second term on the right-hand side of (\ref{approximation_partition_function_1}) is independent of the first term on the right-hand side of (\ref{approximation_partition_function_1}) and the event $A$ when $T$ is large enough, since $A \in \sigma(\xi(x,t): (x,t) \in K)$. Moreover, the convergence in distribution of the second term on the right-hand side of (\ref{approximation_partition_function_1}) follows from (\ref{SHE_gaussian_limit_0}) with $t_{j} = 1$. Consequently, to prove (\ref{proof:CCP_goal}), it remains to show the following, due to (\ref{approximation_partition_function_1}):
\begin{align} \label{proof:measure_convergence}
    \mathbb{P}\sbkt{\sbkt{\mathcal{Z}^{\beta}_{T}(x_{j}\sqrt{T};\xi)}_{1\leq j\leq M} \in \cdot |A} \overset{d}{\longrightarrow} \mathbb{P}(\mathcal{Z}_{\infty}^{\beta}(0;\xi)\in \cdot |A) \times \mathbb{P}\sbkt{\sbkt{\mathbf{Z}^{\beta}(x_{j},1)}_{2\leq j\leq M} \in \cdot}.
\end{align}
To verify (\ref{proof:measure_convergence}), we show the convergence of $\mathbb{E}[\exp\sbkt{i\Vec{\lambda'}\cdot \sbkt{\mathcal{Z}^{\beta}_{T}(x_{j}\sqrt{T};\xi)}_{1\leq j\leq M}} | A]$ for every $\Vec{\lambda'} \in \mathbb{R}^{M}$. Our approach is to show that, when $T$ is large, the input noise of $\sbkt{\mathcal{Z}^{\beta}_{T}(x_{j}\sqrt{T};\xi)}_{2\leq j\leq M}$ can be replaced by the noises $\xi_{2},\ldots,\xi_{M}$. Here, $\xi,\xi_{2},\ldots,\xi_{M}$ are i.i.d. space-time white noises. 

For this purpose, we recall Lemma \ref{lemma_ restricted_partition_functions} as follows: 
\begin{equation} \label{stable_approximate}
    \mathcal{Z}_{T}^{\beta}(\sqrt{T}x_{j};\xi)
    = \mathcal{Z}_{\tau_{T};A_{T}}^{\beta}(\sqrt{T}x_{j};\xi) + o(1) \quad\text{in } L^{1} \quad 1\leq j\leq M,
\end{equation}
where $\tau_{T} \to \infty$ such that $\tau_{T} = o(\sqrt{T})$, and $A_{t}$ is given in \eqref{definition_A_T}. Therefore, it holds that
\begin{align} \label{proof:CCP_1}
    \mathbb{E}\bigbkt{\exp\bkt{i\Vec{\lambda'}\cdot \sbkt{\mathcal{Z}^{\beta}_{T}(x_{j}\sqrt{T};\xi)}_{1\leq j\leq M}} ; A} = \mathbb{E}\bigbkt{\exp\bkt{i\Vec{\lambda'}\cdot \sbkt{\mathcal{Z}_{\tau_{T};A_{T}}^{\beta}(\sqrt{T}x_{j};\xi)}_{1\leq j\leq M}} ; A} + o(1).
\end{align}

Now, we analyze the noise used by the restricted partition function on the right-hand side of (\ref{stable_approximate}). To this end, following the proof of Lemma \ref{lemma_asymptotic_independence} with $\delta_{j} = 0$, we note that the components of the vector $(\mathcal{Z}^{\beta}_{\tau_{T};A_{T}}(x_{j}\sqrt{T};\xi))_{1\leq j\leq M}$ occupy disjoint regions of the noise $\xi$ when $T$ is sufficiently large. As a result, the components of the vector $(\mathcal{Z}^{\beta}_{\tau_{T};A_{T}}(x_{j}\sqrt{T};\xi))_{2\leq j\leq M}$ are independent. In particular, they are independent of  $(\mathcal{Z}^{\beta}_{\tau_{T};A_{T}}(0;\xi),A)$, since $A \in \sigma(\xi(x,t): (x,t) \in K)$. Therefore, we conclude
\begin{align} \label{proof:CCP_2}
    \text{(R.H.S) of (\ref{proof:CCP_1})}
    = \mathbb{E}\bigbkt{\exp\bkt{i\Vec{\lambda'}\cdot \sbkt{\mathcal{Z}_{\tau_{T};A_{T}}^{\beta}(\sqrt{T}x_{j};\xi_{j})}_{1\leq j\leq M}} ; A},
\end{align}
where $(\xi_{j})_{2\leq j\leq M}$ are i.i.d. space-time white noises that are independent of $\xi_{1} := \xi$. Consequently, after recovering $(\mathcal{Z}^{\beta}_{\tau_{T};A_{T}}(x_{j}\sqrt{T};\xi_{j}))_{1\leq j\leq M}$ by using (\ref{stable_approximate}), we obtain 
\begin{align*}
    \text{(R.H.S) of (\ref{proof:CCP_2})} = \mathbb{E}\bigbkt{\exp\bkt{i\Vec{\lambda'}\cdot \sbkt{\mathcal{Z}_{T}^{\beta}(0;\xi_{j})}_{1\leq j\leq M}} ; A} + o(1).
\end{align*}
Therefore, thanks to (\ref{limit_CDP}), the proof of (\ref{proof:measure_convergence}) is complete.

\paragraph{Step 3.} Finally, the case of $x_{j} \neq 0$ follows from (\ref{proof:CCP_goal}) with $\Vec{\lambda}(1) = 0$. The proof of Corollary \ref{remark_stability} is complete.

\appendix

\section{Limiting approximation scheme} \label{section_asymptotics of the renormalizations}

The goal of this section is to establish the limit of $\mathbf{X}^{\beta}_{\varepsilon}(x,t)$ at multiple space–time points.

\begin{lemma}\label{pointwise_convergence}
Assume that $d\geq 3$ and $\beta < \beta_{c}$, where $\beta_{c}$ is defined in (\ref{limit_CDP}). Then, as $\varepsilon\to 0$, it holds that
\begin{equation} \label{pointwise_convergence_formula_1}
    \sbkt{\mathbf{X}^{\beta}_{\varepsilon}(x,t)}_{(x,t) \in \mathbb{R}^{d} \times \mathbb{R}_{+}}
    \overset{\textup{f.d.m.}}{\longrightarrow} \sbkt{\mathbf{Z}^{\beta}(x,t)}_{(x,t) \in \mathbb{R}^{d} \times \mathbb{R}_{+}}.
\end{equation}
Here, the random field $\sbkt{\mathbf{Z}^{\beta}(x,t)}_{(x,t) \in \mathbb{R}^{d} \times \mathbb{R}_{+}}$ is defined in Theorem \ref{ST_SHE_and_KPZ}.
\end{lemma}

Before proceeding, we explain the heuristics behind (\ref{pointwise_convergence_formula_1}). To prove Lemma \ref{pointwise_convergence}, it is enough to show 
\begin{equation} \label{pointwise_convergence_formula_proof_section_goal_1}
    \sbkt{\mathcal{Z}^{\beta}_{t_{j} T}(x_{j}\sqrt{T};\xi(\cdot,\cdot+\delta_{j} T))}_{1\leq j\leq M}
    \overset{d}{\longrightarrow} \sbkt{\mathbf{Z}^{\beta}(x_{j},t_{j})}_{1\leq j\leq M} \quad \text{as } T\to\infty,
\end{equation}
by using (\ref{SHE_D0}). Here $(x_{j},t_{j})_{1\leq j\leq M}$ is a collection of space-time points defined in (\ref{definition:space_time_points}) and $\delta_{j}$ is defined in (\ref{notation:change_of_variable}). The idea is to establish the \emph{asymptotic independence} of the components on the left-hand side of (\ref{pointwise_convergence_formula_proof_section_goal_1}). As a result, the convergence \eqref{pointwise_convergence_formula_proof_section_goal_1} follows from (\ref{limit_CDP}). The key observation is that these components depend on disjoint regions of the white noise, which is inspired by the proof of \cite[Lemma 3.1]{COSCO2022127}.

\paragraph{Step 1.} We begin by approximating the partition function by removing the region of the white noise that the polymer is unlikely to visit with high probability.
\begin{lemma} \label{lemma_ restricted_partition_functions}
Assume that $d\geq 3$ and $\beta<\beta_{c}$. Then, for each $t > 0$, it holds that
\begin{equation} \label{path_restriction}
    \mathcal{Z}_{tT}^{\beta}(\sqrt{T}x;\xi)
    = \mathcal{Z}_{\tau_{T};A_{T}}^{\beta}(\sqrt{T}x;\xi) + o(1) \quad\text{in } L^{1} \quad \text{as $T\to\infty$},
\end{equation}
where $\tau_{T} \to\infty$ such that $\tau_{T} = o(\sqrt{T})$, $A_{T}$ is defined by
\begin{align} \label{definition_A_T}
    A_{T}:= \set{w \in C([0,\tau_{T}];\mathbb{R}^{d}): \sup_{0\leq s\leq \tau_{T}}|w(s) - w(0)|\leq \tau_{T}},
\end{align}
and the restricted partition function is given by
\begin{equation}\label{path_restriction_def}
    \mathcal{Z}_{T;\Gamma}^{\beta}(x;\xi)
    := \mathbf{E}_{x}\bigbkt{\exp\bkt{\beta \int_{0}^{T} ds \xi_{1}(B(s),s) - \frac{\beta^{2} R(0) T }{2}};(B(s))_{0\leq s\leq T} \in \Gamma}.
\end{equation}
\end{lemma}
\begin{proof}
The approximation (\ref{path_restriction}) follows by combining 
\begin{equation} \label{1_4}
    \mathbb{E}[\sabs{\mathcal{Z}^{\beta}_{tT}(\sqrt{T}x;\xi)-\mathcal{Z}^{\beta}_{\tau_{T}}(\sqrt{T}x;\xi)}] = o(1)
\end{equation}
and
\begin{align*}
    \mathbb{E}[\sabs{\mathcal{Z}^{\beta}_{\tau_{T}}(\sqrt{T}x;\xi)-\mathcal{Z}^{\beta}_{\tau_{T};A_{T}}(\sqrt{T}x;\xi)}]
    = \mathbf{P}_{0}(\sup_{t\leq \tau_{T}}|B(t)|> \tau_{T})
    =\mathbf{P}_{0}(\sup_{t\leq 1/\tau_{T}}|B(t)|> 1)
    = o(1).
\end{align*}
Here, (\ref{1_4}) is a immediate consequence of the uniformly integrability mentioned above  (\ref{limit_CDP}).
\end{proof}

\paragraph{Step 2.} We can now show the asymptotic independence. 
\begin{lemma} \label{lemma_asymptotic_independence}
When $T$ is large enough, the restricted partition functions $(\mathcal{Z}^{\beta}_{\tau_{T};A_{T}}(x_{j}\sqrt{T};\xi(\cdot,\cdot+\delta_{j} T)))_{1\leq j\leq M}$ are independent.     
\end{lemma}
\begin{proof}
Recall the notation (\ref{notation:sigma}). According to \eqref{path_restriction_def}, we know that 
\begin{align} \label{meausurability}
    \mathcal{Z}^{\beta}_{\tau_{T};A_{T}}(x_{j}\sqrt{T};\xi(\cdot,\cdot+\delta_{j} T)) \in \sigma\bkt{ \xi(y,s): \delta_{j}T\leq s\leq \tau_{T}+\delta_{j}T \text{ and } |y - x_{j}\sqrt{T}| \leq \tau_{T}+r_{\phi}},
\end{align}
where $r_{\phi}$ is a positive number such that $\text{supp}(\phi) \subseteq B(0,r_{\phi})$. For all sufficiently large $T$, the independence follows from $\Gamma_{T}(i) \cap \Gamma_{T}(j) = \emptyset \quad \forall i\neq j$. Here, $\Gamma_{T}(j)$ denotes the collection of space-time points $(y,s)$ satisfying the conditions on the right-hand side of \eqref{meausurability}. To see this disjointness, suppose for contradiction that there exist indices $i<j$ such that $\Gamma_{T}(i) \cap \Gamma_{T}(j) \neq \emptyset$ for infinitely many $T$. Then one has
\begin{equation*}
    \delta_{i}T \leq \tau_{T}+\delta_{j} T 
    \quad \text{and}\quad |x_{i}\sqrt{T}-x_{j}\sqrt{T}| \leq 2\tau_{T}+2r \quad \text{for infinitely many $T$},
\end{equation*}
which is a contradiction. Indeed, since we have either $t_{i} < t_{j}$ or $x_{i} \neq x_{j}$, the former shows that $\delta_{i}T > \tau_{T}+\delta_{j}T$ for all sufficiently large $T$, while the second case yields $0<|x_{i} - x_{j}|$.
\end{proof}

\paragraph{Step 3.} Thanks to Lemma \ref{lemma_asymptotic_independence}, one has
\begin{align*}
    &\mathbb{P}\bkt{\mathcal{Z}^{\beta}_{\tau_{T};A_{T}}(x_{1}\sqrt{T};\xi(\cdot,\cdot+\delta_{1}T)) \in \cdot,\ldots,\mathcal{Z}^{\beta}_{\tau_{T};A_{T}}(x_{M}\sqrt{T};\xi(\cdot,\cdot+\delta_{M}T))\in \cdot}\nonumber\\
    &= \mathbb{P}\bkt{\mathcal{Z}^{\beta}_{\tau_{T};A_{T}}(x_{1}\sqrt{T};\xi(\cdot,\cdot+\delta_{1}T)) \in \cdot} \times\ldots \times\mathbb{P}\bkt{\mathcal{Z}^{\beta}_{\tau_{T};A_{T}}(x_{M}\sqrt{T};\xi(\cdot,\cdot+\delta_{M}T))\in \cdot}\nonumber\\
    &= \mathbb{P}\bkt{\mathcal{Z}^{\beta}_{\tau_{T};A_{T}}(0;\xi) \in \cdot} \times\ldots \times\mathbb{P}\bkt{\mathcal{Z}^{\beta}_{\tau_{T};A_{T}}(0;\xi)\in \cdot},
\end{align*}
where we have used the translation-invariance of
the white noise in the last equality. Hence, Lemma \ref{pointwise_convergence} follows from (\ref{limit_CDP}) after we recover each restricted partition functions by using Lemma \ref{lemma_ restricted_partition_functions}.

\section{Spatially averaged fluctuations} \label{section:fluctuation_overline}
The purpose of this section is to prove (\ref{fluctuation:overline}). As $\varepsilon \to 0$, the following approximations hold in $L^{1}$-sense:
\begin{alignat}{2}
    & \overline{\mathbf{F}}^{\textup{SHE};\beta}_{\varepsilon}&&(f_{j},t_{j})
    = \int_{\mathbb{R}^{d}} dx f_{j}(x)\mathcal{Z}^{\beta}_{t_{j} T}(x\sqrt{T};\xi^{(T^{-1/2},0,t)}(\cdot,\cdot+\delta_{j} T))\nonumber\\
    & &&\times T^{\frac{d-2}{4}}\int_{\mathbb{R}^{d}} dy G_{t_{j}}(x - y)
    (\mathcal{Z}^{\beta}_{\infty}(y\sqrt{T};\xi^{(T^{-1/2},0,t)}(\cdot,\cdot+t_{M}T))-1) + o(1) \label{error:overline_fluctuation_1},\\
    & \overline{\mathbf{F}}^{\textup{KPZ};\beta}_{\varepsilon}&&(f_{j},t_{j})
    = \int_{\mathbb{R}^{d}} dx f_{j}(x)T^{\frac{d-2}{4}}\int_{\mathbb{R}^{d}} dy G_{t_{j}}(x - y)
    (\mathcal{Z}^{\beta}_{\infty}(y\sqrt{T};\xi^{(T^{-1/2},0,t)}(\cdot,\cdot+t_{M}T))-1)+o(1), \label{error:overline_fluctuation_2}
\end{alignat}
where $T = \varepsilon^{-2}$ and $\delta_{j}$ is defined in (\ref{notation:change_of_variable}). Indeed, (\ref{error:overline_fluctuation_1}) follows by using (\ref{SHE_D0_1}) and (\ref{SHE_key_lemma_formula}), and (\ref{error:overline_fluctuation_2}) follows from by combining (\ref{SHE_D0_1}), (\ref{KPZ_D1}), and (\ref{KPZ_D2}). Also, applying some homogenization argument yields the following $L^{1}$-approximation:
\begin{align} \label{error:overline_fluctuation_3}
    \text{(R.H.S) of (\ref{error:overline_fluctuation_1})}& = \int_{\mathbb{R}^{d}} dx f_{j}(x)T^{\frac{d-2}{4}}\int_{\mathbb{R}^{d}} dy G_{t_{j}}(x - y)
    (\mathcal{Z}^{\beta}_{\infty}(y\sqrt{T};\xi^{(T^{-1/2},0,t)}(\cdot,\cdot+t_{M}T))-1)+o(1).
\end{align}
Consequently, since $\xi^{(T^{-1/2},0,t)}(\cdot,\cdot+t_{M}T)$ is a white noise, (\ref{fluctuation:overline}) now follows by applying (\ref{corollary_2.7}) with test function $f_{j}* G_{t_{j}}$ to the right-hand sides of (\ref{error:overline_fluctuation_2}) and (\ref{error:overline_fluctuation_3}). 

\section{Non-convergence of the conditional variance} \label{section:proof_of_lemma:not_converge_in_probability}
This section is dedicated to verifying Lemma \ref{lemma:not_converge_in_probability}. Our strategy is to approximate the conditional variance $(V^{\Vec{\lambda}}_{T,\infty})^{2}$ from \eqref{definition:conditional_variance} by an expression of the form given on the right-hand side of \eqref{sum_lambda}, and then show that the convergence of $(V^{\Vec{\lambda}}_{T,\infty})^{2}$ contradicts \eqref{limit_CDP}; see \eqref{definition:widetidle_V} and Lemma \ref{lemma_not_convergence}, respectively.

Let $M\geq 2$, and fix a non-zero $\Vec{\lambda} \in \mathbb{R}^{M}$. According to the definition (\ref{definition:conditional_variance}), applying the Markov property of Brownian motion, we obtain the following expression for the conditional variance $(V^{\Vec{\lambda}}_{T,\infty})^{2}$:
\begin{align}
    (V^{\Vec{\lambda}}_{T,\infty})^{2} &= \sum_{1\leq j,j'\leq M } \Vec{\lambda}(j) \Vec{\lambda}(j')\sum_{k=0}^{\infty} T^{\frac{d-2}{2}} \int_{\mathbb{R}^{2d}} dy_{j} dy'_{j'} \mathcal{Z}_{T+k-1}^{\beta}(\sqrt{T}x_{j},y_{j};\xi) \mathcal{Z}_{T+k-1}^{\beta}(\sqrt{T}x_{j'},y'_{j'};\xi) \nonumber\\
    &\quad\quad\quad\quad\quad\quad \times \bkt{ \mathbf{E}^{B}_{\frac{y_{j}-y'_{j'}}{\sqrt{2}}}\bigbkt{\exp\bkt{\beta^{2}\int_{0}^{1}R(\sqrt{2}B(s)) ds}} -1}, \nonumber\\
    \mathcal{Z}^{\beta}_{N}(x,y;\xi)
    &:= \mathbf{E}^{B}_{x}\bigbkt{\exp\bkt{\beta \int_{0}^{N} ds \xi_{1}(B(s),s) - \frac{\beta^{2} R(0) N }{2}} ; B(N) = y}. \label{expression:conditional_variance}
\end{align}

As the idea explained in \cite[Section 2.2]{GFDDP}, we have the following approximation of the conditional variance. We omit the detailed verification, as it follows from the analogous idea in \cite[Section 2.2]{GFDDP}, with \cite[Theorem 2.3]{GFDDP} replaced by \cite[Theorem B.1]{COSCO2022127}.

\begin{lemma} \label{lamma:expression:conditional_variance_1}
Assume that $d\geq 3$, $\beta < \beta_{L^{2}}$, and $M\geq 2$. Let $\tau_{k} \to \infty$ such that $\tau_{k} = o(k^{a})$ for some $a < 1/2$. Then the following holds in $L^{1}$:
\begin{align} \label{expression:conditional_variance_1}
    &(V^{\Vec{\lambda}}_{T,\infty})^{2} = \sum_{1\leq j,j'\leq M } \Vec{\lambda}(j) \Vec{\lambda}(j') \mathcal{Z}_{\tau_{T}}^{\beta}(\sqrt{T}x_{j};\xi)\mathcal{Z}_{\tau_{T}}^{\beta}(\sqrt{T}x_{j'};\xi) \nonumber\\
    &\times\sum_{k=0}^{\infty} T^{\frac{d-2}{2}} \int_{\mathbb{R}^{d}} dy_{j} dy'_{j'} 
    \mathbf{E}^{B}_{\frac{y_{j}-y'_{j'}}{\sqrt{2}}}\bigbkt{\exp\bkt{\beta^{2}\int_{0}^{\tau_{k}}R(\sqrt{2}B(s)) ds}}
    G_{T+k}(\sqrt{T}x_{j} - y_{j})G_{T+k}(\sqrt{T}x_{j'} - y'_{j'})
    \nonumber\\
    &\times\bkt{ \mathbf{E}^{B}_{\frac{y_{j}-y'_{j'}}{\sqrt{2}}}\bigbkt{\exp\bkt{\beta^{2}\int_{0}^{1}R(\sqrt{2}B(s)) ds}} -1} + o(1).
\end{align}
\end{lemma}
\begin{remark}
Formally, the approximation (\ref{expression:conditional_variance_1}) can be viewed as the continuous analogue of \cite[(32)]{GFDDP}, where the first and the second Brownian expectations correspond to the second moment of the backward partition function and constant $\kappa(\beta)$, respectively.
\end{remark}
For each $1\leq j,j'\leq M$, we set $C_{T}(j,j')$ to be the coefficient defined by the second and the third lines on the right-hand side of (\ref{expression:conditional_variance_1}). Let us now show that $C_{T}(j,j')$ converges to the corresponding the coefficient in $(U^{\Vec{\lambda}})^{2}$, where $U^{\Vec{\lambda}}$ is defined in (\ref{limting_martinagle_ST}).
\begin{lemma} \label{lemma:final_coefficient}
Assume that $d\geq 3$, $\beta < \beta_{L^{2}}$, and $M\geq 2$. Recall that $\gamma(\beta)^{2}$ is defined in (\ref{gamma_and_R}). Then, as $T\to \infty$,
\begin{align} \label{final_coefficient}
    C_{T}(j,j') = \gamma^{2}(\beta) \int_{0}^{\infty} G_{2s+2}(x_{j}-x_{j'}) ds + o(1) \quad \forall 1\leq j,j' \leq M.
\end{align}
\end{lemma}
\begin{proof}
For simplicity, we only outline the main ideas. By applying a change of variable (\ref{change_of_variable:two_body}) to $(y_{j},y'_{j})$, and by replacing $\tau_{k}$ by infinite via using the condition $\beta<\beta_{L^{2}}$, $C_{T}(j,j')$ can be approximated by
\begin{align*}
    C_{T}(j,j') &=   \int_{\mathbb{R}^{d}} dz \sum_{k=0}^{\infty} T^{\frac{d-2}{2}}G_{T+k}\bkt{\sqrt{T}\cdot \frac{x_{j}-x_{j'}}{\sqrt{2}} - z} \\
    &\times \mathbf{E}^{B}_{z}\bigbkt{\exp\bkt{\beta^{2}\int_{0}^{\infty}R(\sqrt{2}B(s)) ds}} \bkt{ \mathbf{E}^{B}_{z}\bigbkt{\exp\bkt{\beta^{2}\int_{0}^{1}R(\sqrt{2}B(s)) ds}} -1} + o(1).
\end{align*}
Furthermore, for each $z \in \mathbb{R}^{d}$, a Riemann sum approximation yields
\begin{align*}
    \lim_{T\to\infty} \sum_{k=0}^{\infty}\frac{1}{T}G_{1+k/T}\bkt{\frac{x_{j}-x_{j'}}{\sqrt{2}} - \frac{z}{\sqrt{T}}} = \int_{0}^{\infty} ds G_{1+s}\bkt{\frac{x_{j}-x_{j'}}{\sqrt{2}}}.
\end{align*}
Therefore, applying dominated convergence theorem gives
\begin{align} \label{fluctuation_identity_C_T}
    &C_{T}(j,j') 
    = \int_{0}^{\infty} ds G_{1+s}\bkt{\frac{x_{j}-x_{j'}}{\sqrt{2}}} \nonumber\\
    &\times \int_{\mathbb{R}^{d}} dz\mathbf{E}^{B}_{z}\bigbkt{\exp\bkt{\beta^{2}\int_{0}^{\infty}R(\sqrt{2}B(s)) ds}} \bkt{ \mathbf{E}^{B}_{z}\bigbkt{\exp\bkt{\beta^{2}\int_{0}^{1}R(\sqrt{2}B(s)) ds}} -1} + o(1).
\end{align}
In particular, the two Brownian expectations on the right-hand side of \eqref{fluctuation_identity_C_T} can be merged by using 
\begin{align} \label{fluctuation_Brownian_identity}
    \mathbf{E}^{B}_{z}\bigbkt{\exp\bkt{\beta^{2}\int_{0}^{1}&R(\sqrt{2}B(s)) ds}} -1 \nonumber\\
    &= \int_{\mathbb{R}^{d}} dw \beta^{2} R(\sqrt{2}w) \mathbf{E}^{B}_{w}\bigbkt{\exp\bkt{\beta^{2}\int_{0}^{1}R(\sqrt{2}B(s)) ds}; B(1) = z}.
\end{align}
Specifically, plugging (\ref{fluctuation_Brownian_identity}) into (\ref{fluctuation_identity_C_T}) and 
integrating out $z$ imply
\begin{align*}
    C_{T}(j,j') = \int_{0}^{\infty} ds G_{1+s}\bkt{\frac{x_{j}-x_{j'}}{\sqrt{2}}} \frac{1}{2^{d/2}} \gamma(\beta)^{2} +o(1)= \text{(R.H.S) of (\ref{final_coefficient})},
\end{align*}
where $\gamma(\beta)^{2}$ is given in (\ref{gamma_and_R}). The proof is complete.
\end{proof}
Thus, by combining Lemma \ref{lamma:expression:conditional_variance_1} and Lemma \ref{lemma:final_coefficient}, 
$(V^{\Vec{\lambda}}_{T,\infty})^{2}$ can be approximated in $L^{1}$ as
\begin{align} \label{definition:widetidle_V}
    (V^{\Vec{\lambda}}_{T,\infty})^{2} = &(\widetilde{V}^{\Vec{\lambda}}_{T,\infty})^{2} + o(1), \quad (\widetilde{V}_{T,\infty})^{2} := \sum_{1\leq j,j'\leq M } \Vec{\lambda}(j) \Vec{\lambda}(j') C_{\infty}(j,j') \mathcal{Z}_{\tau_{T}}^{\beta}(\sqrt{T}x_{j};\xi)\mathcal{Z}_{\tau_{T}}^{\beta}(\sqrt{T}x_{j'};\xi),
\end{align}
where $C_{\infty}(j,j')$ is given by the right-hand side of \eqref{final_coefficient}. Combining \eqref{definition:widetidle_V} with the following property allows us to conclude Lemma \ref{lemma:not_converge_in_probability}.
\begin{lemma} \label{lemma_not_convergence}
Assume that $d\geq 3$, $\beta < \beta_{L^{2}}$, and $M\geq 2$. Then $(\widetilde{V}^{\Vec{\lambda}}_{T,\infty})^{2}$ given in \eqref{definition:widetidle_V}
does not converge in probability whenever $\Vec{\lambda} \neq \Vec{0}$.
\end{lemma}
\begin{proof}
Suppose that $(\widetilde{V}^{\Vec{\lambda}}_{T,\infty})^{2}$ converges in probability to some random variable $(\widetilde{V}^{\Vec{\lambda}}_{\infty,\infty})^{2}$. Consequently, applying (\ref{pointwise_convergence_formula_proof_section_goal_1}) with $t_{j} = 1$ yields the probability distribution of $(\widetilde{V}^{\Vec{\lambda}}_{\infty,\infty})^{2}$:
\begin{align*}
    (\widetilde{V}^{\Vec{\lambda}}_{\infty,\infty})^{2} \overset{d}{=} \sum_{1\leq j,j'\leq M } \Vec{\lambda}(j) \Vec{\lambda}(j') C_{\infty}(j,j') \mathbf{Z}^{\beta}(x_{j},1)\mathbf{Z}^{\beta}(x_{j'},1).
\end{align*}
Hence, since $\Vec{\lambda} \neq \Vec{0}$, $(\widetilde{V}^{\Vec{\lambda}}_{\infty,\infty})^{2}$ is a non-degenerate random variable, since
$\mathcal{Z}_{\infty}^{\beta}(0;\xi)$ is a  non-degenerate random variable, which was proved in \cite[Corollary 2.5]{weakandstrong}.

\begin{figure} [!htb]
\centering
\begin{tikzpicture} [scale=0.5]

\node (a1) at (11,4.5) [anchor=west]{\textcolor{teal}{$B(\sqrt{T_{b}}x_{j},\tau_{T_{b}}+r)\times (0,\tau_{T_{b}})$}};

\node (b1) at (7,1.5) [anchor=west]{\textcolor{blue}{$B(\sqrt{T_{a}}x_{j'},\tau_{T_{a}}+r)\times (0,\tau_{T_{a}})$}};

\node (c1) at (-1,3) [anchor=east]{$|\sqrt{T_{b}}x_{j}-\sqrt{T_{a}}x_{j'}|=$};

\node (d1) at (11,0) [anchor=north]{$\tau_{T_{b}}$};
\filldraw[black] (11,0) circle (1.5pt);

\node (e1) at (7,0) [anchor=north]{$\tau_{T_{a}}$};
\filldraw[black] (7,0) circle (1.5pt);

\filldraw[lightgray] (0,3.5) rectangle (11,6); 
\filldraw[lightgray] (0,0.5) rectangle (7,2.5);

\filldraw[teal] (0,4) rectangle (11,5.5); 
\filldraw[blue] (0,1) rectangle (7,2);

\draw[->] (0,0) -- (12,0) node[anchor=west]{$\mathbb{R}_{+}$};

\draw[->] (0,0) -- (0,7) node[anchor=south]{$\mathbb{R}^{d}$};

\draw[<->] (-0.5,1.5) -- (-0.5,4.5);

\end{tikzpicture}
\caption{ \justifying For a fix $T_{a}$, since $\tau_{T} = o(\sqrt{T})$, we can choose $T_{b}\gg T_{a}$ such that the following holds. For a fixed $1\leq j\leq M$, the above green and blue regions represent the white noise used by  $\mathcal{Z}_{\tau_{T_{b}};A_{T_{b}}}^{\beta}(\sqrt{T_{b}}x_{j};\xi)$ and $\mathcal{Z}_{\tau_{T_{a}};A_{T_{a}}}^{\beta}(\sqrt{T_{a}}x_{j'};\xi)$, respectively. The top and bottom gray regions cover all the white noise occupied by
$\sbkt{\mathcal{Z}_{\tau_{T_{b}};A_{T_{b}}}^{\beta}(\sqrt{T_{b}}x_{j};\xi)}_{1\leq j\leq M}$ and 
$\sbkt{\mathcal{Z}_{\tau_{T_{a}};A_{T_{a}}}^{\beta}(\sqrt{T_{a}}x_{j'};\xi)}_{1\leq j'\leq M}$, respectively.
} 
\label{figure:observation_3}
\end{figure}

On the other hand, by using Lemma \ref{lemma_ restricted_partition_functions}, we have
\begin{align*}
    (\widetilde{V}^{\Vec{\lambda}}_{T,\infty})^{2} - &(\overline{V}^{\Vec{\lambda}}_{T,\infty})^{2} \overset{\mathbb{P}}{\to} 0, \\
    &(\overline{V}^{\Vec{\lambda}}_{T,\infty})^{2}:=\sum_{1\leq j,j'\leq M } \Vec{\lambda}(j) \Vec{\lambda}(j') C_{\infty}(j,j') \mathcal{Z}_{\tau_{T};A_{T}}^{\beta}(\sqrt{T}x_{j};\xi)\mathcal{Z}_{\tau_{T};A_{T}}^{\beta}(\sqrt{T}x_{j'};\xi),
\end{align*}
where $\mathcal{Z}_{\tau_{T};A_{T}}^{\beta}(\sqrt{T}x_{j};\xi)$ is defined in (\ref{path_restriction_def}). In particular, according to the noise used by $\mathcal{Z}_{\tau_{T};A_{T}}^{\beta}(\sqrt{T}x_{j};\xi)$, described in (\ref{meausurability}) with $\delta_{j} = 0$,  we know that $\sbkt{\mathcal{Z}_{\tau_{T_{a}};A_{T_{a}}}^{\beta}(\sqrt{T_{a}}x_{j};\xi)}_{1\leq j\leq M}$ and $\sbkt{\mathcal{Z}_{\tau_{T_{b}};A_{T_{b}}}^{\beta}(\sqrt{T_{b}}x_{j};\xi)}_{1\leq j\leq M}$ are independent when $T_{b}$ is much larger than $T_{a}$, due to the setup $\tau_{T} = o(\sqrt{T})$ in Lemma \ref{lamma:expression:conditional_variance_1}; see Figure \ref{figure:observation_3} for a precise clarification. Combining these observations, we can construct an increasing sequence $(T_{k})_{k\geq 1}$ such that $\sbkt{(\overline{V}^{\Vec{\lambda}}_{T_{k},\infty})^{2}}_{k\geq 1}$ are independent random variables and converges in probability to the same limit $(\widetilde{V}^{\Vec{\lambda}}_{\infty,\infty})^{2}$. However, this would imply that $(\widetilde{V}^{\Vec{\lambda}}_{\infty,\infty})^{2}$ is almost surely a constant, which contradicts to the fact that $(\widetilde{V}^{\Vec{\lambda}}_{\infty,\infty})^{2}$ is non-degenerate.
\end{proof}

\section*{Acknowledgements}
The author is deeply grateful to Professor Yu-Ting Chen for bringing this problem to his attention and for many insightful discussions, and to Professor Clément Cosco for valuable comments. The author also thanks Professor Jhih-Huang Li for helpful discussions, and the National Center for Theoretical Sciences for providing this opportunity.

\printbibliography

@article{weakandstrong,
author = {C. Mukherjee and A. Shamov and O. Zeitouni},
title = {{Weak and strong disorder for the stochastic heat equation and continuous directed polymers in $d\geq 3$}},
volume = {21},
journal = {Electronic Communications in Probability},
year = {2016},
     PAGES = {Paper No. 61, 12},
note = {\href{https://mathscinet.ams.org/mathscinet/article?mr=3548773}{MR3548773}}
}

@article{Feynman-Kac,
author = {B. Lorenzo and C. Nicoletta},
title = {{The stochastic heat equation: Feynman-Kac formula and intermittence}},
volume = {78},
journal = {Journal of Statistical Physics},
year = {1995},
number = {5-6},
PAGES = {1377--1401},
note = {\href{https://mathscinet.ams.org/mathscinet/article?mr=1316109}{MR1316109}}
}

@article{comets2019space,
author={Comets, F. and Cosco, C. and Mukherjee, C.},
title={Space-time fluctuation of the {K}ardar-{P}arisi-{Z}hang equation in $d\geq 3$ and the {G}aussian free field},
volume = {60},
journal = {Annales de l'Institut Henri Poincaré, Probabilités et Statistiques},
number = {1},
publisher = {Institut Henri Poincaré},
pages = {82 -- 112},
year = {2024},
note = {\href{https://mathscinet.ams.org/mathscinet/article?mr=4718375}{MR4718375}}
}

@article{normal,
author={Comets, F. and Cosco, C. and Mukherjee, C.},
title = {{Renormalizing the Kardar–Parisi–Zhang equation in $d\geq 3$ in weak disorder}},
volume = {179},
number = {3},
pages = {713-728},
journal = {Journal of Statistical Physics},
year = {2020},
note = {\href{https://mathscinet.ams.org/mathscinet/article?mr=4099995}{MR4099995}}
}

@article{KPZ,
  title = {Dynamic Scaling of Growing Interfaces},
  author = {Kardar, M. and Parisi, G. and Zhang, Y.-C.},
  journal = {Phys. Rev. Lett.},
  volume = {56},
  issue = {9},
  pages = {889--892},
  numpages = {0},
  year = {1986},
  publisher = {American Physical Society},
  note = {\href{https://journals.aps.org/prl/abstract/10.1103/PhysRevLett.56.889}{DOI: 10.1103/PhysRevLett.56.889}}
}

@article{Burgers,
author = {L. Bertini and G. Giacomin},
title = {{Stochastic Burgers and KPZ equations from particle systems}},
volume = {183},
journal = {Communications in Mathematical Physics},
number = {3},
publisher = {Springer},
pages = {571--607},
year = {1997},
note = {\href{https://mathscinet.ams.org/mathscinet/article?mr=1462228}{MR1462228}}
}

@article{Corwin1,
author = {Corwin, I.},
title = {The {K}ardar-{P}arisi-{Z}hang equation and universality class},
journal = {Random Matrices: Theory and Applications},
volume = {1},
number = {1},
pages = {1130001},
year = {2012},
note = {\href{https://mathscinet.ams.org/mathscinet/article?mr=2930377}{MR2930377}}
}

@book{lieb2001analysis,
  title={Analysis},
  author={Lieb, E. H. and Loss, M.},
  EDITION = {Second},  
  ISBN = {0-8218-2783-9},
  VOLUME = {14},  
  SERIES = {Graduate Studies in Mathematics},
  year={2001},
  publisher={American Mathematical Society},
  note = {\href{https://mathscinet.ams.org/mathscinet/article?mr=1817225}{MR1817225}}
}

@article{COSCO2022127,
title = {{Law of large numbers and fluctuations in the sub-critical and $L^{2}$ regions for SHE and KPZ equation in dimension $d\geq 3$}},
author = {C. Cosco and S. Nakajima and M. Nakashima},
journal = {Stochastic Processes and their Applications},
volume = {151},
pages = {127-173},
year = {2022},
issn = {0304-4149},
note = {\href{https://mathscinet.ams.org/mathscinet/article?mr=4441505}{MR4441505}}
}

@book{GHS, 
place={Cambridge}, 
series={Cambridge Tracts in Mathematics}, 
title={{Gaussian Hilbert spaces}}, 
note = {\href{https://mathscinet.ams.org/mathscinet/article?mr=1474726}{MR1474726}},
publisher={Cambridge University Press}, 
author={Janson, S.}, 
VOLUME = {129},
year={1997}, 
ISBN = {0-521-56128-0},
collection={Cambridge Tracts in Mathematics},
}

@article{solving,
 ISSN = {0003486X},
 author = {M. Hairer},
 journal = {Annals of Mathematics},
 number = {2},
 pages = {559--664},
 publisher = {Annals of Mathematics},
 title = {{Solving the KPZ equation}},
 volume = {178},
 year = {2013},
 note = {\href{https://mathscinet.ams.org/mathscinet/article?mr=3204494}{MR3204494}}
}

@article{Para, 
    title={{Paracontrolled distributions and singular PDEs}}, 
    volume={3}, 
    note = {\href{https://mathscinet.ams.org/mathscinet/article?mr=3406823}{MR3406823}},
    journal={Forum of Mathematics. Pi}, 
    publisher={Cambridge University Press}, 
    author={M. Gubinelli and P. Imkeller and N. Perkowski},
    year={2015}, 
    pages={e6, 75}
}

@article{DDP,
note = {\href{https://mathscinet.ams.org/mathscinet/article?mr=968950}{MR968950}},
year = {1988},
volume = {52},
number = {3-4},
pages = {609--626},
author = {J. Z. Imbrie and T. Spencer},
title = {Diffusion of directed polymers in a random environment},
journal = {Journal of Statistical Physics}
}

@article{DDP2,
  title = {Pinning and Roughening of Domain Walls in {I}sing Systems Due to Random Impurities},
  author = {Huse, D. A. and Henley, C. L.},
  journal = {Phys. Rev. Lett.},
  volume = {54},
  issue = {25},
  pages = {2708--2711},
  numpages = {0},
  year = {1985},

  publisher = {American Physical Society},
  note = {\href{https://journals.aps.org/prl/abstract/10.1103/PhysRevLett.54.2708}{DOI: 10.1103/PhysRevLett.54.2708}}
}

@article{GFDDP,
author = {C. Cosco and S. Nakajima},
title = {{Gaussian fluctuations for the directed polymer partition function in dimension $d\ge 3$ and in the whole ${L^{2}}$-region}},
volume = {57},
journal = {Annales de l'Institut Henri Poincaré, Probabilités et Statistiques},
number = {2},
publisher = {Institut Henri Poincaré},
pages = {872--889},
year = {2021},
note = {\href{https://mathscinet.ams.org/mathscinet/article?mr=4260488}{MR4260488}}
}

@article{Alberts_2013,
        note = {\href{https://mathscinet.ams.org/mathscinet/article?mr=3162542}{MR3162542}},
	year = 2013,
  
	publisher = {Springer Science and Business Media {LLC}
},
  
	volume = {154},
  
	number = {1-2},
  
	pages = {305--326},
  
	author = {T. Alberts and K. Khanin and J. Quastel},
  
	title = {The Continuum Directed Random Polymer},
        
	journal = {Journal of Statistical Physics}
}

@article{betaclessthenbetaL2_3,
author = {M. Birkner and R. Sun},
title = {{Annealed vs quenched critical points for a random walk pinning model}},
volume = {46},
journal = {Annales de l'Institut Henri Poincaré, Probabilités et Statistiques},
number = {2},
publisher = {Institut Henri Poincaré},
pages = {414--441},
year = {2010},
note = {\href{https://mathscinet.ams.org/mathscinet/article?mr=2667704}{MR2667704}}
}

@article{scaling,
  title={{The scaling limit of the KPZ equation in space dimension 3 and higher}},
  author={Magnen, J. and Unterberger, J.},
  journal={Journal of Statistical Physics},
  volume={171},
  pages={543--598},
  year={2018},
  number = {4},
  publisher={Springer},
  note = {\href{https://mathscinet.ams.org/mathscinet/article?mr=3790153}{MR3790153}}
}

@article{gu2018edwards,
  title={{The Edwards--Wilkinson limit of the random heat equation in dimensions three and higher}},
  author={Gu, Y. and Ryzhik, L. and Zeitouni, O.},
  journal={Communications in Mathematical Physics},
  volume={363},
  number = {2},
  pages={351--388},
  year={2018},
  publisher={Springer},
  note = {\href{https://mathscinet.ams.org/mathscinet/article?mr=3851818}{MR3851818}}
}

@article{guKPZ,
  title={{Fluctuations of the solutions to the KPZ equation in dimensions three and higher}},
  author={Dunlap, A. and Gu, Y. and Ryzhik, L. and Zeitouni, O.},
  journal={Probability Theory and Related Fields},
  volume={176},
  number={3-4},
  pages={1217--1258},
  year={2020},
  publisher={Springer},
  note = {\href{https://mathscinet.ams.org/mathscinet/article?mr=4087492}{MR4087492}}
}

@article{rate_Comets,
  title={Rate of convergence for polymers in a weak disorder},
  author={Comets, F. and Liu, Q.},
  journal={Journal of Mathematical Analysis and Applications},
  volume = {455},  
  number={1},
  pages={312--335},
  year={2017},
  issn = {0022-247X},  
  publisher={Elsevier},
  note = {\href{https://mathscinet.ams.org/mathscinet/article?mr=3665102}{MR3665102}}
}

@article{nonlinearSHE,
  title={Fluctuations of a nonlinear stochastic heat equation in dimensions three and higher},
  author={Gu, Y. and Li, J.},
  journal={SIAM Journal on Mathematical Analysis},
  volume={52},
  number={6},
  pages={5422--5440},
  year={2020},
  publisher={SIAM},
  note = {\href{https://mathscinet.ams.org/mathscinet/article?mr=4169750}{MR4169750}}
}

@article{nonlinearSHE2,
    title={{Fluctuations of stochastic PDEs with long-range correlations}},
    author={Gerolla, L. and Hairer, M. and Li, X.-M.},
    volume = {35},
    journal = {The Annals of Applied Probability},
    number = {2},
    publisher = {Institute of Mathematical Statistics},
    pages = {1198 -- 1232},
    year = {2025},
    note = {\href{https://mathscinet.ams.org/mathscinet/article?mr=4897759}{MR4897759}}
}

@article{junk_new,
  title={New characterization of the weak disorder phase of directed polymers in bounded random environments},
  author={Junk, S.},
  journal={Communications in Mathematical Physics},
  volume={389},
  number={2},
  pages={1087--1097},
  year={2022},
  publisher={Springer},
  note = {\href{https://mathscinet.ams.org/mathscinet/article?mr=4369727}{MR4369727}}
}

@article{moment,
  title={{Subcritical Gaussian multiplicative chaos in the Wiener space: construction, moments and volume decay}},
  author={Bazaes, R. and Lammers, I. and Mukherjee, C.},
  journal={Probability Theory and Related Fields},
  volume={190},
  number={3-4},
  pages={753--801},
  year={2024},
  publisher={Springer},
  note = {\href{https://mathscinet.ams.org/mathscinet/article?mr=4811815}{MR4811815}}
}

@incollection{quastel,
  title={Introduction to {KPZ}},
  author={Quastel, J.},
   BOOKTITLE = {Current developments in mathematics, 2011},
  number={1},
  year={2012},
  PAGES = {125--194},
 PUBLISHER = {Int. Press, Somerville, MA},
  ISBN = {978-1-57146-239-8},
  note = {\href{https://mathscinet.ams.org/mathscinet/article?mr=3098078}{MR3098078}}
}

@book{ZZZ,
    AUTHOR = {Wheeden, R.L. and Zygmund, A.},
     TITLE = {Measure and integral},
    SERIES = {Pure and Applied Mathematics (Boca Raton)},
   EDITION = {Second},
 PUBLISHER = {CRC Press, Boca Raton, FL},
      YEAR = {2015},
      ISBN = {978-1-4987-0289-8},
      note = {\href{https://mathscinet.ams.org/mathscinet/article?mr=3381284}{MR3381284}}
}

@article{polymer_3d,
  title={{Edwards--Wilkinson fluctuations for the directed polymer in the full $L^{2}$-regime for dimensions $d\geq 3$}},
  author={Lygkonis, D. and Zygouras, N.},
  journal={Annales de l'Institut Henri Poincare (B) Probabilites et statistiques},
  volume={58},
  number={1},
  pages={65--104},
  year={2022},
  note = {\href{https://mathscinet.ams.org/mathscinet/article?mr=4374673}{MR4374673}}
}

@article{junk_new_new,
      title={{Fluctuations of partition functions of directed polymers in weak disorder beyond the $L^2$-phase}}, 
      author={S. Junk},
volume = {53},
journal = {The Annals of Probability},
number = {2},
publisher = {Institute of Mathematical Statistics},
pages = {557 -- 596},
year = {2025},
note = {\href{https://mathscinet.ams.org/mathscinet/article?mr=4888140}{MR4888140}}
}

@incollection {comets2004probabilistic,
    AUTHOR = {Comets, F. and Shiga, T. and Yoshida, N.},
     TITLE = {Probabilistic analysis of directed polymers in a random environment: a review},
 BOOKTITLE = {Stochastic analysis on large scale interacting systems},
    SERIES = {Adv. Stud. Pure Math.},
    VOLUME = {39},
     PAGES = {115--142},
 PUBLISHER = {Math. Soc. Japan, Tokyo},
      YEAR = {2004},
      note = {\href{https://mathscinet.ams.org/mathscinet/article?mr=2073332}{MR2073332}}
}

@article{ZYGOURAS2024104431,
title = {Directed polymers in a random environment: A review of the phase transitions},
journal = {Stochastic Processes and their Applications},
volume = {177},
pages = {104431},
year = {2024},
issn = {0304-4149},
author = {N. Zygouras},
note = {\href{https://mathscinet.ams.org/mathscinet/article?mr=4793446}{MR4793446}}
}

@incollection {walsh1986introduction,
    AUTHOR = {Walsh, J.B.},
     TITLE = {An introduction to stochastic partial differential equations},
 BOOKTITLE = {\'Ecole d'\'et\'e{} de probabilit\'es de {S}aMeasure and Integralint-{F}lour,
              {XIV}---1984},
    SERIES = {Lecture Notes in Math.},
    VOLUME = {1180},
     PAGES = {265--439},
 PUBLISHER = {Springer, Berlin},
      YEAR = {1986},
      ISBN = {3-540-16441-3},
   note = {\href{https://mathscinet.ams.org/mathscinet/article?mr=876085}{MR876085}}
}

@book{LT,
  title={Limit theorems for stochastic processes},
  author={Jacod, J. and Shiryaev, A.},
  volume={288},
  ISBN = {3-540-17882-1},
  year={2013},
  publisher={Springer Science \& Business Media},
  note = {\href{https://mathscinet.ams.org/mathscinet/article?mr=959133}{MR959133}}
}

@article{ferrari2010random,
  title={Random growth models},
  author={P L. Ferrari and H Spohn},
  year={2010},
  note = {\href{https://arxiv.org/abs/1003.0881}{arXiv:1003.0881}}
}

@article{quastel2015one,
  title={The one-dimensional {KPZ} equation and its universality class},
  author={Quastel, J and Spohn, H},
  journal={Journal of Statistical Physics},
  volume={160},
  pages={965--984},
  number = {4},
  year={2015},
  publisher={Springer},
  note = {\href{https://mathscinet.ams.org/mathscinet/article?mr=3373647}{MR3373647}}
}

@article{continuousdirectedpolymer,
  title={{Localization Transition for Polymers in Poissonian Medium}},
  author={Comets, F. and Yoshida, N.},
  journal={Communications in Mathematical Physics},
  volume={323},
  number = {1},
  pages={417--447},
  year={2013},
  note = {\href{https://mathscinet.ams.org/mathscinet/article?mr=3085670}{MR3085670}}
}

\end{document}